\documentclass[reqno,11pt]{amsart}

\usepackage{amscd}
\usepackage{amsmath}
\usepackage{amssymb}
\usepackage[american]{babel}
\usepackage{bbm}
\usepackage{bookmark}
\usepackage{cmap} 
\usepackage{dsfont}
\usepackage{enumerate}
\usepackage{epigraph}
\usepackage[mathscr]{euscript}
\usepackage[myheadings]{fullpage}
\usepackage{graphicx}
\usepackage[geometry]{ifsym}
\usepackage{mathabx}
\usepackage{accents}
\usepackage{mathrsfs}
\usepackage{mathtools}
\usepackage{refcount}
\usepackage{setspace}
\usepackage{stmaryrd}
\usepackage{thinsp}
\usepackage{verbatim}
\usepackage[all,2cell]{xy} \UseTwocells
\usepackage{xr}
\usepackage[multiple]{footmisc}
\usepackage{hyperref}

\usepackage{tikz}
\usetikzlibrary{decorations.pathmorphing}
\usetikzlibrary{cd}

\numberwithin{equation}{subsection}

\newtheorem{thm}{Theorem}[subsubsection]
\newtheorem*{thm*}{Theorem}

\newtheorem{cor}[thm]{Corollary}
\newtheorem*{cor*}{Corollary}
\newtheorem{lem}[thm]{Lemma}

\newtheorem{prop}[thm]{Proposition}
\newtheorem{prop-const}[thm]{Proposition-Construction}

\newtheorem*{conjecture*}{Conjecture}

\newtheorem*{princ*}{Principle}

\theoremstyle{remark}
\newtheorem{rem}[thm]{Remark}
\newtheorem{example}[thm]{Example}

\newtheorem{defin}[thm]{Definition}

\newtheorem{notation}[thm]{Notation}

\newtheorem{variant}[thm]{Variant}
\newtheorem{question}[thm]{Question}

\newcounter{steps}[thm]
\newtheorem{step}[steps]{Step}

\newcommand{\presup}[1]{\prescript{#1}{}}

\newcommand{\xar}[1]{\xrightarrow{#1}}
\newcommand{\rar}[1]{\xar{#1}}
\newcommand{\isom}{\rar{\simeq}}

\newcommand{\into}{\hookrightarrow}
\newcommand{\onto}{\twoheadrightarrow}

\newcommand{\bA}{{\mathbb A}}
\newcommand{\bB}{{\mathbb B}}

\newcommand{\bG}{{\mathbb G}}

\newcommand{\bR}{{\mathbb R}}

\newcommand{\bZ}{{\mathbb Z}}

\newcommand{\sA}{{\EuScript A}}
\newcommand{\sB}{{\EuScript B}}
\newcommand{\sC}{{\EuScript C}}
\newcommand{\sD}{{\EuScript D}}
\newcommand{\sE}{{\EuScript E}}
\newcommand{\sF}{{\EuScript F}}
\newcommand{\sG}{{\EuScript G}}
\newcommand{\sH}{{\EuScript H}}

\newcommand{\sK}{{\EuScript K}}
\newcommand{\sL}{{\EuScript L}}
\newcommand{\sM}{{\EuScript M}}
\newcommand{\sN}{{\EuScript N}}
\newcommand{\sO}{{\EuScript O}}
\newcommand{\sP}{{\EuScript P}}

\newcommand{\sS}{{\EuScript S}}

\newcommand{\sU}{{\EuScript U}}

\newcommand{\sX}{{\EuScript X}}
\newcommand{\sY}{{\EuScript Y}}
\newcommand{\sZ}{{\EuScript Z}}

\newcommand{\fL}{{\mathfrak L}}

\newcommand{\fX}{{\mathfrak X}}

\newcommand{\fg}{{\mathfrak g}}

\newcommand{\on}{\operatorname}
\newcommand{\ol}[1]{\overline{#1}{}}
\newcommand{\ul}{\underline}

\newcommand{\mathendash}{\text{\textendash}}

\newcommand{\ldotsplus}{\mathinner{\ldotp\ldotp\ldotp\ldotp}}

\newcommand{\Ker}{\on{Ker}}
\newcommand{\Coker}{\on{Coker}}
\newcommand{\End}{\on{End}}
\newcommand{\Hom}{\on{Hom}}
\newcommand{\Ext}{\on{Ext}}

\newcommand{\Spec}{\on{Spec}}
\newcommand{\Spf}{\on{Spf}}

\newcommand{\id}{\on{id}}

\newcommand{\ev}{\on{ev}}

\newcommand{\ind}{\on{ind}}

\newcommand{\cotr}{\on{cotr}}
\newcommand{\Rep}{\mathsf{Rep}}
\newcommand{\gr}{\on{gr}}

\newcommand{\act}{\on{act}}
\newcommand{\actson}{\curvearrowright}

\renewcommand{\dot}{\bullet}

\newcommand{\Tor}{\on{Tor}}

\newcommand{\vph}{\varphi}
\newcommand{\vareps}{\varepsilon}

\newcommand{\Vect}{\mathsf{Vect}}

\newcommand{\Res}{\on{Res}}
\newcommand{\Gr}{\on{Gr}}

\newcommand{\Sym}{\on{Sym}}
\newcommand{\Bun}{\on{Bun}}

\newcommand{\LocSys}{\on{LocSys}}

\renewcommand{\mod}{\mathendash\mathsf{mod}}

\newcommand{\colim}{\on{colim}}

\newcommand{\DGCat}{\mathsf{DGCat}}
\newcommand{\ShvCat}{\mathsf{ShvCat}}

\newcommand{\Gpd}{\mathsf{Gpd}}

\renewcommand{\lim}{\on{lim}}
\newcommand{\Ind}{\mathsf{Ind}}

\newcommand{\TwoHom}{\mathsf{Hom}}
\newcommand{\TwoEnd}{\mathsf{End}}

\newcommand{\heart}{\heartsuit}

\newcommand{\Oblv}{\on{Oblv}}
\newcommand{\Av}{\on{Av}}
\newcommand{\Nm}{\on{Nm}}

\newcommand{\Perf}{\mathsf{Perf}}
\newcommand{\Coh}{\mathsf{Coh}}
\newcommand{\IndCoh}{\mathsf{IndCoh}}
\newcommand{\QCoh}{\mathsf{QCoh}}
\newcommand{\Sch}{\mathsf{Sch}}
\newcommand{\AffSch}{\mathsf{AffSch}}
\newcommand{\PreStk}{\mathsf{PreStk}}
\newcommand{\IndSch}{\mathsf{IndSch}}

\newcommand{\Alg}{\mathsf{Alg}}

\newcommand{\ComAlg}{\mathsf{ComAlg}}

\newcommand{\Lie}{\on{Lie}}

\newcommand{\AJ}{\on{AJ}}
\newcommand{\Pol}{\on{Pol}}
\newcommand{\Maps}{\sM aps}

\renewcommand{\o}[1]{\mathring{#1}}

\renewcommand{\subset}{\subseteq}

\newcommand{\YM}{\mathsf{YM}}
\newcommand{\Dir}{\sD ir}
\newcommand{\Neu}{\sN eu}

\newcommand{\pair}[2]{\langle #1,#2 \rangle}

\makeatletter
\newcommand{\biggg}{\bBigg@{4}}
\newcommand{\Biggg}{\bBigg@{5}}
\makeatother

\date{\today}

\begin{document}

\frenchspacing

\setlength{\epigraphwidth}{0.4\textwidth}
\renewcommand{\epigraphsize}{\footnotesize}

\begin{abstract}

We calculate the category of $D$-modules on the loop space of the affine line 
in coherent terms. Specifically, we find that this category is derived equivalent
to the category of 
ind-coherent sheaves on the moduli space of rank one de Rham local systems
with a flat section. Our result establishes a conjecture coming out of the 
$3d$ mirror symmetry program, which obtains new compatibilities
for the geometric Langlands program from rich dualities of QFTs that 
are themselves obtained from string theory conjectures.   

\end{abstract}

\title{Tate's thesis in the de Rham setting}

\author{Justin Hilburn}

\address{Perimeter Institute,
31 Caroline St N, Waterloo, ON N2L 2Y5, Canada}

\email{jhilburn@perimeterinstitute.ca}

\author{Sam Raskin}

\address{The University of Texas at Austin, 
Department of Mathematics, 
RLM 8.100, 2515 Speedway Stop C1200, 
Austin, TX 78712}

\email{sraskin@math.utexas.edu}

\maketitle

\setcounter{tocdepth}{1}
\tableofcontents

\section{Introduction}\label{s:intro}

\subsection{Statement of the main results}

We work over a field $k$ of characteristic zero. 

Let $\sY$ be the moduli space of rank $1$ de Rham local systems on 
the punctured disc equipped with a flat section, and 
let $\fL \bA^1$ denote the algebraic loop space of $\bA^1$.

Our main result asserts:

\begin{thm}[Thm. \ref{t:main-precise}]\label{t:main}

There is a canonical equivalence of DG categories:
\[
\Delta:D^!(\fL \bA^1) \simeq \IndCoh^*(\sY).
\]

Moreover, this equivalence is compatible with the local geometric 
class field theory of Beilinson-Drinfeld. 

\end{thm}

Here the left hand side is the category of $D$-modules on $\fL \bA^1$,
as defined in \cite{dario-*/!}, \cite{dmod}. The right hand side
is the category of ind-coherent sheaves on $\sY$, which we construct
in \S \ref{s:indcoh} (following \cite{methods}). In both cases,
there are infinite-dimensional subtleties in the definitions; 
these are far more severe on the right hand side. 

\begin{rem}

The assertions of the theorem are made more precise in the body of the paper. 
For now, we content ourselves with the following remark on local class field theory.

Beilinson-Drinfeld construct an equivalence:
\begin{equation}\label{eq:intro-lcft}
D^*(\fL\bG_m) \simeq \QCoh(\LocSys_{\bG_m})
\end{equation}

\noindent for $\LocSys_{\bG_m}$ the moduli space of rank $1$ de Rham local systems on the
punctured disc. The left hand side here naturally acts on 
the left hand side of Theorem \ref{t:main}, while the right hand side
naturally acts on the right hand side of Theorem \ref{t:main}. The compatibility
asserts that the equivalence of Theorem \ref{t:main} is compatible with 
these actions.

\end{rem}

\begin{rem}

Both sides of the above equivalence have natural $t$-structures.
However, this equivalence is \emph{not} $t$-exact; it is necessarily
an equivalence of derived categories. This is in contrast to \eqref{eq:intro-lcft},
which largely amounts to an equivalence of abelian categories. 

\end{rem}

\begin{rem}

For $K$ a local field, Tate's thesis \cite{tate-thesis} 
(see also \cite{tate-thesis-weil}) 
considers the decomposition of the space
$\sD(K)$ of tempered distributions on $K$ 
as a representation for $K^{\times}$. We consider Theorem \ref{t:main} 
as an analogue of the arithmetic situation. We observe that
our result applies not only at the level of eigenspaces, 
but describes a categorical analogue of the whole space $\sD(K)$.

We plan to return to global aspects of the subject in future work. 

%% reference?

\end{rem}

\begin{rem}

As far as we are aware, our work is the first one in geometric representation
theory to prove a theorem about coherent sheaves on a 
space of the form $\Maps(\o{\sD}_{dR},Y)$ for $Y$ an Artin stack
that is not a scheme and is not a classifying stack (for us: $Y = \bA^1/\bG_m$). 
As we will see in \S \ref{s:spectral} and \ref{s:indcoh}, there are substantial
technical difficulties that arise in this setting.
Roughly speaking, the singularities of the space $\sY$ are genuinely of infinite type. 
We can only overcome these difficulties in our specific (abelian) setting.

\end{rem}

\begin{rem}

As discussed later in \S \ref{ss:intro-strategy}, the main piece of our
construction is a realization of $D$-modules on $\bA^n$ in spectral terms;
this is the content of \S \ref{s:weyl}.
Already in the $n = 1$ case, our realization of $D$-modules on 
$\bA^1$ as coherent sheaves on some space is novel; in this case, 
we highlight that the expression
is as coherent sheaves on the subspace $\sY^{\leq 1} \subset \sY$ 
discussed in detail in  \S \ref{ss:rs-sketch}.

\end{rem}

\begin{rem}

In the physicist's language, we are 
comparing line operators for the $A$-twist of a pure hypermultiplet
with the $B$-twist of a $U(1)$-gauged hypermultiplet. Moreover, we
consider these $3d$ $\sN = 4$ theories as boundary theories for 
(suitably twisted) pure $U(1)$-gauge theory (with electro-magnetic duality
exchanging its $A$ and $B$-twists).

\end{rem}

\subsection{Connections to $3d$ mirror symmetry}

The equivalence of Theorem \ref{t:main} is a first\footnote{Specifically, 
as an equivalence of categories of line operators, with both matter and
a non-trivial gauge group on the $B$-side.}
instance of what is expected to be a broad family of equivalences, 
which goes under the heading
\emph{$3d$ mirror symmetry}. 
This subject is closely tied to the geometric Langlands program.

As context for our results, we
provide an informal introduction to these ideas below.  
Our objective is to connect our equivalence as closely as possible with the 
mathematical physics literature, and to promote those ideas to mathematicians
interested in the circle of ideas around geometric Langlands.

We hope the reader will forgive the 
redundancy of our discussion given 
the numerous other great sources in the literature.
% other refs?

We emphasize that we do not claim originality for any of the ideas
appearing below. We withhold attributions and leave much contact
with the existing literature until \S \ref{ss:3d-refs}.

\begin{rem}

For another exposition of this subject also targeted at mathematicians,
see \cite{bf-notes}. For a recent physics-oriented exposition, we refer
to \cite{dggh} \S 2.

\end{rem}

\subsubsection{}

The physics of the last thirty years has highlighted the role of \emph{dualities}:
quantum field theories that are equivalent by non-obvious means.
We emphasize that
physicists regard these dualities as conjectural: they do not claim
to know how to match the QFTs, only (at best) parts of them.

Examples abound. For instance,
Montonen--Olive's (conjectural) $S$-duality for 
$4$-dimensional Yang-Mills theory with $\sN = 4$ supersymmetry for two 
Langlands dual gauge groups $G$ and $\check{G}$ is of keen interest. 

The dualities physicists consider fit into a sophisticated logical hierarchy.
But briefly, most are subsumed in the existence of Witten's (conjectural) $\sM$-theory,
which is supposed to recover other (known) QFTs by various constructions.

\subsubsection{}

A general problem is to extract mathematical structures
from quantum field theories.\footnote{Often, this is by a sort of analogy. Physicists
are drawn to differential geometry: their bread and butter are moduli spaces
of solutions to non-linear PDEs (namely, Euler-Lagrange equations).

But for a mathematician,
it may be preferable to consider algebraic varieties as analogous to K\"ahler
manifolds, or symplectic varieties as analogous to hyper-K\"ahler manifolds. 
Often, rich algebraic geometry results.} 
In this case, physical dualities relate to mathematical
conjectures.

For instance, in \cite{kapustin-witten}, 
Kapustin and Witten found such a relationship 
between Montonen--Olive's (conjectural) duality for 
$4$-dimensional Yang-Mills with $\sN = 4$ supersymmetry and geometric Langlands
conjectures in mathematics. Their perspective was developed in 
\cite{costello-ss-2-4} and 
\cite{elliott-yoo}, which exactly clarified some means of extracting
algebro-geometric invariants from Lagrangian field theories.

\subsubsection{$3d$ $\sigma$-models}

In one setting, so-called\footnote{We remind what these parameters
encode in the subsequent discussion.}
$3d$ $\sN = 4$ quantum field theories, there are 
rich relations with algebraic geometry.

First, for every algebraic stack $\sX$, we suppose
there is (in some algebraic sense) 
a $3d$ (Lagrangian) quantum field theory
$T_{\sX}$ with $\sN = 4$ supersymmetry; 
we spell out some more precise expectations below. 

\begin{rem}\label{r:no-qft}

At the quantum level, physicists agree our supposition is problematic: 
cf. the discussion at the
end of \cite{rozansky-witten} \S 2.3. (The 
classical field theory is fine, but may be unrenormalizable, and even if it is
renormalizable, there may not be a distinguished scheme.)
So it is better 
to regard $T_{\sX}$ as a heuristic:
its twists (see below) are all that is defined
(at the quantum level). 

\end{rem}

\begin{rem}\label{r:3d-t*}

In fact, physicists
say that $T_{\sX}$ depends only on the symplectic stack 
$T^* \sX$. At the classical level, this theory
is the $3d$ $\sigma$-model with target $T^*\sX$. 
The $\sN = 4$ supersymmetry comes from the 
symplectic structure on $T^*\sX$. 

I.e., more generally, there is a \emph{classical} $3d$ $\sN = 4$ theory
corresponding to a $\sigma$-model with target any symplectic stack $\sS$. 
One hopes for a quantization associated with a Lagrangian
foliation $\lambda:\sS \to \sX$; 
at the moment, this quantization (or rather, its $A/B$ twists) 
is only understood for $\sS$ a twisted cotangent bundle over some $\sX$
(and $\lambda$ the canonical projection).

This perspective is important, but sometimes 
inconvenient (especially at the quantum level), 
so we emphasize the role of $\sX$ e.g. in our notation.

\end{rem}

% reference for A&B twists specifically? A little muddled in [ES]....

\subsubsection{Algebraic QFTs}

Next, we describe how we think about a $3d$ QFT $T$ in algebraic geometry.
We include this discussion to make precise the connection between the mathematical
physics we wish to discuss and the algebraic geometry we wish to study.

\subsubsection{}

Recall from \cite{lurie-tft} that an \emph{oriented $3$-dimensional fully-extended TFT} 
$T$ would 
attach a number $T(M)$ to a closed 3-manifold $M$, a vector space
$T(M)$ to a closed 2-manifold $M$, a DG category $T(M)$ to a closed
1-manifold $M$, and a DG 2-category (i.e., $\DGCat_{cont}$-module
category) to a 0-manifold $M$. Cobordisms define morphisms. Disjoint unions
go to tensor products.

A variant: given an ambient $3$-manifold $N$, a fully-extended TFT \emph{on $N$} should
assign such data to manifolds $M$ equipped with embeddings into $N$.
(We are not aware of a reference.)

\subsubsection{}\label{ss:alg-3d}

The algebraic situation is similar, but only some data is defined.
We heuristically describe the main structures, without giving formal definitions.

Fix a smooth, projective, geometrically connected algebraic curve $X$,
which we regard as analogous to a $2$-manifold. (The curve $X$ should not be confused
with the $\sX$ that sometimes occurs when considering a $\sigma$-model with 
target $T^*\sX$.)

A\footnote{As the input to the discussion is an algebraic curve $X$,
we sometimes refer to this formalism as \emph{algebraic QFT}.
We emphasize that we are \emph{not} referring to some kind of
``algebraic twist" (analogous to a ``holomorphic twist") 
of a supersymmetric Lagrangian field theory \textemdash{}  the
formalism is insensitive to such considerations. Rather, this
terminology refers to the fact that this formalism is defined without
analysis, and makes sense in the general setting of algebraic geometry (of
curves here, with evident adaptations in the terminology otherwise)
over fields (perhaps of characteristic zero).}
\emph{$3d$ field theory} $T$ on\footnote{It might be better
to say the theory lives on $X \times \bR$.} $X$ includes the
data of a vector space $T(X) \in \Vect$ and a category 
$T(\o{\sD}_x) \in \DGCat_{cont}$ for every point $x \in X$,
regarding $\o{\sD}_x$ as analogous to a circle. 

We regard the formal disc $\sD_x$ as a cobordism
$\emptyset \to \o{\sD}_x$. There is an associated functor:
\[
\Vect = T(\emptyset) \to T(\o{\sD}_x) \in \DGCat_{cont} 
\]

\noindent i.e., there is a preferred object of $T(\o{\sD}_x)$,
the so-called \emph{vacuum} (or \emph{unit}) object. 

We regard $X \setminus x$ as a cobordism
$\o{\sD}_x \to \emptyset$. There is an associated functor:
\[
T(\o{\sD}_x) \to T(\emptyset) = \Vect \in \DGCat_{cont}.
\]

\noindent This functional evaluated on the vacuum is
$T(X) \in \Vect$.

\begin{rem}

Unlike in the topological situation, the above is not symmetric.
I.e., we only allow cobordisms in the directions specified above.

\end{rem}

\begin{rem}

In this $3d$ setting, the category $T(\o{\sD}_x)$ is often
called the \emph{category of line operators} for the theory. (Physicists would 
draw the line $x \times \bR \subset X \times \bR$ passing through
the interior of our circle $\o{\sD}_x \times 0 \subset X \times \bR$.)

\end{rem}

\begin{rem}\label{r:fact-cats}

The assignment $x \rightsquigarrow T(\o{\sD}_x)$ should extend
to a \emph{factorization category on $X$} in the sense of
\cite{chiralcats}. The vacuum objects should correspond to a unital
structure on this factorization category. The functionals
above should extend to a well-defined functional on the
\emph{independent category} of this unital factorization category;
see \cite{indep-cat} for the definition.

\end{rem}

\begin{example}\label{e:3d-chiral}

In \cite{chiral}, Beilinson and Drinfeld defined such a structure
for any chiral algebra $\sA$ on $X$, i.e., they (in effect)
defined a $3d$ field theory $T_{\sA}$. 
The vector space $T_{\sA}(X)$ is the
space of \emph{conformal blocks} of $\sA$. The category
$T_{\sA}(\o{\sD}_x)$ is $\sA\mod_x$, the category of (unital) chiral
modules for $\sA$ supported at $x$. The vacuum object is the
vacuum representation of $\sA$, and the functional
$\sA\mod_x \to \Vect$ is the functor $C^{ch}(X,\sA,-)$ from
\emph{loc. cit}. \S 4.2.19.

(This theory is analogous to the $3d$ TFT associated
with a $\sE_2$-algebra $\sA$ in 
\cite{lurie-tft} Theorem 4.1.24.\footnote{We 
note that \emph{loc. cit}. regards this theory as a fully-extended
$2d$ theory with more categorical complexity than we outlined
before. We instead regard the theory as a not-fully-extended
$3d$ theory of the type outlined above. 
This compares to Remark 4.1.27 from \cite{lurie-tft}.})

\end{example}

\begin{rem}

Informally, it is good to think of $3d$ theories on $X$ as 
the natural home for the Morita theory of chiral algebras,
just as $\DGCat_{cont}$ is the natural home for Morita theory of
associative algebras.

\end{rem}

\begin{rem}\label{r:voa-3d}

The relationship between chiral algebras and $3d$ $\sN = 4$ (and
sometimes $\sN = 2$) theories is the starting point of the series of
works \cite{voa-3d}, \cite{ccg}, and \cite{cdg}, which aim to use
chiral algebras to 
study mirror dual theories in this way. To describe this work in more detail,
we use ideas reviewed later in this introduction. 

Specifically, a $3d$ $\sN = 4$
theory $T$ with a suitably deformable 
(cf. \cite{voa-3d} \S 2.3; there is an implicit
choice of $A$ or $B$ twists fixed in our discussion here) 
$\sN = (0,4)$ boundary condition $\sB$ should yield an 
algebraic $3d$ field theory $T^{alg}$ 
and a boundary condition $\sB^{alg}$ for it; here
we regard $\sB^{alg}$ as an interface $T^{alg} \to \on{triv}$ 
\emph{to} the trivial theory. 
This data yields a factorization algebra $\sA_{T,\sB}$ on $X$: its fiber at $x$
is by evaluating $\sB^{alg}(\o{\sD}_x):T(\o{\sD}_x) \to \on{triv}(\o{\sD}_x) = \Vect$
at the vacuum object for $T$. In this case, there is a canonical
interface $T^{alg} \to T_{\sA_{T,\sB}}$. Sometimes, it is even an equivalence,
meaning that one can study the full theory $T^{alg}$ via the factorization
algebra $\sA_{T,\sB}$. 

\end{rem}

\begin{rem}

As David Ben-Zvi emphasized to us, unital factorization categories
with a functional on their independent categories
do seem to provide a robust
definition of $[1,2]$-extended $3d$ algebraic QFT (while allowing for
some quite degenerate theories). He also informs us that
these ideas were developed in collaboration with Sakellaridis and Venkatesh,
and will be developed in their forthcoming work. 

\end{rem}

\begin{question}

Is there some sense in which the theories of Example \ref{e:3d-chiral}
are $[0,2]$-extended, as for their $E_2$-counterparts? 
And what general constructions of morphisms
(alias: \emph{interfaces}) between such theories exist?\footnote{This question
is of practical importance to us. Our
main theorem amounts to a construction of a non-trivial interface 
between two $3d$ theories. It is desirable to understand this construction on 
better conceptual grounds.}

\end{question}

\subsubsection{Supersymmetry}

We now recall the main practical 
application of supersymmetry: supersymmetric
QFTs (typically) come with canonical 
deformations,\footnote{Sometimes, we think of
deformations \emph{as} quantizations. For us, the difference
is as follows.

Briefly, in $0$-dimensions, $\sP_0$-algebras can be quantized
by $\sE_0$-algebras, 
i.e., pointed vector spaces. On the other hand, pointed vector
spaces $V$ may be further \emph{deformed} by 
finding a filtered vector space $V^{def}$ with $\gr_{\dot} V^{def} = V$.
(In homological settings, such a structure is equivalent to equipping $V$ with 
a differential in a suitable sense, or one might say, deforming
the differential on $V$.)

In higher dimensions, one can say essentially the same words, following
\cite{costello-gwilliam-i}. Lagrangian
densities give rise to factorization $\sP_0$-algebras, 
which govern the corresponding
classical field theory. These may be quantized by factorization
$\sE_0$-algebras, i.e., factorization algebras; this amounts to 
quantizing the classical field theory (or more specifically, constructing
the OPE for local operators). These factorization algebras
may be further deformed to other factorization algebras.}
called \emph{twists}.\footnote{For instance,
for a manifold $X$, its de Rham complex deforms
its complex of global differential forms (considered
as equipped with zero differential). Similarly,
for $X$ a smooth
algebraic variety, its de Rham complex 
deforms its Hodge cohomology (or rather, the cochain
analogue). These examples appear in supersymmetric
quantum mechanics.} Twisting preserves dimensions
but lowers the amount of supersymmetry; in our examples,
there is no residual supersymmetry.

\begin{rem}

See e.g. \cite{costello-ss-2-4} and \cite{costello-scheimbauer} 
for an introduction to the twisting procedure. See 
\cite{elliott-safronov} and \cite{esw}
for a detailed classification of the twists of supersymmetric QFTs. For the specific
twists considered in $3d$ mirror symmetry, see \cite{dggh} \S 2.2.

\end{rem}

\subsubsection{}

In the setting of $3d$ QFTs $T$ with $\sN = 4$ supersymmetry, 
there are two twists of interest for us: the $A$-twist $T_A$ and
the $B$-twist $T_B$. In practice, these are \emph{algebraic} theories.

For our theories $T_{\sX}$ of interest,
we let $T_{\sX,A}$ and $T_{\sX,B}$ denote the corresponding
$3d$ QFTs. 

\begin{example}

The category $T_{\sX,A}(\o{\sD}_x)$ of line operators for the $A$-twisted theory
should be the category\footnote{At this point,
we are operating heuristically and do not want to be overly prescriptive, so
we do not consider finer points such as $D^!$ vs. $D^*$. Similarly in 
Example \ref{e:b}.

With that said, in our example, $D^!$ is what appears.}
$D(\fL\sX)$ of $D$-modules on the algebraic loop space: 
\[
\fL \fX = \Maps(\o{\sD}_x,\sX)
\]

\noindent of $\sX$.

\end{example}

\begin{example}\label{e:b}

The category $T_{\sX,B}(\o{\sD}_x)$ of line operators for the $B$-twisted
theory should be:
\[
\IndCoh\big(\Maps(\o{\sD}_{x,dR},\sX)\big).
\]

\end{example}

\begin{rem}

If $\sX$ is smooth affine, each of the categories above essentially
come from chiral algebras, as in Example \ref{e:3d-chiral}.
Indeed, up to mild corrections, 
the $A$-twist is essentially governed by a CDO for $\sX$, while the
$B$-twist is governed by the commutative chiral algebra of functions 
on\footnote{As in \cite{chiral}, it is convenient to think of 
$\sX$ as $\Maps(\sD_{dR},\sX)$ here.} $\sX$.

\end{rem}

\begin{rem}

The algebraic QFTs appearing above are of a special nature. First, they
are defined functorially on every smooth curve $X$. Moreover, for 
$X = \bA^1$, the resulting sheaves of categories are strongly
$\bG_a$-equivariant (cf. \cite{butson-thesis-i} Chapter 2).
These observations are a sort of de Rham analogue of the
fact that the $A$ and $B$-twists are \emph{topological} twists. 

\end{rem}

\subsubsection{} Below, we speak of both algebraic and non-algebraic (or \emph{analytic})
theories. Typically, we speak of supersymmetry for the non-algebraic theories, 
and obtain algebraic theories by twisting. 

\subsubsection{$3d$ mirror symmetry (first pass)}
%\hspace{-.15cm}\footnote{Specifically, the gauge theory in 
%question is pure gauge theory in $4$-dimensions, cf. \S \ref{sss:s} and the subsequent 
%discussion. (If $\sX$ is a stack but not a scheme, physicists might say 
%$T_{\sX}$ is a gauge theory; certainly they would when
%$\sX$ is equipped with a representable map to $\bB G$. Our emphasis here is not the 
%distinction between stacks and schemes, but rather 
%the additional symmetry of being a boundary
%condition for a $4d$ (pure) gauge theory.)} 

Given a $3d$ $\sN = 4$ theory $T$, there is another $3d$ $\sN = 4$ 
theory $T^{\star}$; this theory has the same underlying QFT as $T$, but the
\emph{embedding} of the supersymmetry algebra is modified by an automorphism
of that algebra. The salient property for us is that the $A$-twist 
$T_A^{\star}$ of $T^{\star}$ is the $B$-twist $T_B$ of the original theory.
Moreover, $(T^{\star})^{\star} = T$. We refer to $T^{\star}$ as the
\emph{abstract mirror dual} to the theory $T$.

\subsubsection{}

We now describe a simplified version of $3d$ mirror symmetry conjectures.

These conjectures state that for certain \emph{3d mirror dual pairs} 
$(\sX,\sX^{\star})$ of algebraic stacks, there is
an equivalence:
\[
T_{\sX}^{\star} \simeq T_{\sX^{\star}}
\]

\noindent of $3d$ field theories.

For instance, passing to $B$-twists and categories of line operators on both sides, 
such a conjecture predicts an equivalence:
\[
D(\fL\sX) \simeq \IndCoh\big(\Maps(\o{\sD}_{x,dR},\sX^{\star})\big).
\]

\begin{example}\label{e:mirror-tate-ungauged}

The pair $(\sX,\sX^{\star}) = (\bA^1,\bA^1/\bG_m)$ is supposed to be a $3d$ mirror
dual pair of stacks. In this case, Theorem \ref{t:main} amounts to 
an equivalence of line operators for the corresponding twisted\footnote{This
may seem incomplete, given that $3d$ mirror symmetry is stated more symmetrically 
in terms of untwisted theories. However, as in Remark \ref{r:no-qft}, 
the untwisted QFTs are on unsteady ground.}  
theories.

\end{example}

\subsubsection{$S$-duality}\label{sss:s}

The theory of $3d$ mirror symmetry as presented above admits a generalization
in which there is an auxiliary pair $(G,\check{G})$ of Langlands dual
reductive groups; the previous setting amounts to a pair of trivial groups.

The theory becomes much richer in this setting. There are connections
to $S$-duality in $4d$ gauge theory and the geometric Langlands correspondence.
Moreover, many of the examples in conventional (i.e., trivial group) $3d$ mirror
symmetry can be better understood as being built from more primitive 
examples in the general setting. 

The major cost is that even the formulation of the conjectures 
becomes conditional on forms of $S$-duality conjectures.  

We discuss these ideas in more detail below. 

\subsubsection{}

First, we need to discuss $4d$ algebraic QFTs, which we denote $\mathsf{T}$.

This is easy: the formalism is the same as in \S \ref{ss:alg-3d},
but one categorical level higher. For instance, there should 
be a DG category attached to $X$, and a DG 2-category $\mathsf{T}(X)$
(of \emph{surface defects}) attached to
$\o{\sD}_x$.

We can speak of \emph{interfaces} $\mathsf{T}_1 \to \mathsf{T}_2$ between 
theories, which are morphisms in the natural sense. An interface 
$\mathsf{triv} \to \mathsf{T}$ from
the trivial\footnote{This theory attaches $\Vect$ to $X$ and 
$\DGCat_{cont}$ to $\o{\sD}_x$.}
theory is a \emph{boundary condition} for $\mathsf{T}$.
This amounts to objects:
\[
\begin{gathered}
\sB_X \in \mathsf{T}(X) (\in \DGCat_{cont})  \\
\sB_x \in \mathsf{T}(\o{\sD}_x) (\in \mathsf{2DGCat}).
\end{gathered}
\] 

\begin{rem}\label{r:3d-as-bdry}

A boundary condition $\mathsf{triv} \to \mathsf{triv}$ is the same
as a $3d$ theory. 

\end{rem}

\begin{rem}

Note that theories can be tensored. The $4d$ theories we consider are naturally
self-dual. So we can (and do) consider interfaces and boundary conditions as
operating in both directions. In such a situation, given boundary conditions 
$\sB_1,\sB_2$ for $\mathsf{T}$, we let $\pair{\sB_1}{\sB_2}$ denote 
the resulting boundary condition $\mathsf{triv} \to \mathsf{triv}$
obtained by composing $\mathsf{triv} \xar{\sB_1} \mathsf{T} \xar{\sB_2} \mathsf{triv}$
and applying Remark \ref{r:3d-as-bdry}.

\end{rem}

\subsubsection{} 

We now discuss twists.

For $4d$ $\sN = 4$ (non-algebraic) theory $\mathsf{T}$, there are supposed
to be $A$ and $B$ twists $\mathsf{T}_A$ and $\mathsf{T}_B$.
Again, there is an involutive operation of 
\emph{abstract mirror dual} $\mathsf{T}^{\star}$, exchanging $A$ and $B$ twists.

%Given a $3d$ theory $T$, we may speak of it 

\subsubsection{}

We now recall that in $4d$, for a reductive group $G$, 
there is Yang-Mills $\YM_G$ theory with $\sN = 4$ 
supersymmetry. Again, there are $A$ and $B$-twists with algebraic
meaning.

\begin{example}

The $A$-twisted theory $\YM_{G,A}$ attaches to $X$ the DG category 
$D(\Bun_G(X))$ of $D$-modules on $\Bun_G(X)$, and to $\o{\sD}_x$ 
the DG 2-category $\fL G\mod$ of DG categories with a (strong)
action of the loop group $\fL G$ to $\o{\sD}_x$.
For the latter, the vacuum object is $D(\Gr_G)$ and the (``chiral homology")
functional $\fL G\mod \to \Vect$ is given by tensoring over $D^*(\fL G)$ with 
$D^*(\Bun_G^{level,x})$.

\end{example}

\begin{example}

The $B$-twisted theory $\YM_{G,B}$ attaches to $X$ the DG category 
$\QCoh(\LocSys_G(X))$ of quasi-coherent sheaves on 
$\LocSys_G(X)$, and to $\o{\sD}_x$ attaches
the DG 2-category $\ShvCat_{/\LocSys_{G}(\o{\sD}_x)}$.

\end{example}

\subsubsection{}\label{sss:bx-ym}

Given a stack $\sX$ with a $G$-action, there is a corresponding
boundary condition $\sB_{\sX}$ for $\YM_G$.

For $\sX = G$, this boundary condition is called the \emph{Dirichlet}
boundary condition; we denote it by $\Dir_G$.

For $\sX = \Spec(k)$, this boundary condition is called the \emph{Neumann}
boundary condition; we denote it by $\Neu_G$.

For any boundary condition $\sB$ of $\YM_G$, we let $T_{\sB}$ denote
the $3d$ theory $\pair{\sB}{\Dir_G}$. By standard compatibility of $3d$ and $4d$ 
supersymmetry algebras, any such $T_{\sB}$ has $\sN = 4$ supersymmetry.

For a $3d$ $\sN = 4$ theory $T$, the data of \emph{$G$-flavor symmetry} 
is the data of a boundary condition $\sB$ for $\YM_G$ and an identification
$T \simeq T_{\sB}$. Similarly, \emph{$G$-gauge symmetry} for $T$
is the data of a boundary condition $\sB$ and an identification
$T \simeq \pair{\sB}{\Neu_G}$.

The above is compatible with twists; $\sB$ gives a boundary condition 
$\sB_A$ for the $A$-twist $\YM_{G,A}$, and similarly for $B$-twists. 
The resulting $3d$ (algebraic) theories are $T_{\sB,A}$ and $T_{\sB,B}$.

\begin{rem}

For $\sX$ as above, the boundary condition $\sB_{\sX,A}$ should define
the object $D(\fL \fX) \in \fL G\mod$, and the boundary condition 
$\sB_{\sX,B}$ should define $\IndCoh(\Maps(\o{\sD}_{x,dR},\sX/G)) \in 
\ShvCat_{/\LocSys_G(\o{\sD})}$. Globally, there should be objects
of $D(\Bun_G)$ and $\QCoh(\LocSys_G)$ as well. These objects 
are considered by Ben-Zvi, Sakellaridis, and Venkatesh, who 
term them \emph{period/$L$-sheaves}, interpreting 
them as sheaf-theoretic analogues of period integrals/$L$-function values
from harmonic analysis and number theory. They were also considered
(away from singular loci) in \cite{gaiotto-s} \S 9-10. 

\end{rem}

\begin{rem}

Parallel to Remark \ref{r:3d-t*}, at least at the classical (non-quantum) level,
the boundary condition $\sB_{\sX}$ can be defined more generally for
Hamiltonian $G$-spaces, with $\mu:T^*\sX \to \fg^{\vee}$ inducing
$\sB_{\sX}$ (and again, something weaker than 
a suitable Lagrangian foliation should be needed to quantize).

\end{rem}

\subsubsection{}

In the above setting, \emph{$S$-duality} is supposed to be an equivalence:
\[
\YM_G^{\star} \simeq \YM_G
\]

\noindent of $4d$ theories. Passing to $B$-twists, we obtain:
\[
\YM_{G,A} \simeq \YM_{\check{G},B}.
\]

\noindent Up to smearing\footnote{Throughout this exposition of the general ideas, 
we give ourselves the freedom of ignoring these subtleties. 
In the body of this paper (specifically, \S \ref{s:spectral} and \S \ref{s:indcoh}), 
we treat these technical points in detail in our context. The non-abelian future of the
subject will of course similarly need to confront these points more seriously.} 
away homological algebra subtleties 
(cf. \cite{arinkin-gaitsgory}), this is a form of the geometric Langlands
conjectures, encoding local and global theories and compatibilities
between them (cf. \cite{quantum-langlands-summary}). 

The subsequent discussion is predicated on the existence of $S$-duality, so 
we assume it in what follows. 

\subsubsection{}

Given a boundary condition $\sB$ for $\YM_G$, we let 
$\sB^{\star}$ denote the resulting \emph{abstract mirror dual}
boundary condition for $\YM_G^{\star}$, and then let
$\check{\sB}$ denote the resulting $S$-dual boundary condition for 
$\YM_{\check{G}}$.

\subsubsection{$3d$ mirror symmetry redux}

Now fix a reductive group $G$.

In general, a\footnote{Traditionally, \emph{3d mirror symmetry} refers to
the case where $G = \check{G} = \on{triv}$. We find the present terminology
convenient, although it departs somewhat from established conventions.}
\emph{$3d$ mirror dual pair (for $(G,\check{G})$)}
consists of a pair $(\sX,\sX^{\star})$ where $G \actson \sX$,
$\check{G} \actson \sX^{\star}$, and there is an equivalence:
\[
\check{\sB}_{\sX} \simeq \sB_{\sX^{\star}}
\]

\noindent of boundary conditions for $\YM_{\check{G}}$. I.e., the 
boundary condition $S$-dual to $\sB_{\sX}$ should be $\sB_{\sX^{\star}}$.
For $G$ trivial, this recovers our earlier notion of $3d$ mirror dual pairs.

\subsubsection{}

Suppose $(\sX,\sX^{\star})$ as above. Passing to $B$-twists, we find that the objects:
\[
D(\fL\sX) \in \fL G\mod \simeq 
\ShvCat_{/\LocSys_{\check{G}}(\o{\sD})} 
\ni 
\IndCoh(\Maps(\o{\sD}_{x,dR},\sX^{\star}/\check{G}))
\]

\noindent are supposed to match, where the middle isomorphism is local geometric
Langlands (considered modulo homological subtleties). 

We remark that this does not amount to an equivalence of categories
between the categories on the left and on the right here. Rather, this
statement is inherently conditional on local geometric Langlands. 
(And we intend it to be a little fuzzy regarding homological subtleties.) 

With that said, in the case that $G$ is \emph{abelian}, the above does
amount to an equivalence of categories compatible with certain symmetries. 
(In general, the Whittaker category for the LHS above should be close to the
RHS above.)

\subsubsection{Some examples of mirror dual pairs}

To illustrate the above, we briefly include some examples.
For many other examples, we refer to \cite{wang-table}, which is 
a table of examples maintained by Jonathan Wang. (The terminology in the
first two examples is taken from Ben-Zvi, Sakellaridis, and Venkatesh, 
who are appealing to the analogy with harmonic analysis.)

\begin{example}[Godement-Jacquet]\label{e:godement-jacquet}

For $G = \check{G} = GL(V) \times GL(V)$, the pair
$(GL(V) \times V,\End(V))$ should be $3d$ mirror dual.

In particular, the objects:
\[
D(\fL GL(V) \times \fL V) \in \fL GL(V) \times \fL GL(V) \mod 
\]

\noindent should correspond to:
\[
\IndCoh(\Maps(\o{\sD}_{dR},GL(V) \backslash\End(V)/GL(V)) \in 
\ShvCat_{/\LocSys_{GL(V)\times GL(V)}(\o{\sD})}
\]

\noindent under local geometric Langlands. 

\end{example}

\begin{example}[Tate]

For $V = \bA^1$, 
the above asserts that:
\[
D(\fL\bG_m \times \fL \bA^1) \in \fL \bG_m \times \fL \bG_m \mod 
\]

\noindent corresponds to:
\[
\IndCoh(\sY \times \LocSys_{\bG_m}).
\]

\noindent As our group is abelian, we expect an honest equivalence 
of categories. Passing to $\fL\bG_m$-invariants $\leftrightarrow$ 
$\{$fiber at $\on{triv} \in \LocSys_{\bG_m}\}$, we obtain a conjecture:
\[
D(\fL \bA^1) \simeq \IndCoh(\sY).
\]

\noindent This is our main result.  

\end{example}

\begin{example}[Gaiotto-Witten \cite{gaiotto-witten}, \S 3.3.1]

Let $G = \check{G} = GL(V)$ for $\dim(V) = n$. 
Let ${\sL au}$ be the (Laumon) moduli space
of data $(V_1,V_2,\ldots,V_{n-1},f_1,\ldots,f_{n-1})$ where
$V_i$ is a vector space of dimension $i$ and $f_i:V_i \to V_{i+1}$ is a 
morphism, where by definition, $V_n = V$.

Then the pair:
\[
(\sX,\sX^{\star}) = (GL(V),{\sL au})
\] 

\noindent is expected to be $3d$ mirror dual (with respect to our $(G,\check{G})$).

Unwinding, this in effect gives a sort of formula (or kernel) for 
geometric Langlands for $GL_n$; unfortunately, the necessary symmetries on the
resulting object are not apparent.  

\end{example}

\subsubsection{Where do mirror dual pairs come from?}

We briefly address the stated question. There are two styles of answer.

First, one can try to reverse engineer examples. If one believes a mirror
dual $\sX^{\star}$ to some $\sX$ should exist, 
one can sometimes calculate the ring of
functions on $T^*{\sX^{\star}}$ as the Coulomb
branch\footnote{I.e., total cohomology of endomorphisms of the unit object
in $T_{\sX^{\star},B} = \IndCoh(\Maps(\o{\sD}_{dR},\sX))$.} of $\sX^{\star}$,
which must correspond to the Higgs branch\footnote{I.e., total cohomology 
of endomorphisms of the unit object
in $T_{\sX,A} = D(\fL \sX)$.} of $\sX$. 

Alternatively, as in the work of Ben-Zvi, Sakellaridis, and Venkatesh,
one can reverse engineer examples by fitting phenomena from number theory
into the above framework via the analogy between automorphic forms
and automorphic $D$-modules.

But sometimes (always?) there is a better answer than either. 
Many $3d$ mirror symmetry examples can be derived from (super)string 
theory dualities and the existence of $\sM$-theory. (For example, see 
\cite{gaiotto-witten} \S 3.1.2 for the derivation in our example.)
We are not aware of a mathematical counterpart to this idea, i.e., a simple conjectural
framework that subsumes all examples in $3d$ mirror symmetry. Clearly this
would be highly desirable. 

\begin{rem}

Forthcoming work of the first author and Philsang Yoo will develop in detail the
mathematical derivation of mirror pairs from $S$-duality in string theory.

\end{rem}

\subsubsection{The role of this paper}

Our objective in writing this paper was to test the above ideas in the
simplest\footnote{Not covered by geometric Langlands/$S$-duality for $\YM$, 
and in which the full weight of the infinite dimensional geometry appears.} 
case of interest, which is the Tate case discussed above. 
Here $(G,\check{G}) = (\bG_m,\bG_m)$ and\footnote{The discrepancy here with the
ungauged version considered in Example \ref{e:mirror-tate-ungauged} is 
explained by the fact that $\Dir$ and $\Neu$ are $S$-dual for tori.} 
$(\sX,\sX^{\star}) = (\bA^1,\bA^1)$. 

As we consider abelian gauge groups,
for which geometric Langlands is unconditional, this example can be 
studied in complete detail. 
We perform this study at the level of categories of line operators. We will
return to global considerations (compatibility with chiral homology) in 
future work.

Given that we obtain complete positive results in this case, and
because many of the technical difficulties inherent in the $3d$ mirror symmetry
project occur in our example, we believe that our results provide strong
support for the general conjectures. 

\subsubsection{Some references}\label{ss:3d-refs}

Many of the ideas discussed above were developed in 
collaborative work, often unpublished, of a number of mathematical physicists. 
We are grateful to many people who have shared these ideas
with us over the years, and find their intellectual generosity inspiring.
The downside of this situation is that we find some difficulty in accurately
attributing priority for the ideas. We do our best here, but apologize in 
advance for any omissions or inaccuracies, which are not intended as slights.

The first instances of $3d$ mirror symmetry were considered in 
\cite{intriligator-seiberg}. The connection with string theory dualities
was made in \cite{hanany-witten}. These ideas were developed further
in \cite{gaiotto-witten}, which considered interactions between the
\cite{hanany-witten} constructions and $S$-duality for super Yang-Mills.
In turn, \cite{gaiotto-s} translated those ideas into conjectures regarding
geometric Langlands via the Kapustin-Witten dictionary \cite{kapustin-witten}
between Langlands and $S$-duality. The physics literature on BPS line operators 
is too vast to survey here, but we refer the reader to \cite{dggh} for a
review of the literature in the $3d$ $\sN = 4$ context. 

The algebro-geometric interpretation of the $A$-side discussed
above is suggested by work of Braverman, Finkelberg and Nakajima
in \cite{nakajima-coulomb}, \cite{bfn}, \cite{bfn-slices}, 
which emphasized the role of Coulomb branches.
The description of categories of line operators in twisted $3d$ $\sN = 4$
theories was given in unpublished work of Kevin Costello, 
Tudor Dimofte, Davide Gaiotto, Philsang Yoo, and the first 
author.\footnote{These ideas were recorded in \S 1.1 of 
\cite{dggh} and in \cite{bf-notes} \S 7.} 
The earliest derivation
was based on applying the method of Elliott-Yoo \cite{elliott-yoo} 
in the $3d$ $\sN = 4$ context, 
i.e., calculating (derived) moduli spaces of solutions to Euler-Lagrange equations and 
quantizing (in the shifted symplectic sense).\footnote{Because 
we are not aware of a publicly available
account of this derivation, we direct the reader to \cite{costello-x-talks},
which is a series of talks Kevin Costello gave in 2014 that discuss some of 
the main ingredients. It is the earliest talk we know of that contains
these components.} 
By applying the yoga of $3d$ mirror symmetry, this description of 
categories of line operators also 
led to a conjectural form of Theorem \ref{t:main} 
(and other related examples).\footnote{For a physics-oriented discussion of this
work, see \cite{tudor-talk}.}
Philsang Yoo and the first author considered these line operator categories
in the framework of local geometric Langlands as a mathematical interpretation 
of \cite{gaiotto-witten} (inspired in part by \cite{bfn-ring}).\footnote{See
\cite{phil-talk} and \cite{jh-talk} for discussions of this work.}
On the $A$-side, an alternative derivation of the category of
line operators is suggested by 
\cite{voa-3d} \S 4; here the method is as described in Remark \ref{r:voa-3d}. 

Connections between $3d$ mirror symmetry 
and number theory (e.g., duality for spherical varieties and period
integrals \cite{sakellaridis-venkatesh}) were developed by Ben-Zvi, Sakellaridis, 
and Venkatesh, in forthcoming\footnote{The work is also highly publicized -- 
there are recorded lectures about the work readily available online 
(see e.g. \cite{bz-talk}).}
work, which has led to profound new insights and many new examples of 
dual pairs.

Finally, in addition to the above works, 
we have drawn inspiration particularly from \cite{bf-notes},
 \cite{blpw-hypertoric}, \cite{blpw}
\cite{bpw}, \cite{loop-spaces-and-connections},
\cite{char-field-thry}, 
\cite{krs}, \cite{ksv}, and \cite{teleman-icm} in our thinking about
mirror symmetry. 

\subsubsection{Some related works}

We also highlight some works that are closely related to ours.

First, \cite{ccg} studies a version of our problem
near the formal completion of $0 \in \sY$ (i.e., Theorem \ref{t:main} 
\emph{in perturbation theory}) using VOAs; in their picture (based on the
method of Remark \ref{r:voa-3d}), the
Lie algebra controlling the deformation theory is a 
$\mathfrak{psu}(1|1)$ Kac-Moody algebra. We understand that these ideas 
are currently being developed further by Andrew Ballin, Thomas Creutzig,
Tudor Dimofte, and Wenjun Niu.

Second, the works \cite{bfgt} and \cite{bft} 
of Braverman-Finkelberg-Ginzburg-Travkin also obtained geometric results from 
mirror symmetry conjectures, but in the non-abelian setting. 
There it is not possible at present
to prove that boundary conditions match under $S$-duality, since
local geometric Langlands is only conjectural for non-abelian $G$.
Therefore, the cited authors instead consider the categories obtained by 
pairing the vacuum boundary condition(s) for
$\YM_G$/$\YM_{\check{G}}$, and
establish the resulting conjectures for certain examples of 
$3d$ mirror dual pairs $(G,\check{G})$ and $(\sX,\sX^{\star})$. 

\subsection{Outline of the paper}

\subsubsection{} We now describe the contents of the present work.

\subsubsection{}\label{ss:intro-strategy}

First, we describe the general strategy. 

By definition, there are fully faithful functors:
\[
D(\fL_n^+\bA^1) \into D^!(\fL^+ \bA^1) \into D^!(\fL \bA^1).
\]

\noindent Here $\fL_n^+\bA^1$ is the (scheme corresponding to) 
the vector space $k[[t]]/t^nk[[t]]$; $\fL^+ \bA^1$ is the
(scheme corresponding to the pro-finite dimensional) vector 
space $k[[t]]$, and $\fL\bA^1$ is the (indscheme corresponding to
the Tate) vector space $k((t))$. The first functor is a $!$-pullback
and the second functor is a (suitably normalized) $*$-pushforward functor.

The main step in our construction is the construction of corresponding
functors:
\[
\Delta_n:D(\fL_n^+\bA^1) \to \IndCoh^*(\sY).
\]

\noindent As the left hand side is modules over a Weyl algebra, this amounts
to constructing an object of the right hand side with an action of a Weyl
algebra. We construct this object explicitly, via generators and relations. 

From there, we bootstrap up to construct a functor
$\Delta$ as in Theorem \ref{t:main}. We show it is an equivalence using 
the geometry of $\sY$.

\begin{rem}

There is an adage in geometric Langlands that most of the work occurs on the
geometric (i.e., automorphic) side. In our work it is the opposite: the 
geometric side is easy, and most of our work is on the spectral side. 

\end{rem}

\subsubsection{}

We introduce $\sY$ (and related spaces) in \S \ref{s:spectral}. Here we
give the main geometric tools in our study of $\sY$. We study
ind-coherent sheaves on $\sY$ in \S \ref{s:indcoh}, after some preliminary 
remarks about Abel-Jacobi maps in \S \ref{s:aj}. We construct the
functors $\Delta_n$ in \S \ref{s:weyl}; as indicated above, this is our main
construction. We check the corresponding compatibility of this construction 
with class field theory
in \S \ref{s:cft}. We then show fully faithfulness of $\Delta_n$ in 
\S \ref{s:ff}. We then prove Theorem \ref{t:main} in \S \ref{s:main}.

\subsection{Categorical conventions}

At various points, we use homotopical algebra and derived
algebraic geometry in Lurie's higher categorical form,
cf. \cite{htt}, \cite{higheralgebra}, \cite{sag}, though our notation more
closely follows the conventions of \cite{grbook}.

In general, our terminology should be assumed to be derived.
We refer to $(\infty,1)$-categories as \emph{categories} and 
$\infty$-groupoids as \emph{groupoids}. We let
$\Gpd$ denote the category of groupoids.

We let $\DGCat_{cont}$ denote the category of cocomplete\footnote{Rather, presentable.}
DG categories (over $k$) and continuous DG functors, cf. \cite{grbook} \S I.1.10.
We let $\otimes$ denote its standard (\emph{Lurie}) tensor product. 
We let $\Vect \in \DGCat_{cont}$ denote the DG category of (chain complexes
of) vector spaces. For a DG algebra $A/k$, we let $A\mod$ denote the
corresponding DG category. 
For a DG category $\sC$ and $\sF,\sG \in \sC$, we let
$\ul{\Hom}_{\sC}(\sF,\sG) \in \Vect$ denote the complex of
maps from $\sF$ to $\sG$. 

For a DG category $\sC$ with a $t$-structure, we use cohomological
indexing conventions and let $\sC^{\leq 0}$ denote the connective
objects, $\sC^{\geq 0}$ denote the coconnective objects, and
let $\sC^{\heart} \coloneqq \sC^{\leq 0} \cap \sC^{\geq 0}$ denote
the heart of the $t$-structure. Similarly, we let 
$\sC^+$ denote the eventually coconnective objects.

By default, \emph{schemes} are DG\footnote{It will turn out a posteriori 
that the primary objects
we consider are classical schemes (or stacks). But because we are studying
derived categories of coherent sheaves, the machinery and perspective of derived
algebraic geometry is quite convenient.}  
 schemes over $k$. We let
$\AffSch$ denote the category of affine schemes (over $k$),
and $\Sch$ the category of schemes (over $k$).

When we wish to refer to consider various non-derived objects as 
derived objects, we use the
term \emph{classical}. So we may speak of classical schemes, classical rings,
classical vector spaces, classical $A$-modules, etc., all of which 
are full subcategories of the corresponding (suitably) derived categories.

Finally, we sometimes use standard ideas and notation from the theory of
categories with $G$-actions. We refer to \cite{dario-*/!} for an introduction
to these ideas.  

\subsection{Acknowledgements}

We thank Sasha Beilinson, David Ben-Zvi, Dario Beraldo, 
Mat Bullimore, Dylan Butson, Sasha Braverman,
Justin Campbell, Kevin Costello, Gurbir Dhillon, Tudor Dimofte, Chris Elliott,
Tony Feng, Davide Gaiotto, Nik Garner, Dennis Gaitsgory, Sam Gunningham, 
Joel Kamnitzer, Ivan Mirkovic, Tom Nevins, Wenjun Niu, Pavel Safronov, 
Yiannis Sakellaridis, Ben Webster, Jonathan Wang, Alex Weekes, and
Philsang Yoo for many helpful and inspiring conversations related
to this work. The first author particularly wishes to acknowledge and thank 
Philsang Yoo for joint work conjecturing Theorem \ref{t:main}. 

J.H. is part of the Simons Collaboration on Homological Mirror Symmetry supported by 
Simons Grant 390287. S.R. was supported by NSF grant DMS-2101984.
This research was supported in part by Perimeter Institute for Theoretical Physics. 
Research at Perimeter Institute is supported by the Government of Canada through the 
Department of Innovation, Science and Economic Development Canada and by the Province of 
Ontario through the Ministry of Research, Innovation and Science.

\section{Geometry of the spectral side}\label{s:spectral}

\subsection{Overview}

In this section, we define $\sY$ and establish our main algebro-geometric tools
for studying its coherent sheaves.

Here is a more detailed overview of this section. 

In \S \ref{ss:jet-notation}, we give background on jet and loop spaces,
essentially setting up notation for later use.
In \S \ref{ss:rk1-ls}, we give background on $\LocSys_{\bG_m}$.
In \S \ref{ss:y-defin}, we define $\sY$. In \S \ref{ss:y-trun},
we introduce variants of $\sY$ that are more traditional objects
of algebraic geometry. In \S \ref{ss:flatness}, we prove a series 
of flatness results; these are the key tools for studying $\sY$ and
its coherent sheaves. 

The remainder of the section consists of a series of codas, essentially 
developing material and notation for later reference. These may be skipped
at first pass and referred to as needed. In \S \ref{ss:y-not}, we introduce
additional notation for later reference. In \S \ref{ss:fund}, we deduce
some exact sequences from our flatness results; these will be used
later in some inductive arguments. In \S \ref{ss:un}, we describe
certain well-behaved opens in $\sZ^{\leq n}$. 
In \S \ref{ss:coords}, we explicitly describe $\sY$ using
generators and relations. Finally, in \S \ref{ss:rs-sketch}, we draw pictures
explicitly describing the geometry in the regular singular situation.

\subsection{Notation for jet and loop spaces}\label{ss:jet-notation}

\subsubsection{}

For a commutative DG algebra $A \in \ComAlg$, 
$A[[t]] \in \ComAlg$ is defined as 
$(\lim_n A \otimes k[[t]]/t^n)$ and $A((t))$ is defined
as $A[[t]] \otimes_{k[[t]]} k((t))$. We also let
$A[[t]]/t^n \coloneqq A \otimes_k k[[t]]/t^n$.
% emph: notation is derived?

\subsubsection{}

For a prestack $Z$, 
$\fL Z \in \PreStk \coloneqq \TwoHom(\AffSch^{op},\Gpd)$ denote the functor:
\[
\Spec(A) \mapsto \Hom_{\AffSch}(\Spec(A((t))),Z).
\]

\noindent Similarly, we let $\fL^+ Z$ denote the functor:
\[
\Spec(A) \mapsto \Hom_{\AffSch}(\Spec(A[[t]]),Z).
\]

For $n \geq 0$, we let $\fL_n^+$ denote the functor:
\[
\Spec(A) \mapsto \Hom_{\AffSch}(\Spec(A[[t]]/t^n),Z).
\]

There is an evident comparison map:
\[
\vareps_Z:\fL^+ Z \to \underset{n}{\lim} \, \fL_n^+ Z.
\]

We often use the following lemma without mention.

\begin{lem}

For $Z$ an algebraic stack of 
the form $Y/G$ with $Y$ affine and $G$ an affine 
algebraic group, the map $\kappa_Z$ is an isomorphism.

\end{lem}

\begin{proof}

The result is obvious when $Z = Y$ is affine.
For $Z = \bB G$, we refer e.g. to \cite{cpsii} Lemma 2.12.1.
The argument from \emph{loc. cit}. immediately
extends to the general form of $Y$ considered here.

\end{proof}

\begin{rem}

In fact, this result is true much more 
generally for Noetherian algebraic
stacks with affine diagonals: see \cite{bhatt-hl} Corollary 1.5.

\end{rem}

\begin{notation}

We often let $\ev:\fL^+ Z \to \fL_1^+ Z = Z$ 
denote the evaluation map.

\end{notation}

\subsubsection{}

We have the following basic representability result.

\begin{prop}\label{p:loop-rep}

Suppose $Z$ is an affine scheme almost of finite type.
Then $\fL^+ Z$ is an affine scheme and $\fL Z$ is an 
indscheme.

\end{prop}

\subsubsection{}

We have the following variant of the above.

Suppose $Z$ is equipped with a $\bG_m$-action. 
We let $\fL Zdt \in \PreStk$ denote the functor:
\[
\fL Z dt \coloneqq \fL(Z/\bG_m) \underset{\fL \bB \bG_m}{\times} \Spec(k)
\]

\noindent where $\Spec(k) \to \fL \bB \bG_m$ corresponds to the
(continuous) tangent sheaf $T_{\o{\sD}}$ on $\o{\sD}$.
We define $\fL^+Z dt$ similarly. A choice of trivialization
of $T_{\o{\sD}}$ defines an isomorphism $\fL Z dt \simeq \fL Z$
identifying $\fL^+ Z dt$ and $\fL^+ Z$; in particular,
Proposition \ref{p:loop-rep} applies in this setting.

Informally, $\fL Z dt$ parametrizes sections of
the fiber bundle $\o{\Theta}(\Omega_{\o{\sD}}^1) \overset{\bG_m}{\times} Z \to 
\o{\sD}$, where $\Omega_{\o{\sD}}^1$ is the line bundle 
of (continuous) 1-forms on $\o{\sD}$ and $\Theta(-)$ 
(resp. $\o{\Theta}$) denotes the
(resp. punctured) total space of a line bundle. 

\begin{example}

$\fL \bA^1$ is the algebro-geometric version of the space of Laurent
series, while $\fL \bA^1 dt$ is the algebro-geometric verison of the
space of 1-forms on the punctured disc (for the standard $\bG_m$-action 
on $\bA^1$ by homotheties). Similarly, $\fL \bG_m$ is the algebro-geometric
version of the space of invertible Laurent series.
One easily finds that these indschemes are formally smooth and classical.

\end{example}

\subsection{Rank $1$ local systems}\label{ss:rk1-ls}

We now give a detailed study of $\LocSys_{\bG_m}$, the moduli space of 
rank $1$ local systems on the punctured disc. This material is well-known,
but it is convenient to review to introduce notation and constructions
we will need in the more complicated setting of our space $\sY$.

\subsubsection{}

We define $\LocSys_{\bG_m}$ as follows.

We have a map $\fL \bG_m \to \fL \bA^1 dt$ sending
$f \in \fL \bG_m$ to $d\log(f)$, cf. \cite{locsys} \S 1.12.
This map is a map of (classical) group indschemes, where
$\fL \bA^1 dt$ is given its natural additive structure.

We define $\LocSys_{\bG_m}$ as the stack\footnote{I.e., we sheafify
for the fppf topology. It is equivalent to sheafify for the 
Zariski topology as $\Ker(\fL^+ \bG_m \to \bG_m)$ is pro-unipotent.
For the same reason, the resulting prestack is a sheaf for the 
fpqc topology. In other words, there is no room for ambiguity in 
sheafifying.}
quotient $\fL \bA^1 dt/\fL \bG_m$. 

\begin{rem}

For normalization purposes, we remark that for 
a point $\omega \in \fL \bA^1 dt$, we consider 
$\nabla = d - \omega$ as the corresponding rank $1$ connection
(on the trivial line bundle).

\end{rem}

\subsubsection{}

We need variants of $\LocSys_{\bG_m}$ as well.

Define:
\[
\LocSys_{\bG_m,\log} \coloneqq 
\LocSys_{\bG_m} \underset{\fL \bB\bG_m}{\times} \fL^+ \bB \bG_m = 
\fL \bA^1 dt/\fL^+ \bG_m.
\]

\noindent This is the moduli space of line bundles on the
disc with a connection on the punctured disc. 

There is an
evident action of $\Gr_{\bG_m} \coloneqq \fL \bG_m/\fL^+ \bG_m$ 
on $\LocSys_{\bG_m,\log}$ such that 
$\LocSys_{\bG_m} = \LocSys_{\bG_m,\log}/\Gr_{\bG_m}$.

\subsubsection{}

Define $\fL^{pol} \bA^1 dt$ as the quotient
$\fL \bA^1 dt / \fL^+\bA^1 dt$, the quotient being with respect
to the additive structure. In other words, $\fL^{pol} \bA^1 dt$
parametrizes polar parts of differential forms on the disc.
Note that $\fL^{pol} \bA^1 dt$ is an indscheme of ind-finite type.

As $d\log$ maps $\fL^+ \bG_m$ into $\fL^+ \bA^1 dt$, the
projection map:
\[
\fL \bA^1 dt \to \fL^{pol} \bA^1 dt
\]

\noindent intertwines the gauge action of $\fL^+ \bG_m$ on 
the left hand side with its trivial action on the right hand
side. Therefore, we obtain a canonical map:
\[
\Pol: \LocSys_{\bG_m,\log} \to \fL^{pol} \bA^1 dt.
\]

\noindent This map takes the polar part of a connection.

\subsubsection{}

By construction, there is a map:
\[
\LocSys_{\bG_m,\log} \to \bB \fL^+ \bG_m \xar{\ev} \bB \bG_m.
\]

\noindent This map sends a pair $(\sL,\nabla)$ (with $\sL$ a line
bundle on the disc and $\nabla$ its connection on the punctured
disc) to the fiber of $\sL$ at the origin.

The following result is well-known.

\begin{prop}\label{p:locsys-log}

The map:
\[
\LocSys_{\bG_m,\log} \to \fL^{pol} \bA^1 dt \times \bB \bG_m
\]

\noindent is an isomorphism.

\end{prop}

\begin{proof}

The map $\fL^+ \bG_m \xar{d\log,\ev} \fL^+ \bA^1 dt \times \bG_m$
is easily seen to be an isomorphism. The result is immediate
from here.

\end{proof}

\subsubsection{}

We now deduce a similar description of $\LocSys_{\bG_m}$
from Proposition \ref{p:locsys-log}.

There is a canonical residue map 
$\Res:\fL^{pol} \bA^1 dt \to \bA^1$. There is a canonical projection
$\fL^{pol} \bA^1 dt \to \Ker(\Res)$ as the latter identifies
with $\fL \bA^1 dt/t^{-1} \cdot \fL^+ \bA^1 dt$,
so we obtain a product decomposition:
\[
\fL^{pol} \bA^1 dt = \bA^1 \times \Ker(\Res).
\]

\begin{prop}\label{p:locsys}

There is an isomorphism $\LocSys_{\bG_m} \simeq 
\bA^1/\bZ \times \Ker(\Res)_{dR} \times \bB \bG_m$ fitting
into a commutative diagram:
\[
\xymatrix{
\LocSys_{\bG_m,\log} \ar[rr]^{\text{Prop. \ref{p:locsys-log}}} 
\ar[d] & & \fL^{pol} \bA^1 dt \times \bB \bG_m \ar[d] \\
\LocSys_{\bG_m} \ar[rr]^{\simeq}
& &
\bA^1/\bZ \times \Ker(\Res)_{dR} \times \bB \bG_m
}
\]

\end{prop}

\begin{proof}

There is a canonical homomorphism
$\Gr_{\bG_m} \to \bZ$ given by minus\footnote{This sign is included to match
with standard normalizations.} the valuation map. This map splits 
canonically as well: $\bZ = \Gr_{\bG_m}^{red}$.
Therefore, we obtain a canonical product decomposition:
\[
\Gr_{\bG_m}^{red} = \bZ \times \Gr_{\bG_m}^{\circ}
\]

\noindent for $\Gr_{\bG_m}^{\circ}$ the connected component
of the identity. Here we consider $\bZ$ as a discrete indscheme
over $k$ in the natural way, i.e., as $\coprod_{n \in \bZ} \Spec(k)$.

We have a commutative diagram:
\[
\xymatrix{
\Gr_{\bG_m} \ar[r]^{d\log} \ar@{=}[d] & 
\fL^{pol} \bA^1 dt \ar@{=}[d] \\
\bZ \times \Gr_{\bG_m}^{\circ} \ar[r] 
& \bA^1 \times \Ker(\Res) 
}
\]

\noindent where the bottom map is obtained as the product
of the homomorphisms $\bZ \into \bA^1$ and 
$d\log:\Gr_{\bG_m}^{\circ} \to \Ker(\Res)$. The latter
is easily seen to induce an isomorphism between
$\Gr_{\bG_m}^{\circ}$ and the formal group 
$\Ker(\Res)_0^{\wedge}$ of $\Ker(\Res)$. This gives the claim.

\end{proof}

\begin{rem}

This isomorphism actually depends mildly on the choice of uniformizer $t$;
this is needed to trivialize the action of 
$\Gr_{\bG_m}$ on the $\bB \bG_m$-factor. 

\end{rem}

\subsubsection{}

We will also use \emph{truncated} versions of the above spaces in 
which we bound the irregularity of our local systems.

We fix $n \in \bZ^{\geq 0}$ in what follows.

\subsubsection{}

Define $\fL^{\leq n} \bA^1 dt \subset \fL \bA^1 dt$ as the
classical closed subscheme whose points are differential 
forms with poles of order at most $n$. Therefore, 
$\fL^{\leq n} \bA^1 dt$ is the image of $\fL^+ \bA^1 dt$ under the
map $t^{-n} \cdot -:\fL \bA^1 dt \to \fL \bA^1 dt$.

We similarly define $\fL^{pol,\leq n} \bA^1 dt$ as 
$\fL^{\leq n} \bA^1 dt/\fL^+ \bA^1 dt$. For $n > 0$, we define
$\Ker(\Res)^{\leq n}$ as $\Ker(\fL^{\leq n} \bA^1 dt \to \bA^1)$.

\subsubsection{}\label{ss:loop-gr-trun}

For $n>0$, define $\fL^{\leq n} \bG_m$ as the fiber product:
\[
\fL^{\leq n} \bG_m \underset{\fL \bA^1 dt}{\times} \fL^{\leq n} \bA^1 dt
\]

\noindent where we are using $d\log:\fL \bG_m \to \fL \bA^1 dt$.

For\footnote{We separate the cases to be derivedly correct.} 
$n = 0$, define $\fL^{\leq 0} \bG_m$ as $\fL^+ \bG_m$.

Finally, define $\Gr_{\bG_m}^{\leq n}$ as $\fL^{\leq n}\bG_m/\fL^+\bG_m$.

\begin{lem}

$\fL^{\leq n} \bG_m$ is a formally smooth classical indscheme.

\end{lem}

\begin{proof}

It suffices to show the same for $\Gr_{\bG_m}^{\leq n}$.

As in the proof of Proposition \ref{p:locsys}, the map
$\Gr_{\bG_m} \xar{d\log} \fL \bA^1 dt \to \Ker(\Res)$ induces an
isomorphism:
\[
\Gr_{\bG_m}^{\leq n} \isom \bZ \times (\Ker(\Res)^{\leq n})_0^{\wedge}.
\]

\noindent Here $\Ker(\Res)^{\leq n}$
is defined as the kernel of the
residue on polar forms with poles of order $\leq n$. 
As the latter is an affine space, its formal completion at the origin
is formally smooth and classical. This gives the claim.

\end{proof}

\subsubsection{}

Next, define $\LocSys_{\bG_m}^{\leq n}$ as the quotient
$\fL^{\leq n} \bA^1 dt/\fL^{\leq n} \bG_m$ under the gauge action.
We similarly define $\LocSys_{\bG_m,\log}^{\leq n}$ as 
$\fL^{\leq n} \bA^1 dt/\fL^+ \bG_m$. By definition, 
we have a Cartesian diagram:

\begin{prop}\label{p:locsys-trun}

The isomorphisms from Propositions \ref{p:locsys-log} and \ref{p:locsys}
induce further isomorphisms:
\[
\xymatrix{
\LocSys_{\bG_m,\log}^{\leq n} \ar[rr]^{\simeq}
\ar[d] & & \fL^{pol,\leq n} \bA^1 dt \times \bB \bG_m \ar[d] \\
\LocSys_{\bG_m,\log} \ar[rr]^{\simeq}
 & & \fL^{pol} \bA^1 dt \times \bB \bG_m
}
\]

\noindent and:
\[
\xymatrix{
\LocSys_{\bG_m}^{\leq n} \ar[rr]^{\simeq} \ar[d] 
& &
\bA^1/\bZ \times \Ker(\Res)_{dR}^{\leq n} \times \bB \bG_m \ar[d] \\
\LocSys_{\bG_m} \ar[rr]^{\simeq}
& &
\bA^1/\bZ \times \Ker(\Res)_{dR} \times \bB \bG_m.
}
\]

\end{prop}

This result is clear from our earlier analysis.

\subsection{Definition of $\sY$}\label{ss:y-defin}

\subsubsection{}

We define $\sY$, the moduli of rank $1$ de Rham local systems on $\o{\sD}$
with a flat section, as follows.

First, observe that we have a map $d:\fL \bA^1 \to \fL \bA^1 dt$
defined by the exterior derivative on $\o{\sD}$.
We also have a product map $\fL \bA^1 \times \fL \bA^1 dt \to \fL \bA^1 dt$
coming from the product $\bA^1 \times \bA^1 \to \bA^1$.

Let $\sY^{\prime}$ denote the equalizer:
\[
\sY^{\prime} \coloneqq \on{Eq}
\big(\fL \bA^1 \times \fL \bA^1 dt 
\underset{(g,\omega) \mapsto g\omega}{\overset{d\circ \pi_1}{\rightrightarrows}}
\fL \bA^1 dt\big) \in \PreStk
\]

\noindent We emphasize that this is a \emph{derived} equalizer;
although the terms appearing are classical indschemes, this
equalizer is a priori a DG indscheme.

There is a canonical $\fL \bG_m$-action on $\sY^{\prime}$,
heuristically given by the formula:
\[
(f \in \fL \bG_m,(g,\omega) \in \sY^{\prime}) \mapsto 
(fg,\omega+d\log(f)).
\]

\noindent We will define this action more rigorously below,
but in the meantime, we define:
\[
\sY \coloneqq \sY^{\prime}/\fL \bG_m.
\]

\begin{rem}

Let us explain the above formulae.

Note that $\sY^{\prime}$ by definition parametrizes pairs
$(g \in \fL\bA^1, \omega \in \fL \bA^1 dt) $ with
$dg = g\omega$, which we can rewrite as 
$\nabla g = 0$ for $\nabla \coloneqq d - \omega$. 
This data amounts to a connection on the trivial bundle on $\o{\sD}$ 
(corresponding to $\omega$) and a flat section (corresponding to $g$). 

Quotienting by $\fL \bG_m$ amounts to modding out by gauge transformation,
i.e., not fixing a trivialization on our line bundle on $\o{\sD}$.

\end{rem}

\subsubsection{}\label{ss:action-rigorous}

Above, we did not define the action of $\fL \bG_m$ on $\sY^{\prime}$
completely rigorously: as $\sY^{\prime}$ is a priori DG, such
formulae are not sufficient. (In fact, using Theorem \ref{t:cl}, $\sY^{\prime}$
is classical, and the implicit anxiety here is not needed.)

Here are two approaches. 

First, in \cite{dennis-punctured-disc}, 
Gaitsgory defines a prestack $\Maps(\o{\sD}_{dR},Y)$ for any 
prestack $Y$. Taking $Y = \bA^1/\bG_m$ then immediately gives the definition
of $\sY$ as $\Maps(\o{\sD}_{dR},\bA^1/\bG_m)$. (One readily
finds $\Maps(\o{\sD}_{dR},\bB \bG_m) = \LocSys_{\bG_m}$ and
$\sY^{\prime} = \Maps(\o{\sD}_{dR},\bA^1/\bG_m) 
\times_{\Maps(\o{\sD}_{dR},\bB \bG_m)} \Omega_{\bG_m}^1$, consistent
with our constructions.)

We prefer an alternative approach that we find more explicit, and that we
describe in full detail here. The reader who is not concerned about
homotopical details (which are not serious anyway by Theorem \ref{t:cl})
may skip this material.

Below, we make various constructions with $\fL \bA^1 \times \fL \bA^1 dt$. Here we may
just work with formulae as this indscheme is classical.

We consider two monoid structures on $\fL \bA^1 \times \fL \bA^1 dt$. 
The first has product:
\[
(g,\omega) \dot (\tilde{g},\tilde{\omega}) \coloneqq 
(g\tilde{g},\omega+\tilde{\omega})
\]

\noindent while the second has product:
\[
(g,\omega) \star (\tilde{g},\tilde{\omega}) \coloneqq 
(g\tilde{g},g\tilde{\omega}+\tilde{g}\omega).
\]

\noindent Moreover, the map:
\[
\begin{gathered}
\mu:\fL \bA^1 \times \fL \bA^1 dt \to \fL \bA^1 \times \fL \bA^1 dt \\
(g,\omega) \mapsto (g,dg-g\omega)
\end{gathered}
\]

\noindent is a map of monoids $(\fL \bA^1 \times \fL \bA^1 dt,\dot) \to 
(\fL \bA^1 \times \fL \bA^1 dt,\star)$.

We obtain a commutative diagram of maps of monoids:
\[
\xymatrix{
\fL \bG_m \ar[rrr]^(.4){f \mapsto (f,d\log(f))} 
\ar[d] &&& (\fL \bA^1 \times \fL \bA^1 dt,\dot) \ar[d]^{\mu} \\
\on{*} \ar[rrr] &&& (\fL \bA^1 \times \fL \bA^1 dt,\star)
}
\]

Therefore, we obtain a canonical map of monoids:
\[
\fL \bG_m \to (\fL \bA^1 \times \fL \bA^1 dt,\dot) 
\underset{(\fL \bA^1 \times \fL \bA^1 dt,\star)}{\times}
\fL \bA^1
\]

\noindent where the previously unconsidered map is this
fiber product is $\fL \bA^1 \xar{(\id,0)} \fL \bA^1 \times \fL \bA^1 dt$; the source $\fL \bA^1$ is given its natural
product structure. 

The right hand side above evidently
identifies with $\sY^{\prime}$, so we obtain a monoid
structure on $\sY^{\prime}$ and a map of monoids
$\fL \bG_m \to \sY^{\prime}$. In particular, we obtain
an action of $\fL \bG_m$ on $\sY^{\prime}$; this is
our desired action. 

\subsubsection{}

We now formulate the following result, whose proof will
be given in \S \ref{ss:z-refinement}.

\begin{thm}\label{t:cl}

$\sY$ is a classical prestack.

\end{thm}

In particular, $\sY$ is completely determined by its values
on usual commutative rings, i.e., not commutative DG rings.
(We formally deduce the same for $\sY^{\prime}$.)

\subsection{Intermediate spaces}\label{ss:y-trun}

As for $\LocSys_{\bG_m}$, there are several variants of 
$\sY$ that will be crucial to our
study.

\subsubsection{}

First, we define: 
\[
\sY_{\log} = \sY \times_{\fL \bB\bG_m} \fL^+ \bB \bG_m = 
\sY^{\prime}/\fL^+\bG_m.
\]

Note that we have a canonical action of $\Gr_{\bG_m}$ on
$\sY_{\log}$ with $\sY_{\log}/\Gr_{\bG_m} = \sY$.

\subsubsection{}

By construction, $\sY_{\log}$ parametrizes the data
$(\sL,\nabla,s)$ where $\sL$ is a line bundle
on the disc $\sD$, 
a connection $\nabla$ on $\sL|_{\o{\sD}}$, and
$s \in \Gamma(\o{\sD},\sL)$ a flat section. 

We have a natural ind-closed $\sZ \subset \sY_{\log}$ parametrizing
similar data, but with $s \in \Gamma(\sD,\sL)$ flat as a
section on the punctured disc.

Formally, we define: 
\[
\sZ \coloneqq \on{Eq}
\big(\fL^+ \bA^1 \times \fL \bA^1 dt 
\underset{(g,\omega) \mapsto g\omega}{\overset{d\circ \pi_1}{\rightrightarrows}}
\fL \bA^1 dt\big)/\fL^+ \bG_m \in \PreStk.
\]

\noindent Here the $\fL^+ \bG_m$-action is constructed as in
\S \ref{ss:action-rigorous}.

\begin{rem}

The subspace $\sZ \subset \sY_{\log}$ plays a key role in our
main construction in \S \ref{s:weyl}.

\end{rem}

\begin{rem}

In formulae, we have:
\[
\sZ = \Maps(\o{\sD}_{dR},\bA^1/\bG_m) 
\underset{\Maps(\o{\sD},\bA^1/\bG_m)}{\times}
\Maps(\sD,\bA^1/\bG_m)
\] 

\noindent for suitable meaning of $\o{\sD}_{dR}$
(cf. \cite{}).

\end{rem}

\begin{rem}

A little informally, $\sZ$ is the moduli of $(\sL,\nabla,s)$ with $\sL$ 
a line bundle on $\sD$, $\nabla$ a connection on the punctured disc,
and $s \in \Gamma(\sD,\sL)$ with $\nabla(s) = 0 \in 
\Gamma(\o{\sD},\sL \otimes \Omega^1)$.

\end{rem}

\subsubsection{} 

Next, we define truncated versions of the above spaces. 
Fix $n \geq 0$.
 
We then define:
\[
\sY_{\log}^{\leq n} \coloneqq \sY_{\log} 
\underset{\LocSys_{\bG_m,\log}}{\times} 
\LocSys_{\bG_m,\log}^{\leq n}.
\]

We define:
\[
\sZ^{\leq n} \coloneqq \on{Eq}
\big(\fL^+ \bA^1 \times \fL^{\leq n} \bA^1 dt 
\underset{(g,\omega) \mapsto g\omega}{\overset{d\circ \pi_1}{\rightrightarrows}}
\fL^{\leq n} \bA^1 dt\big)/\fL^+ \bG_m 
\]

\noindent similarly to $\sZ$; again, the construction
of \S \ref{ss:action-rigorous} applies and provides rigorous
meaning to the $\fL^+ \bG_m$-action in this formula.

\begin{prop}\label{p:zn-algebraic}

The prestack $\sZ^{\leq n}$ is a (DG) algebraic\footnote{We refer to
\cite{dennis-dag} for the definition.} stack.

\end{prop}

\begin{proof}

By definition, we have:
\[
\on{Eq}\big(\fL^+ \bA^1 \times \fL^{\leq n} \bA^1 dt 
\underset{(g,\omega) \mapsto g\omega}{\overset{d\circ \pi_1}{\rightrightarrows}}
\fL^{\leq n} \bA^1 dt\big) = 
\fL^{\leq n} \bA^1 dt \underset{\LocSys_{\bG_m,\log}^{\leq n}}{\times} 
\sZ^{\leq n}.
\]

\noindent Therefore, $\sZ^{\leq n} \to \LocSys_{\bG_m,\log}^{\leq n}$ is
an affine morphism. As $\LocSys_{\bG_m,\log}^{\leq n}$ is an algebraic stack
by Proposition \ref{p:locsys-trun}, we obtain the result.

\end{proof}

\begin{rem}

We remind that $\bB \fL^+ \bG_m$ is not an algebraic stack
while $\bB \bG_m$ is: by definition, algebraic stacks are required
to admit fppf covers, not merely fpqc covers. 

\end{rem}

\subsection{Flatness results}\label{ss:flatness}

We now establish some technical results, showing that certain morphisms
are flat. Ultimately, these results are the technical backbone of
our study of $\sY$ and its coherent sheaves.

\subsubsection{}

We begin with the following result.

\begin{lem}\label{l:flat-mult}

The map:
\[
\begin{gathered}
\mu_n:\bA^n \times \bA^n = \bA^{2n} \to \bA^n \\
(a_0,\ldots,a_{n-1},b_0,\ldots,b_{n-1}) \mapsto 
(a_0b_0,a_0b_1+a_1b_0,\ldots,\sum_{i = 0}^{n-1} a_ib_{n-1-i})
\end{gathered}
\]

\noindent is flat.

\end{lem}

\begin{proof}

This map is defined by a family of homogeneous polynomials
(of degree $2$). Therefore,\footnote{As is standard: there exists
an open $U \subset \bA^n$ such that geometric fibers over
points in $U$ are equidimensional with 
the expected dimension. By homogeneity, $U$
is closed under the $\bG_m$-action. As $U$ contains $0$,
it must be all of $\bA^n$. Then recall that a morphism
of smooth schemes is flat if and only if its geometric fibers are
equidimensional of the expected dimension.}
it suffices to show that the
fiber $Z \coloneqq \mu_n^{-1}(0)$ 
over $0$ is equidimensional with the expected
dimension $2n - n = n$. 

First, for $0 \leq m \leq n$, let $Z_m \subset Z$ be the locally
closed subscheme where $a_0 = a_1 = \ldots = a_{m-1} = 0$ 
and $a_m \neq 0$, the last condition being considered as vacuous
for $m = n$. We claim that $\dim(Z_m) = n$ for all $m$. (We
will also see that $Z_m$ is smooth and connected, so we deduce
that $Z$ has the $n+1$ irreducible components $\overline{Z_m}$.)

Clearly $Z_n = \bA^n$, so 
suppose $m \neq n$. Then:
\[
Z_m \subset (\bA^1 \setminus 0) \times \bA^{n-m-1} \times \bA^n
\]

\noindent is closed, where the coordinates on the latter affine
space are $a_m,a_{m+1},\ldots,a_{n-1},b_0,\ldots,b_{n-1}$;
the equations defining $Z_m$ here are:
\[
\sum_{i=0}^j a_{m+i} b_{j-i} = 0, \ldots j = 0,\ldots,n-m-1.
\]

\noindent In particular, 
$b_j = -\frac{\sum_{i=1}^j a_{m+i} b_{j-i}}{a_0}$ for
$j = 0,\ldots,n-m-1$. It follows
that the morphism:
\[
\begin{gathered}
Z_m \subset 
(\bA^1 \setminus 0) \times \bA^{n-m-1} \times  \bA^{m} \\
(a_m,\ldots,a_{n-1},b_0,\ldots,b_{n-1}) \mapsto 
(a_m,\ldots,a_{n-1},b_{n-m},b_{n-m+1},\ldots,b_{n-1})
\end{gathered}
\]

\noindent is an isomorphism, giving the claim.

\end{proof}

\begin{rem}

The above lemma is standard. See for example 
\cite{goward-smith} Theorem 2.2, especially the remarks 
following its proof.

\end{rem}

We now have the following variant.

\begin{cor}\label{c:flat-d}

Fix a linear map $T:\bA^{2n} \to \bA^n$.
Then $\mu_n+T:\bA^{2n} \to \bA^n$ is flat.

\end{cor}

\begin{proof}

Introduce an auxiliary parameter $\lambda$ and
consider the map $\bA^{2n} \times \bA_{\lambda}^1 \to 
\bA^n \times \bA_{\lambda}^1$ given by
$(\mu_n+\lambda T,\lambda)$. This map is defined by 
homogeneous polynomials (all but one are degree $2$). 
Its fiber over $0$ coincides with $\mu_n^{-1}(0)$, so this
map is flat. 
Restricting to 
$\bA^n \overset{x\mapsto (x,1)}{\subset} \bA^n \times \bA_{\lambda}^1$
gives the claim.

\end{proof}

\subsubsection{}

We now consider variants of the above in which
we pass to a limit.

Let $\bA^{\infty} \in \AffSch$ denote the affine scheme
$\Spec(\Sym(k^{\oplus \bZ^{\geq 0}}))$.  

\begin{cor}\label{c:flat-mult-infty}

The map:
\[
\begin{gathered}
\mu_{\infty}:\bA^{\infty} \times \bA^{\infty} \to \bA^{\infty} 
\\
\big((a_0,a_1,\ldots),(b_0,b_1,\ldots)\big) \mapsto 
(a_0b_0,a_0b_1+a_1b_0,\ldots,\sum_{i = 0}^{j} a_ib_{j-i},\ldots)
\end{gathered}
\]

\noindent is flat.

More generally, suppose we are given linear maps
$T_n:\bA^{2n} \to \bA^n$ fitting into commutative diagrams:
\[
\xymatrix{
\bA^{2n+2} = \bA^{n+1} \times \bA^{n+1} \ar[d] \ar[rr]^(.65){T_{n+1}} & &
\bA^{n+1} \ar[d] \\
\bA^{2n} \ar[rr]^{T_n} & & \bA^n
}
\]

\noindent with vertical maps induced by the projection
$\bA^{n+1} \xar{(a_0,\ldots,a_n) \mapsto (a_0,\ldots,a_{n-1})} \bA^n$.
Let $T$ denote the induced map 
$\bA^{\infty} \times \bA^{\infty} \to \bA^{\infty}$.

Then $\mu_{\infty}+T$ is flat.

\end{cor}

\begin{proof}

As $\mu_{\infty}+T$ is obtained from the morphisms
$\mu_n+T_n$ by passing to the inverse limit in $n$, the
result follows from Lemma \ref{l:flat-mult}.

\end{proof}

\subsubsection{}\label{ss:refinement}

We also need a refinement of the above. The reader may skip this
material and return to it as needed.

Let $C = \Spec(k[x,y]/xy)$. Geometrically, $C$ is a the union of the
$x$ and $y$-axes in the plane. For a scheme $S$ and a morphism 
$\vph:S \to C$, we say that $\vph$ is \emph{flat along the $x$-axis}
if the derived fiber product $S \times_C \bA_x^1$ is a classical
scheme. I.e., for $S$ classical, this means no Tors are formed in
forming this fiber product.

We say $\vph$ is \emph{flat along the $y$-axis} if 
the parallel condition holds for the $y$-axis.
We say $\vph$ is \emph{flat along the axes} if it
is flat along both the $x$ and $y$-axis.

Fix $n>0$ and let $Z = \mu_n^{-1}(0) \subset \bA^{2n}$ as in
the proof of Lemma \ref{l:flat-mult}. 
There
is an evident map:
\[
Z \xar{(a_0,\ldots,a_{n-1},b_0,\ldots,b_{n-1}) \mapsto (a_0,b_0)} C 
\subset \bA^2.
\]

\begin{prop}\label{p:flat-axis}

The above map is flat along the axes.
Moreover, there is a canonical isomorphism:
\[
\mu_{n-1}^{-1}(0) \times \bA^1 \simeq Z \underset{C}{\times} \bA_x^1
\]

\noindent given by:

\[
\big((a_0,a_1,\ldots,a_{n-2},b_0,\ldots,b_{n-2}),\lambda\big) \in 
\mu_{n-1}^{-1}(0) \times \bA^1 \mapsto 
(a_0,a_1,\ldots,a_{n-2},\lambda,0,b_0,b_1,\ldots,b_{n-2}) \in Z.
\]

\end{prop}

\begin{rem}

Recall that in the proof Lemma \ref{l:flat-mult} that we calculated 
the irreducible components of $Z$. In \cite{yuen}, the multiplicities
of these components were also calculated.\footnote{In the notation from 
the proof of Lemma \ref{l:flat-mult}, the multiplicity of
$Z_m$ is $\binom{n}{m}$.}
This latter result follows from Proposition \ref{p:flat-axis}:
given $S \xar{(f,g)} C$ flat along axes, for an irreducible 
component $S_m$ of $S$, $\on{mult}_S(S_m) = 
\on{mult}_{\{f = 0\}}(S_m \cap \{f = 0\}) + \on{mult}_{\{g = 0\}}(S_m \cap \{g = 0\})$;
one can then calculate the multiplicities by induction on $n$. 
We remark that this argument is quite similar to the one given in \emph{loc. cit}.

\end{rem}

We will deduce the above result from the following general lemma.

\begin{lem}\label{l:genl-flat-axes}

Suppose $T$ is a scheme equipped with a map $(f,g):T \to \bA^2$.
Let $S = T \times_{\bA^2} C = \{fg = 0\} \subset T$, where this fiber
product is understood as a derived fiber product.

If $g:T \to \bA^1$ is flat, then $S \to C$ is flat along the $x$-axis.

\end{lem}

\begin{proof}

This is tautological:
\[
S \underset{C}{\times} \bA_x^1 = 
T \underset{\bA^2}{\times} C \underset{C}{\times} \bA_x^1 = 
T \underset{\bA^2}{\times} \bA_x^1 = 
T \underset{\bA_y^1}{\times} 0.
\]

\noindent As $g$ is assumed flat, 
the latter scheme is classical by definition.

\end{proof}

\begin{lem}\label{l:pre-axes}

The morphism:
\[
\begin{gathered}
\mu_n^{\prime}:\bA^n \times \bA^n = \bA^{2n} \to \bA^n \\
(a_0,\ldots,a_{n-1},b_0,\ldots,b_{n-1}) \mapsto 
(b_0,a_0b_1+a_1b_0,\ldots,\sum_{i = 0}^{n-1} a_ib_{n-1-i})
\end{gathered}
\]

\noindent is flat. 

\end{lem}

\begin{rem}

We highlight the (only) difference between $\mu_n^{\prime}$ and
$\mu_n$:
in the former, the first coordinate entry is $b_0$, not $a_0b_0$. 

\end{rem}

\begin{proof}[Proof of Lemma \ref{l:pre-axes}]

As in the proof of Lemma \ref{l:flat-mult}, 
as each coordinate of $\mu_n^{\prime}$ is homogeneous, we
reduce to showing that $(\mu_n^{\prime})^{-1}(0)$ has the expected
dimension $n$. But at the classical level, 
this fiber clearly identifies with 
$\mu_{n-1}^{-1}(0) \times \bA_{b_{n-1}}^1$ (via the map in the statement
of Proposition \ref{p:flat-axis}),
which has dimension $n$ by Lemma \ref{l:flat-mult}.

\end{proof}

\begin{proof}[Proof of Proposition \ref{p:flat-axis}]

Consider the map $\bA^{2n} \xar{\mu_n} \bA^n \to \bA^{n-1}$ where
the last map projects onto the last $n-1$-coordinates.
Let $T$ denote the inverse image of $0$ under this map.

By Lemma \ref{l:pre-axes}, the map $b_0:T \to \bA^1$ is flat.
Therefore, Lemma \ref{l:genl-flat-axes} gives the flatness along
the $x$-axis. Flatness along the $y$-axis obviously
follows by symmetry. 
The resulting description of the fiber product is 
evident. 

\end{proof}

\subsubsection{}

We now have the following generalizations.

\begin{cor}\label{c:flat-axis-d}

Suppose:
\[
T:\bA^{2n} = \bA^{n} \times \bA^{n} \to 0 \times \bA^{n-1} \subset \bA^n
\]

\noindent is a linear map.

Let $Z^q = (\mu_n+T)^{-1}(0)$. Then the natural map 
$Z^q \xar{(a_0,b_0)} C$ is flat along the axes.

Moreover, in the notation of Corollary \ref{c:flat-mult-infty}, we may take 
$n = \infty$ in this result.

\end{cor}

\begin{proof}

For $1\leq n<\infty$, the map:
\[
\mu_n^{\prime}+T:\bA^{2n} \to \bA^n
\]

\noindent is flat by the argument for Lemma \ref{l:pre-axes},
using Corollary \ref{c:flat-d} instead of Lemma \ref{l:flat-mult}.
The proof of Proposition \ref{p:flat-axis} then applies to see that
$Z^q \to C$ is flat along the $x$-axis.

The proof of Corollary \ref{c:flat-mult-infty} allows us to deduce the
$n = \infty$ case.

\end{proof}

\subsubsection{}

We now apply the above results to $\sZ$ and $\sY$.

\begin{prop}\label{p:zn-cl}

For every $n \geq 0$, $\sZ^{\leq n}$ is a
classical algebraic stack.

\end{prop}

\begin{proof}

In coordinates, the morphism:
\[
\mu:\fL^+ \bA^1 \times \fL^{\leq n} \bA^1 dt 
\xar{(g,\omega) \mapsto g\omega - dg}
\fL^{\leq n} \bA^1 dt
\]

\noindent is given by:\footnote{As always, the meaning of this
formula is on $A$-points for any commutative ring $A$. That is,  
the $a_i$ and $b_i$'s are regarded as elements of $A$.} 
\[
\begin{gathered}
(g = \sum_{i=0}^{\infty} a_i t^i, 
\omega = \sum_{i = 0}^{\infty} b_i t^{i-n} dt) \mapsto 
g \omega - dg = 
\sum_{i \geq 0} 
\big(\sum_{j = 0}^i a_j b_{i-j}\big) t^{-n+i} dt 
+ \sum_{i \geq 0} i b_i t^{i-1} dt. 
\end{gathered}
\]

By Corollary \ref{c:flat-d}, this is a flat morphism of affine schemes.
Therefore, the inverse image $\mu^{-1}(0)$ is a classical affine scheme.
As this inverse image is the universal $\fL^+ \bG_m$-torsor
over $\sZ^{\leq n}$, it follows that $\sZ^{\leq n}$ is classical as well.
We now obtain the result from Proposition \ref{p:zn-algebraic}.

\end{proof}

\subsubsection{}\label{ss:z-refinement}

Let $C = \Spec(k[x,y]/xy)$ as in \S \ref{ss:refinement}.
Consider the $\bG_m$-action on $C$ of horizontal 
homothety. I.e., for the corresponding grading on
$k[x,y]/xy$, $\deg(x) = 1$ and $\deg(y) = 0$. 

Fix a coordinate $t$ on the formal disc.
We have a corresponding 
map $\sZ^{\leq n} \to C/\bG_m$ defined as follows.

First, we we have a natural map 
$\sZ \to \fL^+\bA^1/\bG_m \xar{\ev} \bA^1/\bG_m$.
This map takes $(\sL,\nabla,s) \in \sZ$ to $s|_0 \in \sL|_0$ 
for $0 \in \sD$ the base-point.

Next, we have a map:
\[
\sZ^{\leq n} \to \LocSys_{\bG_m,\log}^{\leq n}
\xar{\Pol} \fL^{pol,\leq n} \bA^1 dt
\simeq \prod_{i=1}^n \bA^1 \frac{dt}{t^i} 
\to \bA^1 \frac{dt}{t^n} = \bA^1.
\]

\noindent This map records the leading (i.e.,
degree $-n$) coefficient of the connection. 

The corresponding map 
$\sZ^{\leq n} \to \bA^1 \times \bA^1/\bG_m$ evidently
maps into $C/\bG_m \subset \bA^1 \times \bA^1/\bG_m$.

\begin{rem}

The above is somewhat non-canonical as the second
map above depends
on the choice of coordinate $t$. 
The more canonical statement would be 
to replace the $\bA^1$ by the (scheme corresponding
to the) line $t^{-n} k[[t]]dt/t^{-n+1}k[[t]] dt$.

\end{rem}

The following result plays a key technical role in our work.

\begin{prop}\label{p:z-flat-axes}

Suppose $n>0$.

\begin{enumerate}

\item \label{i:z-axes-1}

The map $\sZ^{\leq n} \to C/\bG_m$ is flat
along the axes.\footnote{By this, we mean that the 
corresponding map $\sZ^{\leq n} \underset{\bB\bG_m}\times \Spec(k) \to C$ is flat along the axes in the
sense of \S \ref{ss:refinement}.}

\item \label{i:z-axes-2}

We have a canonical isomorphism:
\[
\sZ^{\leq n} \underset{C/\bG_m}{\times} \bA^1/\bG_m
\simeq \sZ^{\leq n-1}.
\]

\item \label{i:z-axes-3}

Let $\iota:\sZ^{\leq n} \to \sZ^{\leq n}$ 
denote the map:
\[
(\sL,\nabla,s) \mapsto (\sL(1),\nabla,s)
\]

\noindent 
where $\sL(1)$ is the line bundle on the disc 
whose sections on the disc are allowed
to have a pole of order $1$ at the base-point $0 \in \sD$.\footnote{Another way to say this: 
$\Gr_{\bG_m}$ acts on $\sY_{\log}$, and the corresponding
action of $1 \in \bZ = \Gr_{\bG_m}(k)$ preserves
$\sZ^{\leq n}$ for all $n$; the induced map
$\sZ^{\leq n} \to \sZ^{\leq n}$ is our $\iota$.}

Then $\iota$ fits into a (derived) Cartesian diagram:
\[
\xymatrix{
\sZ^{\leq n} \ar[r]^{\iota} \ar[d] & \sZ^{\leq n} \ar[d] \\
\bA^1 \times \bB \bG_m \ar[r]^{(\id,0)} & C/\bG_m.
}
\]

\end{enumerate}

\end{prop}

\begin{proof}

For \eqref{i:z-axes-1}, it suffices to check
flatness along axes after passing to the fpqc cover
$\sZ^{\leq n} \times_{\bB \fL^+ \bG_m} \Spec(k)$.
The corresponding map to $C/\bG_m$ evidently lifts
to $C$, and it suffices to check that the corresponding
map to $C$ is flat along axes. Using coordinates
as in the proof of 
Proposition \ref{p:zn-cl}, we deduce the claim
from Corollary \ref{c:flat-axis-d}.

The other assertions are immediate from the constructions.
For instance, the diagram in 
\eqref{i:z-axes-3} is obviously \emph{classically}
Cartesian, so derived Cartesian by \eqref{i:z-axes-1}. 

\end{proof}

\begin{cor}\label{c:iota-afp}

The above morphism
$\iota:\sZ^{\leq n} \into \sZ^{\leq n}$ is an
almost finitely presented\footnote{See 
\cite{higheralgebra} Definition 7.2.4.26
for the definition in the affine case. 
In the following remark in \emph{loc. cit}., it is shown
that this notion is preserved under base-change.
By \cite{sag} Proposition 4.1.4.3, this condition
can be checked flat locally. Therefore, there
is an evident notion of a representable
morphism of prestacks being locally almost of finite
presentation, and we are using the term in this
sense.

We emphasize that in this setting,
passing to an affine cover of $\sZ^{\leq n}$,
the corresponding map of classical affine
schemes is a finitely presented morphism in the sense
of classical algebraic geometry,
but being almost finitely presented is a 
\emph{stronger} notion (in spite of the terminology):
in the more classical reference \cite{sga-6} 
Expos\'e III D\'efinition 1.2, this property
of a morphism is called \emph{pseudo-coherence}.}
closed embedding.

\end{cor}

\begin{proof}

Almost finite presentation is preserved under
(derived) base-change, and any morphism of 
schemes almost of finite presentation 
is itself almost of finite presentation.
So the claim follows from 
Proposition \ref{p:z-flat-axes} \eqref{i:z-axes-3}.

\end{proof}

\begin{cor}\label{c:yn-log-cl}

$\sY_{\log}^{\leq n}$ is a classical ind-algebraic stack.
More precisely, the total space of the canonical $\bG_m$-
torsor on 
$\sY_{\log}^{\leq n}$ is classical and a reasonable indscheme 
in the sense of \cite{methods} \S 6.8.

\end{cor}

\begin{proof}

We clearly have:
\begin{equation}\label{eq:y-log-colim-iota}
\colim \big(\sZ^{\leq n} \xar{\iota} \sZ^{\leq n} \xar{\iota } \ldots \big)
= \sY_{\log}^{\leq n}. 
\end{equation}

\noindent Each of these morphisms is an
almost finitely presented closed
embedding. 
Therefore, $\sY^{\leq n} \times_{\bB \bG_m} \Spec(k)$
is a filtered colimit of the classical affine schemes 
$\sZ^{\leq n} \times_{\bB \bG_m} \Spec(k)$ under 
almost finitely presented closed
embeddings. 

We now remind the general definition of reasonable
indscheme from \cite{methods}: 
it is a (DG) indscheme that can be written
as a colimit of eventually coconnective
quasi-compact quasi-separated schemes
under closed embeddings almost of finite presentation.
So clearly this property is verified here.

\end{proof}

\begin{rem}

One clearly obtains similar results for the Higgs analogue
of $\sY$. A posteriori, one sees that 
$\Maps(\o{\sD},C)$ is a reasonable ind-affine indscheme.
For a weaker (classical) notion of reasonable indscheme, a similar
result with $C$ replaced by any finite type affine scheme
is well-known: see \cite{chiral} Lemma 2.4.8. It is natural to ask: in what
generality is such a result true for the stronger
(derived) notion used here?

\end{rem}

\begin{proof}[Proof of Theorem \ref{t:cl}]

By Corollary \ref{c:yn-log-cl}, 
$\sY^{\leq n} = \sY_{\log}^{\leq n}/\Gr_{\bG_m}^{\leq n}$
is a classical (non-algebraic) prestack. 
We deduce the same for $\sY = \colim_n \sY^{\leq n}$.

\end{proof}

\subsubsection{}\label{ss:c-redux}

We need one mild improvement of Proposition \ref{p:z-flat-axes}.
Roughly, the statement says that the use of the base-point $0 \in \widehat{\sD}$
is not essential in \emph{loc. cit}.

First, we have a map:
\begin{equation}\label{eq:axis-1}
\widehat{\sD} \times \sZ \to \widehat{\sD} \times \fL^+\bA^1/\bG_m \to 
\bA^1/\bG_m
\end{equation}

\noindent where the latter map is evaluation. Explicitly, this map 
sends $(\tau,(\sL,\nabla,s)) \in \widehat{\sD} \times \sZ$ to 
$s|_{\tau} \in \sL|_{\tau}$.

Now fix $n$. We let $\sD^{\leq n} \coloneqq \Spec(k[[t]]/t^n)$. 

We have a second map:
\begin{equation}\label{eq:axis-2}
\sD^{\leq n} \times \sZ^{\leq n} \to \bA^1/\bG_m
\end{equation}

\noindent defined as follows. By definition, this map will
factor as:
\[
\sD^{\leq n} \times \sZ^{\leq n} \to
\sD^{\leq n} \times \LocSys_{\bG_m,\log}^{\leq n} \xar{\id \times \Pol}
\sD^{\leq n} \times \fL^{pol,\leq n} \bA^1 dt \to 
\bA^1/\bG_m 
\]

\noindent where the last map remains to be defined. 
First, note that the dualizing complex
$\omega_{\sD^{\leq n}} \in \IndCoh(\sD^{\leq n})$ lies in degree $0$
and is a line bundle; it corresponds to the module $(k[[t]]/t^n)^{\vee} = 
t^{-n} k[[t]]dt/k[[t]]dt$. Therefore, given $\tau \in \sD^{\leq n}$, 
we obtain a line $\tau^*(\omega_{\sD^{\leq n}})$. Moreover,
a point $\omega \in \fL^{pol,\leq n} \bA^1 dt$ defines an evident section
of the above line bundle $\omega_{\sD^{\leq n}}$. Restricting to the
point $\tau$, we obtain the desired construction.

\begin{rem}

A choice of coordinate $t$ trivializes the above line bundle
on $\sD^{\leq n}$: the basis element is $\frac{dt}{t^n}$.
Such a trivialization lifts the map \eqref{eq:axis-2} to 
a map to $\bA^1$. This map is explicitly given at 
$(\tau,(\sL,\nabla,s))$ by evaluating the function
$\Pol(\nabla) \cdot \frac{t^n}{dt}$ on $\sD^{\leq n}$
at $\tau$.

\end{rem}

\begin{prop}\label{p:z-flat-axes-var}

The map
$\sZ^{\leq n} \to \bA^1/\bG_m \times \bA^1/\bG_m$ factors
through $C/(\bG_m\times\bG_m)$. For $n>0$, 
the resulting map 
$\sZ^{\leq n} \to C/(\bG_m\times\bG_m)$ is flat along
axes. 

\end{prop}

\begin{proof}

For the first assertion, observe that
for $(\sL,\nabla,s) \in \sZ$, 
$s\Pol(\nabla) = 0$ as a polar section
of the line bundle $\sL dt$ on the disc. 
This implies the claim.

For flatness, note that we 
have a canonical evaluation morphism:
\[
\widetilde{\ev}:
\sD^{\leq n} \times \fL_n^+ C \to C.
\]

\noindent It is immediate from 
Proposition \ref{p:flat-axis} that this map 
is flat along axes. Indeed, we need to show
that $\widetilde{\ev}^*(i_*(\sO_{\bA_x^1}))$
is concentrated in cohomological degree $0$
(where $i:\bA_x^1 \to C$ is the embedding).
It suffices to check this after further restriction
to $\fL_n^+ C$, where the assertion is 
exactly \emph{loc. cit}.

From here, the claim proceeds as in the proof of 
Proposition \ref{p:z-flat-axes}.

\end{proof}

\subsection{Some notation}\label{ss:y-not}

We now collect a bit of notation related
to $\sY$.

\subsubsection{}\label{sss:maps-notation}

We begin by introducing notation for various structural maps. 

For $n>0$, recall that we have the morphism:
\[
\iota:\sZ^{\leq n} \to \sZ^{\leq n}.
\]

\noindent For $r \geq 0$, we sometimes let $\iota^r$ denote the $r$-fold
composition of $\iota$.

We let:
\[
\begin{gathered}
\delta_n:\sZ^{\leq n-1} \to \sZ^{\leq n} \\
i_n:\sZ^{\leq n} \to \sY_{\log}^{\leq n} \\
\pi_n:\sZ^{\leq n} \to \sY^{\leq n} \\
\lambda_n:\sY_n \to \sY
\end{gathered}
\]

\noindent denote the canonical maps. 

We also have the maps $\zeta$ and $\widetilde{\zeta}$ fitting into 
a diagram:
\[
\xymatrix{
\LocSys_{\bG_m,\log}^{\leq n} 
\ar[rr]^{\widetilde{\zeta}} \ar[d]^{\widetilde{\pi}_n} 
&& \sZ^{\leq n} 
\ar[d]^{\pi_n} \\
\LocSys_{\bG_m}^{\leq n} \ar[rr]^{\zeta} && \sY^{\leq n}.
}
\]

\noindent Here the horizontal arrows send
$(\sL,\nabla)$ to $(\sL,\nabla,0)$; i.e., we
take $0$ as the flat section of our local system.

\subsubsection{}\label{sss:z-pic}

We introduce the following line bundle on $\sZ^{\leq n}$.

Using our base-point $0 \in \sD$, we obtain a canonical line bundle
$\sO_{\sZ^{\leq n}}(-1)$ on $\sZ^{\leq n}$; its fiber at
$(\sL,\nabla,s)$ is the fiber $\sL|_0$ of $\sL$ at $0$.

For $r \in \bZ$, we let:
\[
\sO_{\sZ^{\leq n}}(r) \coloneqq \sO_{\sZ^{\leq n}}(-1)^{\otimes -r}.
\]

\noindent denote its (suitably normalized) tensor powers. 
For $\sF \in \QCoh(\sZ^{\leq n})$, we let
$\sF(r)$ denote its tensor with $\sO_{\sZ^{\leq n}}(r)$.

We let $\widetilde{\sZ}^{\leq n}$ denote the total space of the $\bG_m$-bundle
defined by $\sO_{\sZ^{\leq n}}(-1)$, i.e., 
$\widetilde{\sZ}^{\leq n} \coloneqq \sZ^{\leq n} \times_{\bB \bG_m} \Spec(k)$.
By Proposition \ref{p:zn-algebraic} (or its proof), we remark that
$\widetilde{\sZ}^{\leq n}$ is an affine scheme, and by Theorem \ref{t:cl},
it is a classical affine scheme. We use notation $\iota$ and $\delta_n$
sometimes for the corresponding maps for $\widetilde{\sZ}^{\leq n}$.

\begin{rem}

Clearly the above line bundle extends canonically to $\sY_{\log}$ and
$\sY_{\log}^{\leq n}$; we use similar notation in those settings.

\end{rem}

\subsection{Fundamental exact sequences}\label{ss:fund}

\subsubsection{}

We now record two exact sequences that we will use for some inductive arguments. 
The results here amount to restatements of Proposition \ref{p:z-flat-axes}.
Suppose $n > 0$ in what follows.

\subsubsection{}

We will presently construct a short exact sequence:
\begin{equation}\label{eq:fund-ses}
0 \to \delta_{n,*}\sO_{\sZ^{\leq n-1}}(1) 
\to \sO_{\sZ^{\leq n}} \to \iota_*\sO_{\sZ^{\leq n}} \to 0
\end{equation}

\noindent in $\QCoh(\sZ^{\leq n})^{\heart}$. 

The map $\sO_{\sZ^{\leq n}} \to \iota_*\sO_{\sZ^{\leq n}}$ is just the
canonical adjunction morphism.

Now observe that there is a canonical map:
\begin{equation}\label{eq:o(1)}
\sO_{\sZ^{\leq n}}(1) \to \sO_{\sZ^{\leq n}};
\end{equation}

\noindent its fiber at $(\sL,\nabla,s)$ is the map:
\[
\sL^{\vee}|_0 \to \sO|_0
\]

\noindent given by pairing with the section $s|_0 \in \sL|_0$. 
We claim that this map factors uniquely as:

\[
\sO_{\sZ^{\leq n}}(1) \onto \delta_{n,*}\sO_{\sZ^{\leq n-1}}(1) 
\dashrightarrow \sO_{\sZ^{\leq n}}
\]

\noindent and that the resulting map defines a short exact sequence 
as in \eqref{eq:fund-ses}.

In fact, this is implicit in work we have done already: these claims are 
obtained by pullback from the short exact sequence:
\[
0 \to k[x,y]/(xy,x) \xar{x \cdot -} k[x,y]/xy \to k[x,y]/(xy,y) \to 0
\]

\noindent of bi-graded $k[x,y]/xy$ modules along the morphism
from Proposition \ref{p:z-flat-axes}. We emphasize that we are using the 
flatness asserted in Proposition \ref{p:z-flat-axes} \eqref{i:z-axes-1}.

\begin{variant}

For later use, we record the following observation. Fix 
a coordinate $t$ on the disc and use the notation of \S \ref{ss:coords} below. 

Then we also have a short exact sequence:
\begin{equation}\label{eq:fund-ses-2}
0 \to \iota_*\sO_{\sZ^{\leq n}} \to \sO_{\sZ^{\leq n}} \to 
\delta_{n,*}\sO_{\sZ^{\leq n-1}} \to 0 
\end{equation}

\noindent in $\QCoh(\sZ^{\leq n})^{\heart}$ in which the right
map is the canonical morphism. The left arrow is the unique map fitting
into a commutative diagram:
\[
\xymatrix{
\sO_{\sZ^{\leq n}} \ar[d] \ar[dr]^{b_{-n} \cdot -} \\
\iota_*\sO_{\sZ^{\leq n}} \ar@{..>}[r] & \sO_{\sZ^{\leq n}}.
}
\]

\noindent Here $b_{-n}$ is as in \S \ref{ss:coords} below. 
Again, the existence of the dotted arrow and the short exact sequence
follow from Proposition \ref{p:flat-axis}.

\end{variant}

\subsection{Nice opens}\label{ss:un}

At some points, it is convenient to refer to the following geometric observations
about $\sZ^{\leq n}$. Roughly speaking, the idea is that 
$\sZ^{\leq n}$ is ``nice" away from $\sZ^{\leq n-1}$, which we sometimes use
for inductive statements. 

We proceed separately in the cases where $n = 1$ and $n > 1$. 

%% make sure references go here as needed.

\subsubsection{$n>1$ case}

Define $\sU_n \coloneqq \sZ^{\leq n}$ as the open substack:
\[
\sU_n \coloneqq \sZ^{\leq n} \setminus \sZ^{\leq n-1}.
\]

\begin{lem}\label{l:un}

The natural map:
\[
\LocSys_{\bG_m,\log}^{\leq n} \setminus \LocSys_{\bG_m,\log}^{\leq n-1}
\to 
\sU_n
\]

\noindent (induced by $\widetilde{\zeta}$ from \S \ref{sss:maps-notation})
is an isomorphism. 

\end{lem}

\begin{rem}

By Proposition \ref{p:locsys-log}, we deduce that $\sU_n \simeq 
(\bA^1 \setminus 0 ) \times \bA^{n-1} \times \bB \bG_m$. In particular,
$\sU_n$ is a smooth stack of finite type. 

\end{rem}

\begin{proof} 

It is convenient to use the notation of \S \ref{ss:coords}, which is introduced
below.
In that notation, we have $\sU_n = \Spec(A_n[b_{-n}^{-1}])/\bG_m$.

In the (classical) commutative ring $A_n[b_{-n}^{-1}]$, we have relations:
\[
\begin{gathered}
0 = b_{-n} a_0 \Rightarrow a_0 = 0 \\
0 = b_{-n} a_1 + b_{-n+1} a_0 = b_{-n} a_1 \Rightarrow a_1 = 0 \\
\ldots \Rightarrow a_i = 0 \text{ for all i}
\end{gathered}
\]

\noindent (using $n>1$).
It follows that we have an isomorphism:
\[
A_n[b_{-n}^{-1}] = k[b_{-1},\ldots,b_{-n},b_{-n}^{-1}].
\]

\noindent This amounts to the claim.  

\end{proof}

\subsubsection{$n = 1$ case} 

This case is slightly more technical.

We define $\sU_1$ as the pro-(Zariski open substack of $\sZ^{\leq 1}$):
\[
\sU_1 \coloneqq \underset{r}{\lim} \, {\sZ^{\leq 1} \setminus  
\big(\sZ^{\leq 0} \cup \ldots \cup \iota^r(\sZ^{\leq 0})\big)}.
\] 

\noindent Here each of the structural maps in the limit is affine, 
so the limit exists and is well-behaved. 

\begin{lem}\label{l:u1}

The natural map:
\[
(\bA^1 \setminus \bZ^{\geq 0}) \times \bB \bG_m \to 
\LocSys_{\bG_m,\log}^{\leq 1} \xar{\widetilde{\zeta}} 
\sZ^{\leq 1}
\]

\noindent maps through $\sU_1$ and induces an isomorphism:

\[
(\bA^1 \setminus \bZ^{\geq 0}) \times \bB \bG_m \isom \sU_1.
\]

\end{lem}

\begin{proof}

We again use the coordinates of \S \ref{ss:coords}.
In this notation, $\sU_1$ the quotient stack:
\[
\Spec(A_1[b_{-1}^{-1},(b_{-1}-1)^{-1},(b_{-1}-2)^{-1},\ldots])/\bG_m
\]

\noindent obtained by inverting the elements $\{b_{-1}-i\}$ for each $i \geq 0$.
In this ring, we have the relations:
\[
(b_{-1}-i)a_i = 0
\]

\noindent for each $i \geq 0$, which means that $a_i = 0$ in this localization
for each $i$. The claim then follows. 

\end{proof}

\subsubsection{}\label{sss:u-reg-noeth}

In both cases above, we observe that $\widetilde{\sU}_n \coloneqq 
\widetilde{\sZ}^{\leq n} \times_{\sZ^{\leq n}} \sU_n$ is a regular Noetherian 
affine scheme. 

\subsubsection{A remark}

In spite of the above observations, we warn that the natural map:
\[
\zeta: \LocSys_{\bG_m} \setminus \LocSys_{\bG_m}^{\leq 0} \to 
\sY|_{\LocSys_{\bG_m} \setminus \LocSys_{\bG_m}^{\leq 0}}
\]

\noindent is \emph{not} an isomorphism. However, it is an isomorphism 
if one truncates to irregularity of order $n \leq 1$, 
or if one only evaluates on Noetherian 
test rings (the left hand side is locally of finite type, 
but the right hand side is not). 

\subsection{Coordinates}\label{ss:coords}

\subsubsection{}

To make the above completely explicit, we provide explicit 
coordinates.

\subsubsection{} 

Define a classical commutative ring $A_n$ by taking
(infinitely many) generators 
$a_0,a_1,a_2,\ldots$ and $b_{-1},\ldots,b_{-n}$
and (infinitely) relations coming from equating Taylor series coefficients
of the formal Laurent series:
\[
\sum_{i=0}^{\infty} ia_it^{i-1} =
\sum_{i=0}^{\infty} a_it^i \cdot 
\sum_{i=-n}^{-1} b_{-i}t^{-i}.
\]

We consider $A_n$ as graded with $\deg(a_i) = 1$ 
for all $i$ and $\deg(b_i) = 0$.
This grading on $A_n$ corresponds geometrically 
to a $\bG_m$-action on $\Spec(A_n)$.

Define a map $\Spec(A_n) \to \sZ^{\leq n}$ 
by sending a point $(a_0,a_1,a_2,\ldots,b_{-1},\ldots,b_{-n}) 
\in \Spec(A_n)$ to the trivial line bundle
$\sO$ on the disc and equipping it with the connection
$\nabla = d-\sum_{i = 1}^n b_{-i} t^{-i} dt$ on the punctured
disc and the section $s = \sum_{i \geq 0} a_i t^i$,
which is annihilated by the connection by design.

This map actually factors through $\Spec(A_n)/\bG_m$; the
latter parametrizes the data of a line $\ell$ and
points as above, except that the $a_i$ are sections
of $\ell$. We then take our line bundle to be
$\sO \otimes \ell$, take $\nabla$ given by the same formula,
and similarly for our section $s$.

From Propositions \ref{p:locsys-log} and \ref{p:zn-cl},
we obtain the next result.

\begin{prop}\label{p:zn-coords}

The above map
$\Spec(A_n)/\bG_m \to \sZ^{\leq n}$ is an isomorphism.
I.e., we have a $\bG_m$-equivariant isomorphism
$\Spec(A_n) \simeq \widetilde{\sZ}^{\leq n}$.

\end{prop}

\subsection{Regular singular sketches}\label{ss:rs-sketch}

\subsubsection{}

For the sake of explicitness, we draw pictures of $\sZ^{\leq 1}$,
$\sY_{\log}^{\leq 1}$, and $\sY^{\leq 1}$ using
the presentation of \S \ref{ss:coords}. 

\subsubsection{}

By the above, $A_1$ is the (classical) algebra with generators
$b_{-1},a_0,a_1,\ldots$ and relations:
\[
(b_{-1}-i)a_i = 0.
\] 

\noindent For an integer $m \geq 0$, let $A_{1,\leq m}$ denote the
subalgebra generated by $b_{-1},a_0,\ldots,a_m$, and let
$\widetilde{\sZ}_{\leq m}^{\leq 1} \coloneqq \Spec(A_{1,\leq m})$.
We obtain the following picture for $\widetilde{\sZ}_{\leq m}^{\leq 1}$:
\[
\begin{tikzpicture}
\draw[gray, thick] (-3,0) -- (3,0);
\draw[gray, thick] (0,-1) -- (0,1);
\draw[gray, thick] (.5,-1) -- (.5,1);
\draw[gray, thick] (1,-1) -- (1,1);
\draw[gray, thick] (1.5,-1) -- (1.5,1);
\node at (-4, 0)   (a) {$\widetilde{\sZ}_{\leq 3}^{\leq 1} = $} ;
\end{tikzpicture}
\]

\noindent where the horizontal axis is the $b_{-1}$-axis, and the fiber
over $b_{-1} = i$ is given the coordinate $a_i$.

The structural maps $\widetilde{\sZ}_{\leq m+1}^{\leq 1} \to 
\widetilde{\sZ}_{\leq m}^{\leq 1}$ coming from the 
embedding $A_{1,\leq m} \into A_{1,\leq m+1}$ correspond to contracting 
the rightmost line in the picture. Therefore, as:
\[
\widetilde{\sZ}^{\leq 1} = \underset{m}{\lim} \, \widetilde{\sZ}_{\leq m}^{\leq 1}
\]

\noindent we obtain the picture:
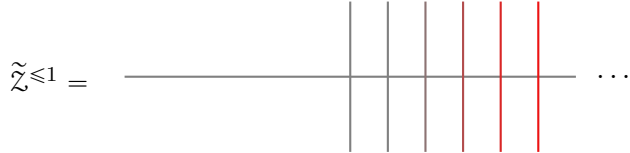
\[
\begin{tikzpicture}
\draw[gray, thick] (-3,0) -- (3,0);
\draw[gray, thick] (0,-1) -- (0,1);
\draw[gray, thick] (.5,-1) -- (.5,1);
\draw[gray!90!red, thick] (1,-1) -- (1,1);
\draw[gray!60!red, thick] (1.5,-1) -- (1.5,1);
\draw[gray!30!red, thick] (2,-1) -- (2,1);
\draw[gray!15!red, thick] (2.5,-1) -- (2.5,1);
\node at (-4, 0)   (a) {$\widetilde{\sZ}^{\leq 1} = $} ;
\node at (3.5, 0)   (b) {$\ldots$} ;
\end{tikzpicture}
\]

\noindent where the reddening indicates that interpret the
picture in the \emph{pro}-sense rather than the \emph{ind}-sense.

\begin{rem}

For $n > 1$, the stack $\sZ^{\leq n}$ is non-reduced, so there
is additional complexity in attempting to draw pictures.
(In addition, its Krull dimension grows with $n$.)

\end{rem}

\subsubsection{}

The map $\iota:\widetilde{\sZ}^{\leq 1} \to \widetilde{\sZ}^{\leq 1}$ 
corresponds to a rightward
shift in the above picture. Therefore,
if we let $\widetilde{\sY}_{\log}^{\leq 1}$ similarly be the total space
of the $\bG_m$-torsor over $\sY_{\log}^{\leq 1}$, 
from \eqref{eq:y-log-colim-iota} we obtain the picture:
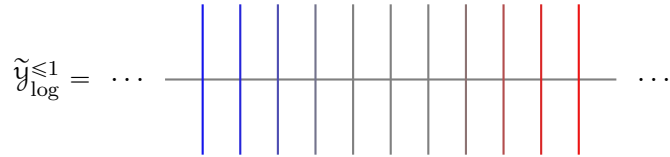
\[
\begin{tikzpicture}
\draw[gray, thick] (-3,0) -- (3,0);
\draw[gray, thick] (0,-1) -- (0,1);
\draw[gray, thick] (.5,-1) -- (.5,1);
\draw[gray!90!red, thick] (1,-1) -- (1,1);
\draw[gray!60!red, thick] (1.5,-1) -- (1.5,1);
\draw[gray!30!red, thick] (2,-1) -- (2,1);
\draw[gray!15!red, thick] (2.5,-1) -- (2.5,1);
\draw[gray, thick] (-.5,-1) -- (-.5,1);
\draw[gray!90!blue, thick] (-1,-1) -- (-1,1);
\draw[gray!60!blue, thick] (-1.5,-1) -- (-1.5,1);
\draw[gray!30!blue, thick] (-2,-1) -- (-2,1);
\draw[gray!15!blue, thick] (-2.5,-1) -- (-2.5,1);
\node at (-4.5, 0)   (a) {$\widetilde{\sY}_{\log}^{\leq 1} = $} ;
\node at (3.5, 0)   (b) {$\ldots$} ;
\node at (-3.5, 0)   (b) {$\ldots$} ;
\end{tikzpicture}
\]

\noindent where the bluing in the picture indicates that we interpret
the limit in the \emph{ind}-sense (and the reddening is as before).

Informally, $\widetilde{\sY}_{\log}^{\leq 1}$ is a \emph{semi-infinite comb}. It has
infinitely many bristles, which are attached to the handle at integer
points; in the positive direction, these bristles have
pro-nature, and in the negative direction they have ind-nature.\footnote{We 
are unfortunately unable to produce non-synesthetic
pictures adequately distinguishing between projective and inductive limits. 
If asked to better explain the above pictures, 
we would resort to the formal descriptions
$\sZ^{\leq 1} = \lim_m \sZ_{\leq m}^{\leq 1}$ and $\sY_{\log}^{\leq 1} = 
\colim_{\iota} \sZ^{\leq 1}$.}

\subsubsection{}

In each of these pictures, $\bG_m$ acts by scaling in the vertical direction, i.e.,
scaling the bristles. So forming stacky quotients for this action, 
we obtain the promised sketches of 
$\sZ^{\leq 1}$ and $\sY_{\log}^{\leq 1}$.

Finally, the action of $\bZ = \Gr_{\bG_m}^{\leq 1}$ on $\sY_{\log}$ is generated
by rightward shift in the above pictures. Quotienting by this action gives a 
picture\footnote{Roughly speaking, this picture looks like $\bA^1/\bZ$ with a 
single bristle attached at $\bZ/\bZ$. But we must not forget the semi-infinite nature
of that bristle, which we find difficult to visually express in this setting.} 
for $\sY^{\leq 1}$.

\section{Local Abel-Jacobi morphisms}\label{s:aj}

\subsection{}

In this section, we collect some standard results
about local Abel-Jacobi maps and (simplified) 
Contou-Carr\`ere pairings for use in \S \ref{s:weyl}.
This material is standard and included for the reader's
convenience.

\subsection{}

Let $\widehat{\sD} \in \IndSch$ denote the formal disc
$\Spf(k[[t]]) \coloneqq \colim_n \Spec(k[[t]]/t^n)$.

Let $\widehat{\Delta}:\widehat{\sD} \to \widehat{\sD} \times \sD$ denote the
graph of the canonical map $\widehat{\sD} \to \sD$. 
Clearly this map is a closed embedding (in particular, schematic).

We define: 
\[
\sO(-\widehat{\Delta}) \coloneqq \Ker\big(
\sO_{\widehat{\sD} \times \sD} \to \widehat{\Delta}_*(\sO_{\widehat{\sD}})\big) 
\in \QCoh(\widehat{\sD} \times \sD).
\]

\noindent One readily checks 
that $\sO(-\widehat{\Delta})$ is a line bundle on 
$\widehat{\sD} \times \sD$: in fact, 
a choice of coordinate on the disc 
gives a trivialization of it. 
Moreover, as:
\[
\widehat{\sD}
\underset{\widehat{\sD} \times \sD }{\times}
(\widehat{\sD} \times \o{\sD} ) = \emptyset
\]

\noindent there is a canonical trivialization
of 
$\sO(-\widehat{\Delta})|_{\widehat{\sD} \times \o{\sD}}$. 

Thus, we obtain a map:
\[
\AJ^{-1}:\widehat{\sD} \to \Gr_{\bG_m}.
\] 

\noindent We 
define the \emph{local Abel-Jacobi map} 
$\AJ$ as $\AJ^{-1}$ composed
with the inversion map $\Gr_{\bG_m} \to \Gr_{\bG_m}$.

Explicitly, we have:
\[
\begin{gathered}
\AJ: \widehat{\sD} \to \Gr_{\bG_m} \\
\tau \mapsto (\sO_{\sD}(\tau),1)
\end{gathered}
\]

\noindent where $\sO_{\sD}(\tau)$ is the line bundle on 
the disc with sections having at worst simple
poles at $\tau \in \widehat{\sD} \subset \sD$
and the evident trivialization ``$1$" on $\o{\sD}$.

\subsection{}\label{ss:dlog-aj}

We now consider a closely related map to the Abel-Jacobi map.

We have an evident evaluation map:
\begin{equation}\label{eq:ev-a1}
\widehat{\sD} \times \fL^+ \bA^1 \to \bA^1
\end{equation}

\noindent that is linear in the second coordinate.
Using the non-degeneracy of the residue pairing, 
we deduce that there is a unique map:
\begin{equation}\label{eq:dlog-aj}
\widehat{\sD} \to
\fL^{pol} \bA^1 dt.
\end{equation}

\noindent such that the composition:
\[
\widehat{\sD} \times \fL^+ \bA^1 
\xar{\eqref{eq:dlog-aj} \times \id}
\fL^{pol} \bA^1 dt \times 
\fL^+ \bA^1 \xar{(f,\omega) \mapsto \Res(f\omega)} 
\bA^1
\]

\noindent recovers the evaluation map above.

\begin{rem}\label{r:dlog-aj-fmla}

Explicitly, \eqref{eq:dlog-aj} is given in coordinates
by:
\[
(\tau \in \widehat{\sD}) \mapsto \sum_{i = 0}^{\infty}
\frac{\tau^i}{t^{i+1}} dt.
\]

\end{rem}

We have the following elementary calculation.

\begin{lem}\label{l:dlog-aj}

The map \eqref{eq:dlog-aj} coincides with the composition:
\[
\widehat{\sD} \xar{\AJ} \Gr_{\bG_m} 
\xar{-d\log} \fL^{pol} \bA^1 dt.
\]

\end{lem}

\begin{proof}

Choose coordinates as in Remark \ref{r:dlog-aj-fmla}.
The map $\AJ:\widehat{\sD} \to \Gr_{\bG_m}$ lifts to a 
map to $\fL\bG_m$ via $\tau \mapsto \frac{1}{t-\tau}$.
We then have:
\[
d\log \frac{1}{t-\tau} = 
\frac{d\frac{1}{t-\tau}}{\frac{1}{t-\tau}} = 
-\frac{dt}{t-\tau} = -\frac{1}{t}\sum_{i = 0}^{\infty} \frac{\tau^i}{t^i} dt
\]

\noindent Comparing to Remark \ref{r:dlog-aj-fmla}, we obtain the 
claim.

\end{proof}

\begin{rem}

Following the lemma, we denote the above map by 
$-d\log \AJ$, and its additive inverse by $d\log \AJ$.

\end{rem}

\subsection{Compatibility with truncations}\label{ss:aj-trun}

We use the following observation. We remind that
$\sD^{\leq n} \coloneqq \Spec(k[[t]]/t^n)$.

Clearly $d\log \AJ$ maps $\sD^{\leq n}$ maps 
into $\fL^{pol,\leq n} \bA^1 dt$. 
Therefore, by definition (see \S \ref{ss:loop-gr-trun}),
we deduce that $\AJ$ maps $\sD^{\leq n}$ into 
$\Gr_{\bG_m}^{\leq n}$.

\subsection{The positive Grassmannian}\label{ss:pos-gr}

Next, we define: 
\[
\fL^{pos} \bG_m \coloneqq
\fL \bG_m \underset{\fL \bA^1} {\times} \fL^+ \bA^1.
\]

\noindent There is an evident commutative monoid
structure on $\fL^{pos} \bG_m$ and homomorphism
$\fL^+ \bG_m \to \fL^{pos} \bG_m$.

We then define:
\[
\Gr_{\bG_m}^{pos} \coloneqq \fL^{pos} \bG_m/\fL^+ \bG_m =
\Ker(\fL^+\bA^1/\bG_m \to \fL \bA^1/\bG_m).
\]

Explicitly, $\Gr_{\bG_m}^{pos}$ parametrizes
the data of $\sL$ a line bundle on the disc and
$\sigma \in \Gamma(\sD,\sL)$ a (regular) section
such that $\sigma|_{\o{\sD}}$ trivializes $\sL$.

\subsection{}

For $n \geq 0$, we define: 
\[
\Sym^n \widehat{\sD} \coloneqq \Spf(k[[t_1,\ldots,t_n]])^{S_n} 
\in \IndSch.
\]

\noindent for $S_n$ the symmetric group.
We then let: 
\[
\Sym\widehat{\sD} \coloneqq 
\coprod_n \Sym^n \widehat{\sD} \in \IndSch.
\]

\subsection{}

As for usual smooth curves, $\Sym\widehat{\sD}$ parametrizes
effective Cartier divisors on $\sD$ supported on $\widehat{\sD}$.
Hence, we obtain an isomorphism:
\[
\Sym\widehat{\sD} \isom \Gr_{\bG_m}^{pos}
\]

\noindent such that the composition:
\[
\widehat{\sD} = \Sym^1 \widehat{\sD} \to 
\Sym\widehat{\sD} \simeq \Gr_{\bG_m}^{pos}
\subset \Gr_{\bG_m}
\]

\noindent is $\AJ$.

This is an isomorphism of commutative monoids for the
evident product on $\Sym(\widehat{\sD})$.

\subsection{Log Contou-Carr\'ere}

\subsubsection{}

We now observe the following.

There is a unique pairing:
\[
\pair{-}{-}:\Gr_{\bG_m}^{pos} \times \fL^+ \bA^1 \to \bA^1
\]

\noindent that is bilinear for multiplicative monoid
structures on both sides and whose restriction 
to $\widehat{\sD} \times \fL^+ \bA^1$ is 
\eqref{eq:ev-a1}. 

Explicitly, for $D \in \Gr_{\bG_m}^{pos}$ an effective 
Cartier divisor
on $\sD$ supported on $\widehat{\sD}$ and 
$f\in \fL^+ \bA^1$, we have:
\[
\pair{D}{f} = f(D) \coloneqq \Nm_D(f).
\]

\subsubsection{}

There are some variants of the pairing above.

First, note that 
$\pair{-}{-}|_{\Gr_{\bG_m}^{pos} \times \fL^+ \bG_m}$
maps into $\bG_m$. As $\Gr_{\bG_m}$ is the group completion
of $\Gr_{\bG_m}^{pos}$ (in the evident sense), we obtain a 
pairing:
\[
\pair{-}{-}:\Gr_{\bG_m} \times \fL^+ \bG_m \to 
\bG_m.
\]

\noindent This pairing is compatible with Contou-Carr\'ere's
pairing \cite{cc} in the evident sense.

\subsubsection{}

Next, let $\Sym^m \sD^{\leq n} \coloneqq
\Spec((k[[t]]/t^n)^{\otimes m,S_m})$, which is an
affine scheme. We let $\Gr_{\bG_m}^{pos,\leq n}\coloneqq 
\coprod_m \Sym^m \sD^{\leq n}$.

The evident map:
\[
\Gr_{\bG_m}^{pos,\leq n} \to \coprod_m \Sym^m \widehat{\sD} \to \Gr_{\bG_m}
\]

\noindent maps into $\Gr_{\bG_m}^{\leq n}$. Indeed,
this map is a map of monoids, so the same is
true of its composition with $d\log:\Gr_{\bG_m} \to 
\fL^{pol} \bA^1 dt$. As $\sD^{\leq n}$ maps into
$\fL^{pol,\leq n} \bA^1 dt$ (cf. \S \ref{ss:aj-trun}),
we obtain the claim by definition of $\Gr_{\bG_m}^{\leq n}$.

For similar reasons, the pairing:
\[
\pair{-}{-}:
\Gr_{\bG_m}^{pos,\leq n} \times \fL^+ \bA^1 \to \bA^1
\]

\noindent factors through a pairing:
\[
\Gr_{\bG_m}^{pos,\leq n} \times \fL_n^+ \bA^1 \to \bA^1 
\]

\noindent that we also denote by $\pair{-}{-}$.

\begin{rem}\label{r:gr-n-pos-gp-cpl}

The base-point $0 \in \widehat{\sD}$ gives a point of 
$\Gr_{\bG_m}^{pos,\leq n}$ (of degree 1), hence 
a translation map $T:\Gr_{\bG_m}^{pos,\leq n} \to \Gr_{\bG_m}^{pos,\leq n}$.
There is a natural map:
\[
\colim \, (\Gr_{\bG_m}^{pos,\leq n} \xar{T} \Gr_{\bG_m}^{pos,\leq n} \xar{T} \ldots)
\to \Gr_{\bG_m}^{\leq n}
\]

\noindent that we claim is an isomorphism. 
Indeed, the assertion is evident at the level of $k$-points (both sides 
are $\bZ$), so it suffices to check the assertion on tangent complexes,
where it is straightforward (both sides naturally identify with 
$t^{-n+1}k[[t]]/k[[t]]$). 

\end{rem}

\subsubsection{}\label{ss:pairing-map}

We will use the following constructions (in a quite simple setting).

Let $V$ be a vector space. We consider $V$ as a prestack\footnote{If $Y$ is
coconnective, this prestack is a (DG) indscheme.}
with functor of points $\Spec(A) \mapsto \Omega^{\infty}(A\otimes V) \in \Gpd$.

Now suppose $X$ is an ind-proper indscheme equipped with a morphism
$\vph:X \to V$. We obtain a canonical ``integration" map:
\begin{equation}\label{eq:integration}
\Gamma^{\IndCoh}(X,\omega_X) \to V.
\end{equation}

\noindent Indeed, our map $\vph$ is (tautologically) equivalent to
a morphism $\sO_X \to V\otimes \sO_X \in \QCoh(X)$. Tensoring with the dualizing
sheaf on $X$, we obtain a morphism $\omega_X \to V \otimes \omega_X \in \IndCoh(X)$,
or what is the same, a map $p^!(k) \to p^!(V)$ for $p:X \to \Spec(k)$ the projection.
By ind-properness and adjunction, 
we obtain a map $p_*^{\IndCoh}p^!(k) \to V$ as desired. 

\begin{rem}\label{r:upsilon-ff}

Using \cite{grbook} \S II.3 Lemma 3.3.7, one finds that this is a bijection;
i.e., specifying a map \eqref{eq:integration} is equivalent to giving
$X \to V$.

\end{rem}

We also have the following variant.
Suppose we are also given $Y$ a prestack and
$\pi:X \times Y \to \bA^1$ a map.

The map $\pi$ can be encoded by a morphism $X \to \Gamma(Y,\sO_Y)$ (thinking of the target
as a prestack as above). Therefore, 
we obtain a comparison map:
\[
c_{\pi}:\Gamma^{\IndCoh}(X,\omega_X) \to \Gamma(Y,\sO_Y) \in 
\Vect
\]

\noindent by the above. (As in Remark \ref{r:upsilon-ff}, 
one can actually recover $\pi$ from $c_{\pi}$.) 

\subsubsection{}

We now have the following non-degeneracy assertion for our
pairings.

\begin{prop}\label{p:cc-log}

Each of the morphisms:
\[
\begin{gathered}
\Gamma^{\IndCoh}(\Gr_{\bG_m}^{pos},\omega_{\Gr_{\bG_m}^{pos}}) 
\to \Gamma(\fL^+ \bA^1,\sO_{\fL^+ \bA^1}) \\
\Gamma^{\IndCoh}(\Gr_{\bG_m},\omega_{\Gr_{\bG_m}}) 
\to \Gamma(\fL^+ \bG_m,\sO_{\fL^+ \bG_m}) \\
\Gamma^{\IndCoh}(\Gr_{\bG_m}^{pos,\leq n},\omega_{\Gr_{\bG_m}^{pos,\leq n}}) 
\to \Gamma(\fL_n^+ \bA^1,\sO_{\fL_n^+ \bA^1}) \\
\Gamma^{\IndCoh}(\Gr_{\bG_m}^{\leq n},\omega_{\Gr_{\bG_m}^{\leq n}}) 
\to \Gamma(\fL_n^+ \bG_m,\sO_{\fL_n^+ \bA^1})
\end{gathered}
\]

\noindent coming from our pairings and \S \ref{ss:pairing-map}
is an isomorphism.

Moreover, each isomorphism is a map of commutative algebras,
where the left hand sides are commutative algebras using
the commutative monoid structures on the relevant spaces,
while the right hand sides are commutative algebras 
as functions on schemes.

\end{prop}

\begin{proof}

Each of these maps is a map of commutative algebras
because of the bilinearity of the pairings.

Next, we have:
\[
\begin{gathered}
\Gamma^{\IndCoh}(\Gr_{\bG_m}^{pos,\leq n},\omega_{\Gr_{\bG_m}^{pos,\leq n}}) =
\underset{m}{\oplus} \,
\Gamma^{\IndCoh}(\Sym^m\sD^{\leq n},\omega_{\Sym^m\sD^{\leq n}})
= \\
\underset{m}{\oplus} \, ((k[[t]]/t^n)^{\otimes m,S_m})^{\vee}
= \Sym((k[[t]]/t^n)^{\vee})
\end{gathered}
\]

\noindent where the second equality is by finiteness of
$\Sym^m \sD^{\leq n}$.

We similarly have 
$\fL_n^+ \bA^1 = \Spec(\Sym((k[[t]]/t^n)^{\vee}))$ by 
definition. Under these identifications, it suffices to
show that the map:
\[
\Sym((k[[t]]/t^n)^{\vee}) = 
\Gamma^{\IndCoh}(\Gr_{\bG_m}^{pos,\leq n},\omega_{\Gr_{\bG_m}^{pos,\leq n}}) 
\to \Gamma(\fL_n^+ \bA^1,\sO_{\fL_n^+ \bA^1}) =
\Sym((k[[t]]/t^n)^{\vee})
\]

\noindent from the proposition is the identity: 
it suffices to check this on
the generators of the left hand side (as we have a map
of commutative algebras), and there it follows 
by construction of the pairing.

Passing to the colimit as $n \to \infty$, we recover
the first map under consideration.

The second and fourth cases follow from Remark \ref{r:gr-n-pos-gp-cpl}.

\end{proof}

\begin{rem}

The assertion $\Gamma^{\IndCoh}(\Gr_{\bG_m},\omega_{\Gr_{\bG_m}}) 
\simeq \Gamma(\fL^+ \bG_m,\sO_{\fL^+ \bG_m})$ via this
pairing is standard from Contou-Carr\`ere \cite{cc}. 
The other assertions
are apparently less well-known,\footnote{For example,
they concern Cartier duality for commutative monoids,
which is less often considered than the group case.}
though they easy results.

\end{rem}

\section{Coherent sheaves on $\sY$}\label{s:indcoh}

\subsection{Overview}

In this section, we define $\IndCoh^*$ for $\sY$ and its relatives.
This section follows \cite{methods} \S 6, though we keep our exposition largely
independent of \emph{loc. cit}. 

We strive to keep our exposition as explicit as possible, so assume
whatever simplifying assumptions we like. We refer to \emph{loc. cit}. for
a more thorough development of the subject.

\subsection{Affine case}

\subsubsection{}

Suppose $S$ is an eventually coconnective (e.g. classical) affine scheme.

We define the (non-cocomplete, DG) subcategory $\Coh(S) \subset \QCoh(S)$
to consist of objects $\sF \in \QCoh(S)^+$ such that for each $n$, the object
$\tau^{\geq -n}(\sF) \in \QCoh(S)^{\geq -n}$ is compact. 
In other words, $\sF$ should be eventually coconnective and 
$\ul{\Hom}_{\QCoh(S)}(\sF,-)$ should commute with
filtered colimits in $\QCoh(S)^{\geq -n}$ for each $n$.
Clearly $\Perf(S) \subset \Coh(S)$.

\begin{notation}

For $S = \Spec(A)$, we also sometimes write $\Coh(A) \subset A\mod$ for 
the subcategory $\Coh(S) \subset \QCoh(S) \simeq A\mod$.

\end{notation}

We define $\IndCoh^*(S)$ as $\Ind(\Coh(S))$. 
Observe that we have a
natural functor $\Psi:\IndCoh^*(S) \to \QCoh(S)$.

Define $\IndCoh^*(S)^{\leq 0}$ as the subcategory generated
under colimits by coherent objects that are connective (with respect to the
$t$-structure on $\QCoh(S)$). There is a corresponding 
$t$-structure on $\IndCoh^*(S)$ with this as its connective subcategory.

\begin{lem}[\cite{methods} Lemma 6.4.1]\label{l:psi-affine}

The above functor $\Psi$ is $t$-exact and induces an equivalence 
$\Psi:\IndCoh^*(S)^+ \isom \QCoh(S)^+$.

\end{lem}

\subsubsection{Pullbacks}

Let $f:S \to T$ be an eventually coconnective morphism of affine, eventually
coconnective schemes.

\subsubsection{Pushforwards}

Now suppose $f:S \to T$ is a morphism of affine, eventually coconnective schemes.
In this case, there is a unique continuous DG functor:
\[
f_*^{\IndCoh}:\IndCoh^*(S) \to \IndCoh^*(T)
\]

\noindent fitting into a commutative diagram:
\[
\xymatrix{
\IndCoh^*(S) \ar[rr]^{f_*^{\IndCoh}} \ar[d]^{\Psi} &&  \IndCoh^*(T)\ar[d]^{\Psi} \\
\QCoh(S) \ar[rr]^{f_*} && \QCoh(T)
}
\]

\noindent and such that $f_*^{\IndCoh}(\Coh(S)) \subset \QCoh(T)^+$.
(Indeed, this functor is necessarily the ind-extension of 
$\Coh(S) \xar{f_*} \QCoh(T)^+ \simeq \IndCoh^*(T)^+ \subset \IndCoh^*(T)$.

In general, the functor $f_*^{\IndCoh}$ can\footnote{Unfortunately, \cite{methods}, which
is currently being revised, erroneously claims otherwise.}be pathological, i.e., it may fail to map $\IndCoh^*(S)^+$ into $\IndCoh^*(T)^+$.
This issue complicates the theory, compared to the finite type setting. 

\subsubsection{}

Suppose $f:S \to T$ as above is eventually coconnective (i.e., of finite
$\Tor$-amplitude). In this case,
$f^*:\QCoh(T) \to \QCoh(S)$ maps $\Coh(T)$ to $\Coh(S)$. By ind-extension,
we obtain a functor $f^{*,\IndCoh}:\IndCoh^*(T) \to \IndCoh^*(S)$.
It is immediate to see that it is left adjoint to $f_*^{\IndCoh}$.

Moreover, $f^{*,\IndCoh}$ is manifestly right $t$-exact. Therefore, 
$f_*^{\IndCoh}$ \emph{is} left $t$-exact in this case. 

\subsubsection{}

Suppose $f:S \to T$ is proper (e.g., an almost finitely presented
closed embedding). 

Then $f_*:\Coh(S) \to \QCoh(T)^+$ maps into $\Coh(T)$. Therefore,
the functor $f_*^{\IndCoh}$ admits a continuous right adjoint, 
that we denote $f^!$ in this case. 

\subsubsection{Semi-coherence}

We now provide a general hypothesis that ensures that $f_*^{\IndCoh}$ is
left $t$-exact, even when $f$ is not eventually coconnective. 

\begin{defin}

A connective and eventually coconnective 
commutative algebra $A$ is \emph{semi-coherent}
if every $M \in A\mod^+$ can be written as a filtered colimit
of objects $M_i$ such that:

\begin{itemize}

\item $M_i \in \Coh(A)$ % notation!!
for every $i$.

\item There exists an integer $r$ independent of $i$ 
such that every $M_i$ lies in $A\mod^{\geq -r}$.

\end{itemize}

We say $S = \Spec(A)$ is \emph{semi-coherent} if $A$ is.

\end{defin}

\begin{rem}\label{r:semi-coh-ab}

In the definition, if one restricts to $M \in A\mod^{\heart}$,
one finds that such an $r$ may be constructed (if it exists)
independent of $M$ (by considering direct sums of modules).  
It follows that it suffices to check the hypothesis just for $M \in A\mod^{\heart}$.

\end{rem}

\begin{example}

Any Noetherian (and eventually coconnective) 
$A$ is obviously semi-coherent. More generally, any 
coherent (and eventually coconnective) $A$ is semi-coherent.

\end{example}

We now have the following simple observation.

\begin{lem}\label{l:semi-coh-functor}

Suppose $S$ is a semi-coherent eventually connective affine scheme.
Suppose $\sC \in \DGCat_{cont}$ is a DG category with a $t$-structure
compatible with filtered colimits.  
Suppose $F:\IndCoh^*(S) \to \sC$ is a continuous left $t$-exact DG functor 
such that $F(\Coh(S)) \subset \sC^+$.
Then $F(\IndCoh^*(S)^+) \subset \sC^+$.

\end{lem}

\begin{proof}

Suppose $\sF \in \IndCoh^*(S)^+$. By assumption and $S$, we can write
$\sF$ as a filtered colimit $\colim_i \sF_i$ with 
$\sF_i \in \Coh(S) \cap \QCoh(S)^{\geq -r}$ for
some $r$; equivalently, $\sF_i \in \Coh(S) \cap \IndCoh^*(S)^{\geq -r}$.
Then $F(\sF) = \colim_i F(\sF_i) \in \sC^{\geq -r}$ by assumption, giving
the claim.

\end{proof}

\begin{cor}\label{c:semi-coh-push}

Suppose $f:S \to T$ is a morphism of eventually coconnective affine schemes
with $S$ semi-coherent. Then $f_*^{\IndCoh}$ is $t$-exact.

\end{cor}

\begin{proof}

By Lemma \ref{l:semi-coh-functor}, $f_*^{\IndCoh}$ is left $t$-exact.
Therefore, we have a commutative diagram:
\[
\xymatrix{
\IndCoh^*(S)^+ \ar[rr]^{f_*^{\IndCoh}} \ar[d]^{\Psi} &&  \IndCoh^*(T)^+\ar[d]^{\Psi} \\
\QCoh(S)^+ \ar[rr]^{f_*} && \QCoh(T)^+
}
\]

\noindent in which the left arrows are $t$-exact equivalences and the
bottom arrow is $t$-exact (by affineness). This gives the claim.

\end{proof}

\subsubsection{}

We have the following important technical result.

\begin{prop}\label{p:semi-coh}

The (classical) commutative algebra $A_n$ defined in \S \ref{ss:coords} is
semi-coherent.

\end{prop}

In other words, $\widetilde{\sZ}^{\leq n}$ is semi-coherent.

We defer the proof of the proposition to \S \ref{ss:z-semi-coh-pf}, 
at the end of this section. Until that point, we assume it.

\subsection{Coherent sheaves with support}

\subsubsection{}

We now consider the following situation. 

Let $Z$ be an affine, eventually coconnective scheme. 
Let $f:Z \to \bA^1$ be a function. Let $U = \{f \neq 0\} \subset Z$, 
and let $j:U \to Z$ be the embedding.
Let $i:Z_0 \into Z$ be a closed subscheme with 
$Z \setminus Z_0 = U$.

The following type of result is standard in the finite type setting.

\begin{lem}\label{l:coh-supp}

Suppose that:

\begin{itemize}

\item $i$ is almost finitely presented.

\item The map $i_*^{\IndCoh}:\IndCoh^*(Z_0) \to \IndCoh^*(Z)$ is $t$-exact.

\end{itemize}

\noindent Then the functor:
\[
(j^{*,\IndCoh},i^!):\IndCoh^*(Z) \to \IndCoh^*(U) \times \IndCoh^*(Z_0)
\]

\noindent is conservative. 

\end{lem}

\begin{proof}

Let $\sF \in \IndCoh^*(Z)$ be given. We suppose $\sF$ is a non-zero
object with $j^{*,\IndCoh}(\sF) = 0$. Our goal is to construct 
an object $\sH \in \IndCoh^*(Z_0)$ and a non-zero map 
$i_*^{\IndCoh}(\sH) \to \sF$. 

\step We have:\footnote{Indeed, $j_*^{\IndCoh}$ is automatically left $t$-exact
(hence $t$-exact) because $j$ is flat; therefore, this formula follows from
the quasi-coherent setting by ind-extension.}
\[
j_*^{\IndCoh}j^{*,\IndCoh}(\sF) = \colim \big(\sF \xar{f} \sF \xar{f} \ldots \big) = 
0.
\]

Let $\sG \in \Coh(Z)$ be given with a non-zero map $\alpha:\sG \to \sF$; such 
$(\sG,\alpha)$
exists because $\sF$ is assumed non-zero. 

By compactness of $\sG$, there is an integer $n$ such that the composition:
\[
\sG \xar{\alpha} \sF \xar{f^n} \sF
\]

\noindent is null-homotopic. This map coincides with the composition:
\[
\sG \xar{f^n} \sG \xar{\alpha} \sF
\]

\noindent so this map must be null-homotopic as well. 
It follows that $\alpha$ factors through a (necessarily non-zero) map: 
\[
\sG/f^n\coloneqq \Coker(f^n:\sG \to \sG) \to \sF
\]

\noindent for some $n \gg 0$.

We have a standard co/fiber sequence:
\[
\sG/{f^{n-1}} \xar{f} \sG/f^n \to \sG/f.
\]

\noindent By descending induction, we see that there must exist
a non-zero map from $\sG/f \to \sF$. Replacing $\sG$ with $\sG/f \in \Coh(Z)$, 
we can assume that $f$ acts by zero on our original object $\sG$.

\step We now digress to make the following simple commutative algebra 
observation.\footnote{In our eventual application, $Z_0$ is the classical
scheme underlying $\{f=0\}$. This step is unnecessary under that
additional hypothesis.}
Suppose $Z = \Spec(A)$ below.

Let $I \subset H^0(A)$ be the ideal of 
elements vanishing on $Z_0^{cl}$. Let $(f) \subset H^0(A)$ 
be the ideal generated by $f$. We claim that there are integers $r,s$
such that:
\[
I^r \subset (f), (f^s) \subset I. 
\]

For any $g \in I$, we have $\{g \neq 0\} \subset \{f \neq 0\}$ by assumption. 
Therefore, $f$ is invertible in $H^0(A)[g^{-1}]$, so $g^{r_0} \in (f)$ for some
$r_0$ (depending on $g$). 
Because $i$ is almost finitely presented, $I$ is finitely generated,
so we obtain $I^r \subset (f)$ for $r \gg 0$.

The other inclusion (which we do not need) follows similarly: $\{f \neq 0\}$ is covered
by $\{g \neq 0\}$ for $g \in I$, so $1 \in I[f^{-1}] = H^0(A)[f^{-1}]$,
so $f^s \in I$ for some $s$.

We deduce the following: any module: 
\[
M \in H^0(A/f)\mod^{\heart} \subset H^0(A)\mod^{\heart} = 
A\mod^{\heart}
\]

\noindent admits a finite filtration:
\[
0 = I^r M \subset I^{r-1} M \subset \ldots \subset IM \subset M
\]

\noindent with subquotients lying in $H^0(A)/I\mod^{\heart} \subset A\mod^{\heart}$.

\step We now return to earlier setting.

Recall that we have $\sG \in \Coh(Z)$ on which $f$ acts null-homotopically.
Because $\sG$ is bounded, it has a finite Postnikov filtration as an object
of $\QCoh(Z)$
with associated graded terms $H^i(\sG)[-i]$. 
Each $H^i(\sG) \in 
\QCoh(Z)^{\heart}$. Because $f$ acts by zero on them, they lie
in $\QCoh(\{f = 0\})^{\heart}$. By the previous step, each
of these terms is obtained by successively extending 
objects in $\QCoh(Z_0)^{\heart}$.

Now the diagram:
\[
\xymatrix{
\IndCoh(Z_0)^+ \ar[r]^{i_*^{\IndCoh}} \ar[d]^{\Psi} & \IndCoh(Z)^+ \ar[d]^{\Psi} \\
\QCoh(Z_0)^+ \ar[r]^{i_*} & \QCoh(Z)^+}
\]

\noindent commutes because we assumed $i_*^{\IndCoh}$ $t$-exact;
moreover, the vertical arrows are equivalences. Therefore, $\sG$
admits a finite filtration as an object of $\IndCoh^*(Z)$
with associated graded terms $\IndCoh$-pushed forward from $Z_0$. 
Therefore, one of these associated graded terms must have a non-zero map
to $\sF$, giving the claim.  

\end{proof}

\subsubsection{Application to $\widetilde{\sZ}^{\leq n}$}

We now apply the above to construct generators of $\IndCoh^*(\widetilde{\sZ}^{\leq n})$.

\begin{prop}\label{p:tilde-z-gens}

For
$n > 0$, the DG category $\IndCoh^*(\widetilde{\sZ}^{\leq n})$ is compactly 
generated by the objects 
$\{\iota_*^{r,\IndCoh}(\sO_{\widetilde{\sZ}^{\leq n}})\}_{r \geq 0}$.

For $n = 0$, the DG category $\IndCoh^*(\widetilde{\sZ}^{\leq 0})$ is compactly 
generated by $\sO_{\widetilde{\sZ}^{\leq 0}})$.

\end{prop}

\begin{proof}

We proceed by induction on $n$. 

\step 

The $n = 0$ case is simple: $\widetilde{\sZ}^{\leq 0}$ is $\bA^1$ (with coordinate
$a_0$). 

\step 

We now fix $n > 0$, and that the proposition for $n - 1$. 
We introduce the following notation.   

Let $\sC \subset \IndCoh^*(\widetilde{\sZ}^{\leq n})$ be the subcategory generated
by the objects $\iota_*^{r,\IndCoh}\sO_{\widetilde{\sZ}^{\leq n}}$. We need
to show that $\sC$ is the full $\IndCoh^*$.

\step \label{st:gen-3}

We observe that \eqref{eq:fund-ses} implies that 
$\delta_{n,*}^{\IndCoh}(\sO_{\sZ}^{\leq n-1}) \in \sC$. 
More generally, we find 
$\iota_*^{r,\IndCoh}\delta_{n,*}^{\IndCoh}(\sO_{\sZ}^{\leq n-1}) \in \sC$ for
any $r\geq 0$. 

\step 

Below, we will apply Lemma \ref{l:coh-supp} with 
$i:Z_0 \to Z$ corresponding to 
$\delta_n:\widetilde{\sZ}^{\leq n-1} \to \widetilde{\sZ}$. We remark
that $\delta_n$ is almost finitely presented by 
Lemma \ref{p:z-flat-axes} \eqref{i:z-axes-2} 
(cf. the proof of Corollary \ref{c:iota-afp}), and that $\delta_{n,*}^{\IndCoh}$
is $t$-exact by Proposition \ref{p:semi-coh}. 

In other words, our application of the lemma is justified. 

\step 

We now treat the $n = 1$ case. 

Suppose $\sF \in \IndCoh^*(\widetilde{\sZ}^{\leq 1})$ lies in the 
right orthogonal to $\sC$. 

By Step \ref{st:gen-3} and the $n = 0$ case, $\delta_1^!(\sF) = 0$, and 
$\delta_1^!(\iota^r)^! \sF = 0$ more generally. Therefore, 
by Lemma \ref{l:coh-supp}, $\sF$ is $*$-extended from:
\[
\widetilde{\sZ}^{\leq 1} \setminus  
\big(\widetilde{\sZ}^{\leq 0} \cup \ldots \cup \iota^r(\widetilde{\sZ}^{\leq 0})\big)
\]

\noindent for each $r$. It follows that $\sF$ is $*$-extended from:
\[
\widetilde{\sU}_1 \coloneqq 
\widetilde{\sZ}^{\leq 1} \times_{\sZ^{\leq 1}} \sU_1.
\]

But by Lemma \ref{l:u1} (cf. \S \ref{sss:u-reg-noeth}), 
$\IndCoh^*(\widetilde{\sU}_1) \simeq \QCoh(\widetilde{\sU}_1)$ 
is generated by its structure sheaf. Since this is the $*$-restriction
of the structure sheaf on $\sZ^{\leq 1}$, we obtain the claim. 

\step 

We now assume $n>1$ case. 

In this case, we have:
\[
\iota_*^{r,\IndCoh}\delta_{n,*}^{\IndCoh} \simeq 
\delta_{n,*}^{\IndCoh}
\iota_*^{r,\IndCoh}.
\]

\noindent (Here we use $\iota$ to denote the map for both $\sZ^{\leq n}$ 
and $\sZ^{\leq n-1}$.)

By induction and Step \ref{st:gen-3}, $\sC$ contains the subcategory
generated under colimits by the essential image of
$\delta_{n,*}^{\IndCoh}$. Therefore, by Lemma \ref{l:coh-supp}, 
any $\sF$ in the right orthogonal to $\sC$ is $*$-extended
from $\widetilde{\sU}_n$. 

But again, $\IndCoh^*(\widetilde{\sU}_n) \simeq \QCoh(\widetilde{\sU}_n)$ is generated
by its structure sheaf by Lemma \ref{l:un}; as 
$\sO_{\widetilde{\sZ}^{\leq n}} \in \sC$, we obtain the claim.

\end{proof}

\subsection{Stacky case}

\subsubsection{}

Now suppose $S$ is a stack of the form $T/G$ for $G$ classical affine group scheme
and $T$ an eventually coconnective affine scheme.

In this case, $G$ acts weakly on $\IndCoh^*(T)$, i.e., $\QCoh(G)$
acts canonically on $\IndCoh^*(T)$ (by flatness of $G$ over a point). 
We set $\IndCoh^*(S) \coloneqq \IndCoh^*(T)^{G,w}$.
%One can see that this is canonically independent of the presentation of
%$S$ as $T/G$, cf. \cite{methods} Corollary 6.25.2.

%We say an object $\sF \in \IndCoh^*(S)$ is \emph{coherent} if
%its pullback to some (equivalently any) cover is coherent. 
%We let $\Coh(S) \subset \IndCoh^*(S)$ denote the corresponding
%full subcategory. We remark that in general, objects of
%$\Coh(S)$ may not be compact, but this is not an issue if
%$G$ is assumed to be of finite type. 

There is a unique $t$-structure on $\IndCoh^*(S)$ such that the
forgetful functor:
\[
\IndCoh^*(S) \to \IndCoh^*(T)
\]

\noindent is $t$-exact. 

\subsubsection{}

In the above setting, given a map $f:S_1 \to S_2$ of stacks of the
above types, we obtain similar functoriality as in the non-stacky case;
we omit the details.  

\subsubsection{}

We now have:

\begin{cor}\label{c:z-gen}

The category $\IndCoh^*(\sZ^{\leq n})$ is compactly generated
by the objects $\iota_*^{r,\IndCoh}(\sO_{\sZ^{\leq n}})(m)$, where $r \geq 0$ and 
$m \in \bZ$.

\end{cor}

\begin{proof}

For $G$ an affine algebraic group (in particular, of finite type) and 
$\sF \in \sC^{G,w}$, it is a general fact that $\sF$ is compact
if and only if its image $\Oblv(\sF) \in \sC$ is compact.\footnote{See 
\cite{methods}, proof of Lemma 5.20.2.}
Moreover, the functor $\Oblv:\sC^{G,w} \to \sC$ always generates
the target under colimits. 

Therefore, writing $\sZ^{\leq n} = \widetilde{\sZ}^{\leq n}/\bG_m$, 
the above objects are indeed compact. The generation follows from 
Proposition \ref{p:tilde-z-gens}.

\end{proof}

\subsubsection{}

There is a natural functor:
\[
\Psi:\IndCoh^*(\sZ^{\leq n}) \to \QCoh(\sZ^{\leq n})
\]

\noindent obtained by passing to weak $\bG_m$-invariants for
the corresponding functor for $\widetilde{\sZ}^{\leq n}$. This functor
is an equivalence on coconnective (equivalently: eventually coconnective) 
subcategories via the corresponding assertion for $\widetilde{\sZ}^{\leq n}$.

We let $\Coh(\sZ^{\leq n}) \subset \IndCoh^*(\sZ^{\leq n})$ 
denote the subcategory of compact objects. By the above, 
the functor $\Coh(\sZ^{\leq n}) \to \QCoh(\sZ^{\leq n})$
induced by $\Psi$ is fully faithful; its essential image is 
exactly the subcategory of 
eventually coconnective almost compact objects. 

\subsubsection{}

We now observe that the functor:
\[
\iota_*^{\IndCoh}:\IndCoh^*(\sZ^{\leq n}) \to \IndCoh^*(\sZ^{\leq n})
\]

\noindent is $t$-exact. 

Indeed, by definition, this assertion reduces to the analogous
statement for $\widetilde{\sZ}^{\leq n}$, which in turn follows 
from Proposition \ref{p:semi-coh}.

The same analysis applies for pushforward functors:
\[
\delta_{n+1,*}^{\IndCoh}:\IndCoh^*(\sZ^{\leq n}) \to \IndCoh^*(\sZ^{\leq n+1}).
\]

\subsection{Ind-algebraic stacks}

We define:
\[
\IndCoh^*(\sY_{\log}^{\leq n}) \coloneqq \colim \, \big(
\IndCoh^*(\sZ^{\leq n}) \xar{\iota_*^{\IndCoh}}
\IndCoh^*(\sZ^{\leq n}) \xar{\iota_*^{\IndCoh}} \ldots \big) \in \DGCat_{cont}.
\]

\noindent We let: 
\[
i_{n,*}^{\IndCoh}: \IndCoh^*(\sZ^{\leq n}) \to \IndCoh^*(\sY_{\log}^{\leq n})
\]

\noindent denote the structural functor. 

Similarly, we define:
\[
\IndCoh^*(\sY_{\log}) \coloneqq \underset{n}{\colim} \, \IndCoh^*(\sY_{\log}^{\leq n}) =
\underset{\iota_*^{\IndCoh},\delta_{n,*}^{\IndCoh}}{\underset{n,m}{\colim}} \, \IndCoh^*(\sZ^{\leq n}) \in \DGCat_{cont}.
\]

By the above and \cite{whit} Lemma 5.4.3, there is a unique 
$t$-structure on $\IndCoh^*(\sY_{\log}^{\leq n})$ such that 
each structural functor:
\[
\IndCoh^*(\sZ^{\leq n}) \to \IndCoh^*(\sY_{\log}^{\leq n})
\]

\noindent is $t$-exact. Similarly, there is a unique $t$-structure on 
$\IndCoh^*(\sY_{\log})$ such that each structural functor:
\[
\IndCoh^*(\sY_{\log}^{\leq n}) \to \IndCoh^*(\sY_{\log})
\]

\noindent is $t$-exact.

Each of these categories is compactly generated as all of our structural
functors preserve compact objects.

\subsection{Grassmannian actions}

Suppose $n>0$ in what follows. 

\subsection{}

Recall that $\Gr_{\bG_m}$ tautologically acts on $\sY_{\log}$.

Observe that $\sZ \subset \sY_{\log}$ is preserved
under the action of the submonoid 
$\Gr_{\bG_m}^{pos} \subset \Gr_{\bG_m}$.

Indeed, suppose more generally that $Y$ is a prestack
mapping to $\fL \bA^1/\bG_m$. Then:
\[
\Gr_{\bG_m}^{pos} = \Ker(\fL^+ \bA^1/\bG_m \to \fL \bA^1/\bG_m)
\]

\noindent acts canonically on:
\[
Z \coloneqq Y \underset{\fL \bA^1/\bG_m}{\times} \fL^+ \bA^1/\bG_m
\]

\noindent compatibly with the canonical $\Gr_{\bG_m} = 
\Ker(\fL^+\bB \bG_m \to \fL \bB \bG_m)$-action on:
\[
Y_{\log} \coloneqq Y \underset{\fL\bB\bG_m}{\times} \fL^+ \bB \bG_m.
\]

Taking $Y = \sY$, we recover the claim.

\subsubsection{}

Next, observe that the action of 
$\Gr_{\bG_m}^{pos,\leq n} \subset \Gr_{\bG_m}^{\leq n}$
on $\sY_{\log}^{\leq n}$ preserves $\sZ^{\leq n}$.

Indeed, this is an immediate consequence of the above.

\subsubsection{}

By the above, there is a canonical action:
\[
\QCoh(\Gr_{\bG_m}^{pos,\leq n}) \actson \QCoh(\sZ^{\leq n})
\]

\noindent where the left hand side is given its convolution monoidal
structure. 

For $\sF \in \Coh(\Gr_{\bG_m}^{pos,\leq n}) \subset \QCoh(\Gr_{\bG_m}^{\leq n})$,
the corresponding action functor:
\[
\QCoh(\sZ^{\leq n}) \to \QCoh(\sZ^{\leq n})
\]

\noindent clearly preserves $\Coh(\sZ^{\leq n})$. Therefore, we obtain 
a canonical action:
\[
\IndCoh(\Gr_{\bG_m}^{pos,\leq n}) \actson \IndCoh^*(\sZ^{\leq n}).
\]

\subsubsection{}

By Remark \ref{r:gr-n-pos-gp-cpl}, we have:
\[
\IndCoh(\Gr_{\bG_m}^{\leq n}) \underset{\IndCoh(\Gr_{\bG_m}^{pos,\leq n})}{\otimes} 
\simeq \big(
\IndCoh^*(\sZ^{\leq n}) \xar{\iota_*^{\IndCoh}}
\IndCoh^*(\sZ^{\leq n}) \xar{\iota_*^{\IndCoh}} \ldots \big) \simeq 
\IndCoh^*(\sY_{\log}^{\leq n}).
\]

\noindent Thus, we obtain a natural action of $\IndCoh(\Gr_{\bG_m}^{\leq n})$
on $\IndCoh^*(\sY_{\log}^{\leq n})$ such that the functor
$i_{n,*}^{\IndCoh}$ is a morphism of $\IndCoh(\Gr_{\bG_m}^{pos,\leq n})$-module
categories. 

\subsubsection{}

The map $\delta_{n,*}^{\IndCoh}:\IndCoh^*(\sZ^{\leq n}) \to \IndCoh^*(\sZ^{\leq n+1})$
is naturally a map of $\IndCoh(\Gr_{\bG_m}^{pos,\leq n})$-module categories,
as is evident by again considering subcategories of compact objects. 

Therefore, $\IndCoh(\Gr_{\bG_m})$ naturally acts on $\IndCoh(\sY_{\log})$,
compatibly with the above actions. 

\subsection{Ind-coherent sheaves on $\sY$}

We now define:
\[
\IndCoh^*(\sY) \coloneqq \IndCoh^*(\sY_{\log})^{\Gr_{\bG_m},w}.
\]

\noindent By definition, the right hand side is:
\[
\TwoHom_{\IndCoh(\Gr_{\bG_m})\mod}(\Vect, \IndCoh^*(\sY_{\log})).
\]

We let:
\[
\Oblv:\IndCoh^*(\sY) \to \IndCoh^*(\sY_{\log})
\]

\noindent denote the (conservative) forgetful functor.
It admits a left adjoint:
\[
\Av_!^{\Gr_{\bG_m},w}.
\] 

The resulting monad on $\IndCoh^*(\sY_{\log})$ is easily seen to be 
$t$-exact.\footnote{This reduces to the assertion that for 
$\sF \in \IndCoh(\Gr_{\bG_m}^{\leq n,pos})^{\heart}$, the 
action functor $\IndCoh^*(\sZ^{\leq n}) \to \IndCoh^*(\sZ^{\leq n})$
is $t$-exact. Filtering $\sF$, we are reduced to the case that
it is the skyscraper sheaf at a $k$-point. In that case, the
relevant functor is a composition of functors $\iota_*^{\IndCoh}$,
so $t$-exact by Proposition \ref{p:semi-coh}.}
Therefore, we obtain:

\begin{prop}\label{p:y-t-str}

There is a unique $t$-structure on $\IndCoh^*(\sY)$ such that the
functor $\Oblv:\IndCoh^*(\sY) \to \IndCoh^*(\sY_{\log})$ is $t$-exact.
Moreover, the functor $\Av_!^{\Gr_{\bG_m},w}$ is $t$-exact.

\end{prop}

Indeed, this follows from:

\begin{lem}

Let $\sC \in \DGCat_{cont}$ be equipped with a $t$-structure compatible with
filtered colimits and a right $t$-exact monad $T:\sC \to \sC$.
Then $T\mod(\sC)$ admits a unique $t$-structure such that the
forgetful functor $T\mod(\sC) \to \sC$ is $t$-exact. 

\end{lem}

\subsubsection{}\label{ss:y-trun-indcoh}

In the truncated setting, we similarly define:
\[
\IndCoh^*(\sY^{\leq n}) \coloneqq 
\IndCoh^*(\sY_{\log}^{\leq n})^{\Gr_{\bG_m}^{\leq n},w}.
\]

\noindent We again have adjoint functors:
\[
\Oblv,\Av_!^{\Gr_{\bG_m}^{\leq n},w}
\]

\noindent and a $t$-structure on $\IndCoh^*(\sY^{\leq n})$.

\subsubsection{}

By functoriality of the constructions, there is a natural functor:
\[
\lambda_{n,*}^{\IndCoh}:\IndCoh^*(\sY^{\leq n}) \to \IndCoh^*(\sY).
\]

\noindent This functor is $t$-exact and preserves compact objects. Moreover, we have:

\begin{lem}\label{l:lambda-ff}

For $n>0$, the functor $\lambda_{n,*}^{\IndCoh}$ is fully faithful.

\end{lem}

Roughly speaking, this is true because the map $\sY^{\leq n} \to \sY$ is 
formally \'etale for $n>0$, which follows from the corresponding
fact for $\LocSys_{\bG_m}^{\leq n} \to \LocSys_{\bG_m}$ (which 
follows from Proposition \ref{p:locsys-log}). Unwinding the constructions
to convert this argument into a proof is straightforward. 

\subsection{Proof of Proposition \ref{p:semi-coh}}\label{ss:z-semi-coh-pf}

We now give the proof of Proposition \ref{p:semi-coh}, which we deferred earlier.

\subsubsection{}

Our argument is by induction on $n$. We proceed in steps.

\subsubsection{Step 1: $n = 0$ case}

The $n = 0$ case is trivial, as $A_0 = k[a_0]$ is Noetherian.

\subsubsection{Step 2: $n = 1$ case} 

The inductive step we give for $n>1$ below may be adapted to treat the 
$n = 1$ case, but it requires somewhat more work; we indicate how this works
in \S \ref{sss:semi-coh-7}. We prefer
to give a direct argument. Actually, we will show that $A_1$ is coherent.

Recall the Noetherian subalgebras $A_{1,\leq m} \subset A_1$ from \S \ref{ss:rs-sketch}.
It suffices to show that for any finitely presented $M \in A_1\mod^{\heart}$,
there is an integer $m \gg 0$ and finitely generated
$N \in A_{1,\leq m}\mod^{\heart}$ such that
there is an isomorphism $N \otimes_{A_{1,\leq m}} A_1 \isom M$, where
the tensor product is understood in the derived sense.

Choose some $m_0$ and $N_0 \in A_{1,\leq m_0}\mod^{\heart}$ with an 
isomorphism $H^0(N_0 \otimes_{A_{1,\leq m_0}} A_1) \isom M$; this may be
done by choosing generators and relations for $M$ and $m_0$ such that 
the relations all involve linear combinations with coefficients in $A_{1,\leq m_0}$.

Choose $m > m_0$ such that $N_0$ does not contain any $(b_{-1}-i)$-torsion 
for any integer $i\geq m$; this can be done because the finitely generated
$A_{1,\leq m_0}$-module $N_0$ only admits finitely many associated primes.
Then it is easy to see that $N \coloneqq 
H^0(A_{1,\leq m} \otimes_{A_{1,\leq m_0}} N_0)$ satisfies the hypotheses.

\begin{rem}

We see here that $A_1$ is actually coherent, not merely semi-coherent.
We do not consider this question for general $n$.

\end{rem}

\subsubsection{Step 3: filtered colimits, extensions, and effective bounds} 
We now make the following observations for d\'evissage.

For an integer $r$, we say that $M \in A_n\mod$ is \emph{$r$-good} if $M$ can be 
expressed as a filtered colimit of coherent objects concentrated in 
degrees $\geq -r$. Clearly any such $M$ lies in $A_n\mod^{\geq -r}$.

In other words, $M$ is $r$-good if it lies in the full subcategory:
\[
\Ind(\Coh(A) \cap A\mod^{\geq -r}) \subset A\mod^{\geq -r}.
\]

\noindent Here the natural functor is fully faithful because
$\Coh(A) \cap A\mod^{\geq -r}$ consists of objects that are 
compact in $A\mod^{\geq -r}$ by assumption.

It follows from the second description that $r$-good objects are closed
under filtered colimits (for fixed $r$). We also observe that $r$-good objects
are closed under extensions.

Finally, we say an object $M \in A_n\mod$ is \emph{good} if it is
$r$-good for some $r$. So our task is to show that any $M \in A_n\mod^+$ 
is good, or equivalently, any $M \in A_n\mod^{\heart}$.

\subsubsection{Step 4: reductions}

We now begin our induction. Take $n>1$ and assume the result for $n-1$.

Suppose $M \in A_n\mod^+$. 
It it suffices to show that:
\[
M[b_{-n}^{-1}] \coloneqq \underset{b_{-n}}{\colim} \, M \text{ and }
\Coker(M \to M[b_{-n}^{-1}])
\]

\noindent are good (where $\Coker$ indicates the \emph{homotopy cokernel}, 
i.e., the cone).

We check these assertions below. We remark that by Remark \ref{r:semi-coh-ab},
we are reduced to considering $M \in A_n\mod^{\heart}$.

\subsubsection{Step 5: generic case}

Note that $M[b_{-n}^{-1}] \in A_n[b_{-n}^{-1}]\mod^{\heart} \subset A_n\mod^{\heart}$.

As $n>1$, we have an isomorphism:
\[
A_n[b_{-n}^{-1}] = k[b_{-1},\ldots,b_{-n},b_{-n}^{-1}].
\]

\noindent by Lemma \ref{l:un}.

As this algebra is regular Noetherian, it follows that
$M[b_{-n}^{-1}]$ has bounded $\Tor$-amplitude as an 
$A_n[b_{-n}^{-1}]$-module (simply because it is bounded from below),
so the same is true of $M[b_{-n}^{-1}]$ as an $A_n$-module.
By standard homological algebra, $M[b_{-n}^{-1}]$ can be represented
by a bounded complex of flat (classical) $A_n$-modules. 
By Lazard, a flat $A_n$-module is good, so $M[b_{-n}^{-1}]$ is good.

\subsubsection{Step 6: torsion case} 

We now show that $\widetilde{M} = \Coker(M \to M[b_{-n}^{-1}])$ is good. 
More generally, we show that any bounded complex 
$\widetilde{M} \in A_n\mod^+$ such that $b_{-n}$ acts locally nilpotently on
its cohomologies is good. As the complex $\widetilde{M}$ is assumed bounded,
it suffices to treat its cohomology groups one at a time, so 
we can assume (up to shifting) that $\widetilde{M} \in A_n\mod^{\heart}$.

By induction, there exists an integer $r$ such that any module
$N \in A_{n-1}\mod^{\heart}$ is $r$-good as an $A_{n-1}$-module.
Because the projection $A_n \xar{b_{-n} \mapsto 0} A_{n-1}$ is
almost finitely presented by Proposition \ref{p:z-flat-axes} \eqref{i:z-axes-2},
so coherent $A_{n-1}$-complexes restrict to coherent $A_n$-complexes, 
$N$ restricts to an $r$-good $A_n$-module. 

Now let $\widetilde{M}_i \subset \widetilde{M}$ be the 
non-derived kernel of the map $b_{-n}^i:\widetilde{M} \to \widetilde{M}$. 
Clearly $\widetilde{M}_i/\widetilde{M}_{i-1} \in A_{n-1}\mod^{\heart} \subset
A_n\mod^{\heart}$, so is $r$-good for the above $r$ (which is independent
of everything in sight except $n$). The module 
$\widetilde{M}_i$ is then $r$-good, as it is obtained by successively extending
$r$-good modules. Finally, $\widetilde{M} = \colim_i \widetilde{M}_i$ is 
a filtered colimit of $r$-good modules, so is $r$-good. This concludes the argument.

\subsubsection{Step 7: revisiting the $n = 1$ case}\label{sss:semi-coh-7}

Finally, we make a remark that the above argument can be adapted to 
treat the $n = 1$ case, if one so desires. As in Lemma \ref{l:u1}, 
we should consider the localization of 
$A_1$ at $\{b_{-1},b_{-1}-1,b_{-1}-2,\ldots\}$
(instead of simply at $b_{-1}$); one again then obtains a Noetherian ring
of finite global dimension. The argument concludes by noting that 
any classical $A_1/(b_{-1}-i) = k[a_i]$-module is $0$-good, so 
any module for $A_1$ that is torsion for the above elements (thus
necessarily the direct sum of its torsion with respect to each) is
$0$-good (by an argument as above).

In other words, we need to localize at infinitely many elements before
obtaining a regular ring; the saving
grace is that the goodness for torsion modules at each is bounded independently
of the element.

\section{Spectral realization of Weyl algebras}\label{s:weyl}

\subsection{Overview}

\subsubsection{}

Let $i_n:\sZ^{\leq n} \into \sY_{\log}^{\leq n}$ denote the 
canonical embedding, as in \S \ref{ss:y-not}.

We have $i_{n,*}^{\IndCoh}(\sO_{\sZ^{\leq n}}) \in \Coh(\sY_{\log}^{\leq n})$,
which lies in $\IndCoh^*(\sY_{\log}^{\leq n})^{\heart}$ by
Proposition \ref{p:zn-cl}.

Let $\sF_n \coloneqq 
\Av_!^{\Gr_{\bG_m}^{\leq n},w}
(i_{n,*}(\sO_{\sZ^{\leq n}}))$, which is a compact object in $\IndCoh^*(\sY^{\leq n})$.

In this section, we construct an action of 
a Weyl algebra in $2n$ generators on $\sF_n$ (considered as an object of 
$\IndCoh^*(\sY^{\leq n})$).

\subsubsection{}

To formulate our construction more canonically, let
$W_n$ denote the algebra of global differential operators
on the scheme $\fL_n^+ \bA^1 = \Spec\Sym(t^{-n}k[[t]]dt/k[[t]]dt)$. 

Let $W_n^{op}$ denote $W_n$ with the reversed multiplication.
We will construct a canonical homomorphism:
\begin{equation}\label{eq:action}
W_n^{op} \to \ul{\End}_{\IndCoh^*(\sY^{\leq n})}(\sF_n) \in \Alg = \Alg(\Vect).
\end{equation}

\begin{rem}

In Proposition \ref{p:ff}, we will show that this map is actually an isomorphism. 

\end{rem}

\subsection{Reduction to generators and relations}\label{ss:coconn-reduction}

Note that $\sF_n$ lies in the heart of the $t$-structure constructed
in Proposition \ref{p:y-t-str} (or rather, its truncated counterpart, as in 
\S \ref{ss:y-trun-indcoh}).

Therefore, the right hand side of \eqref{eq:action} 
lies in cohomological degrees $\geq 0$.

As the left hand side of \eqref{eq:action} is certainly in cohomological 
degree $0$, it suffices to construct a homomorphism:
\[
W_n^{op} \to \tau^{\leq 0} \ul{\End}_{\IndCoh^*(\sY^{\leq n})}(\sF_n) = 
H^0\ul{\End}_{\IndCoh^*(\sY^{\leq n})}(\sF_n) \in \Alg(\Vect^{\heart}).
\]

As both terms are now in degree $0$ and the left hand side
has a standard algebra presentation, such a construction 
may be given by constructing generators and checking relations.

We need to construct maps:
\begin{equation}\label{eq:fns}
\begin{gathered}
t^{-n}k[[t]]dt/k[[t]]dt = (k[[t]]/t^n)^{\vee} \to 
\ul{\End}_{\IndCoh^*(\sY^{\leq n})}(\sF_n) \\
(\omega \in t^{-n}k[[t]]dt/k[[t]]dt) \mapsto \vph_{\omega}
\end{gathered}
\end{equation}

\noindent and:
\begin{equation}\label{eq:vfs}
\begin{gathered}
k[[t]]/t^n \to 
\ul{\End}_{\IndCoh^*(\sY^{\leq n})}(\sF_n) \\
(f \in k[[t]]/t^n) \mapsto \xi_{f}.
\end{gathered}
\end{equation}

\noindent We then need to check that the $\vph_{\omega}$ operators
mutually commute, that the $\xi_f$ operators mutually commute, and
the identity:\footnote{The sign on the right hand side reflects working
with $W_n^{op}$, not $W_n$.}

\begin{equation}\label{eq:uncertainty}
[\xi_f,\vph_{\omega}] = -\Res(f\omega) \cdot \id \in  
H^0\ul{\End}_{\IndCoh^*(\sY^{\leq n})}(\sF_n).
\end{equation}

We provide these constructions and check these identities in what follows.

\subsection{Action of functions}\label{ss:fn-intro}

First, we construct the map $\omega \mapsto \vph_{\omega}$ 
from \eqref{eq:fns}.

We assume the reader is familiar with the notation
from \S \ref{s:aj} below. 

\subsubsection{}\label{ss:fn-concl}

We have a commutative diagram:
\[
\xymatrix{
\Gr_{\bG_m}^{pos,\leq n} \times \sZ^{\leq n} \ar[d]^{\alpha}
 \ar[dr]^{p_1} 
& \\
\Gr_{\bG_m}^{pos,\leq n} \times \sZ^{\leq n} \ar[r]^{p_1}
& \Gr_{\bG_m}^{pos,\leq n}
}
\]

\noindent where $\alpha = (p_1,\act)$ for
$\act:\Gr_{\bG_m}^{pos,\leq n} \times \sZ^{\leq n} \to \sZ^{\leq n}$
the action morphism from above.

We have an evident isomorphism:\footnote{One can work with
$\IndCoh^*$ or $\QCoh$ for our purposes here. 
We chose
$\IndCoh^*$ for the sake of definiteness. 

We also have
written $(-)^*$ in place of $(-)^{*,\IndCoh}$ to keep
the notation simpler. To be clear, this notation refers to
the left adjoint to the $\IndCoh^*$ pushforward functor.}
\[
\omega_{\Gr_{\bG_m}^{pos,\leq n}} \boxtimes 
\sO_{\sZ^{\leq n}} = 
p_1^*(\omega_{\Gr_{\bG_m}^{pos,\leq n}}) \simeq
\alpha^*p_1^*(\omega_{\Gr_{\bG_m}^{pos,\leq n}}) \in 
\IndCoh^*(\Gr_{\bG_m}^{pos,\leq n} \times \sZ^{\leq n})
\]

\noindent giving a morphism:
\[
\omega_{\Gr_{\bG_m}^{pos,\leq n}} \boxtimes 
\sO_{\sZ^{\leq n}} \to 
\alpha_*^{\IndCoh}(\omega_{\Gr_{\bG_m}^{pos,\leq n}} \boxtimes 
\sO_{\sZ^{\leq n}}).
\]

\noindent Applying $p_{2,*}^{\IndCoh}$, we obtain a canonical
morphism:
\[
\Gamma^{\IndCoh}(\omega_{\Gr_{\bG_m}^{pos,\leq n}}) \otimes  
\sO_{\sZ^{\leq n}} \to 
p_{2,*}^{\IndCoh}
\alpha_*^{\IndCoh}(\omega_{\Gr_{\bG_m}^{pos,\leq n}} \boxtimes 
\sO_{\sZ^{\leq n}}) = 
\act_*^{\IndCoh}(\omega_{\Gr_{\bG_m}^{pos,\leq n}} \boxtimes 
\sO_{\sZ^{\leq n}}).
\]

We also have a canonical adjunction morphism
$\gamma_*^{\IndCoh}(\omega_{\Gr_{\bG_m}^{pos,\leq n}}) \to 
\omega_{\Gr_{\bG_m}}$ for $\gamma:\Gr_{\bG_m}^{pos,\leq n} \to 
\Gr_{\bG_m}^{\leq n}$ the (ind-closed) embedding.

Pushing forward along $i_n:\sZ^{\leq n} \to \sY_{\log}^{\leq n}$
and composing, we obtain a canonical morphism:
\[
\Gamma^{\IndCoh}(\omega_{\Gr_{\bG_m}^{pos,\leq n}}) \otimes  
i_{n,*}^{\IndCoh}(\sO_{\sZ^{\leq n}}) \to 
\act_*^{\IndCoh}(\omega_{\Gr_{\bG_m}}^{\leq n} \boxtimes 
i_{n,*}^{\IndCoh}(\sO_{\sZ^{\leq n}}) = 
\Oblv \Av_!^{\Gr_{\bG_m}^{\leq n},w} 
i_{n,*}(\sO_{\sZ^{\leq n}}).
\]

% check notation on Av_!. 
% And should we emphasize averaging for \Gr_{\bG_m}^{\leq n}?

By Proposition \ref{p:cc-log}, we have a canonical isomorphism:
\[
\Gamma^{\IndCoh}(\omega_{\Gr_{\bG_m}^{pos,\leq n}}) \simeq
\Gamma(\fL_n^+ \bA^1,\sO_{\fL_n^+ \bA^1}).
\]

\noindent Therefore, by adjunction, the above gives a morphism:
\[
\begin{gathered}
\Gamma(\fL_n^+ \bA^1,\sO_{\fL_n^+ \bA^1}) \to 
\ul{\Hom}_{\IndCoh^*(\sY_{\log}^{\leq n})}
\big(i_{n,*}(\sO_{\sZ^{\leq n}}),
\Oblv\Av_!^{\Gr_{\bG_m}^{\leq n},w}  i_{n,*}(\sO_{\sZ^{\leq n}})\big) = \\
\ul{\End}_{\IndCoh^*(\sY^{\leq n})}
\big(\Av_!^{\Gr_{\bG_m}^{\leq n},w}  i_{n,*}(\sO_{\sZ^{\leq n}})\big).
\end{gathered}
\]

\noindent By \emph{loc. cit}., this is a morphism
of algebras. 

For $\omega \in t^{-n}k[[t]]dt/k[[t]]dt = (k[[t]]/t^n)^{\vee}$,
there is a corresponding linear function on $\fL_n^+ \bA^1$;
its image under the above map is by definition $\vph_{\omega}$ in 
\eqref{eq:fns}. As the above map extended to the symmetric
algebra on this vector space, we have verified that
the operators $\vph_{\omega}$ commute.

\begin{rem}

More evocatively, but a little less rigorously, we have algebra maps:
\[
\Gr_{\bG_m}^{pos,\leq n} \to \Maps_{/\sY^{\leq n}}(\sZ^{\leq n},\sZ^{\leq n}) \to 
\ul{\End}_{\IndCoh^*(\sY^{\leq n})}(\sF_n)
\]

\noindent where the first map is given by the action and the second 
map sends a map $f:\sZ^{\leq n} \to \sZ^{\leq n}$ over $\sY^{\leq n}$ 
to the map: 
\[
\sF_n = \pi_{n,*}^{\IndCoh}(\sO_{\sZ^{\leq n}}) 
\to \pi_{n,*}^{\IndCoh}f_*^{\IndCoh}(\sO_{\sZ^{\leq n}}) = 
\pi_{n,*}^{\IndCoh}(\sO_{\sZ^{\leq n}}) = \sF_n
\]

\noindent where $\pi_n$ is the ind-proper morphism $\sZ^{\leq n} \to \sY^{\leq n}$ 
and the first map comes by functoriality from the canonical\footnote{This
map is tautologically equivalent via $\Psi$ to the adjunction morphism
$\sO_{\sZ} \to f_*^{\IndCoh}(\sO_{\sZ^{\leq n}}) \in \QCoh(\sZ^{\leq n})$.}
map $\sO_{\sZ^{\leq n}} \to f_*^{\IndCoh}(\sO_{\sZ^{\leq n}})$.

We then obtain an algebra map:
\[
\Gamma(\fL_n^+ \bA^1,\sO_{\fL_n^+ \bA^1}) =
\Gamma^{\IndCoh}(\omega_{\Gr_{\bG_m}^{pos,\leq n}}) \to 
\ul{\End}_{\IndCoh^*(\sY^{\leq n})}(\sF_n)
\]

\noindent from \eqref{eq:integration}. The detailed construction articulates
this idea explicitly.

\end{rem}

\subsection{Action of vector fields}

We now define the map $f \mapsto \xi_f$ from
\eqref{eq:vfs}.

\subsubsection{}

We begin with an elementary observation
involving $C = \Spec(k[x,y]/xy)$.
Let $i:\bA_y^1 \into C$ denote the 
embedding of the $y$-axis.

The map:
\[
x\cdot-:\sO_C \to \sO_C
\]

\noindent obviously factors through a map: 
\[
i_*\sO_{\bA_y^1} \to \sO_C
\]

\noindent that we also denote $y \cdot -$.

\subsubsection{}\label{ss:c-ind}

We now vary to above to incorporate $\bG_m$-actions.

Consider $\bG_m \times \bG_m$ acting on
$C$ with the first $\bG_m$ scaling the $x$-axis
and the second $\bG_m$ scaling the $y$-axis.

The map $y\cdot-$ evidently has bidegree
$(0,1)$.
Therefore, on the stack $C/\bG_m\times\bG_m$, we obtain
the following. 

By definition, there are tautological line bundles 
$\sL_1$ and $\sL_2$ on $C/\bG_m\times \bG_m$
equipped with sections $x \in \sL_1$ and 
$y \in \sL_2$ such that
$xy = 0$ as a section of $\sL_1 \otimes \sL_2$. 

From the above, we obtain a canonical map:
\begin{equation}\label{eq:sections-ann}
y\cdot -: i_*\sO_{0/\bG_m \times \bA_y^1/\bG_m} \to \sL_2.
\end{equation}

\subsubsection{}

Before applying the above, we introduce some more
notation.

Let $\Gr_{\bG_m}^{neg} \subset \Gr_{\bG_m}$ denote
the image of $\Gr_{\bG_m}^{pos}$ under the inversion
map $\Gr_{\bG_m} \isom \Gr_{\bG_m}$. 
We use similar notation for $\Gr_{\bG_m}^{neg,\leq n}$.

Note that $\AJ^{-1}:\widehat{\sD} \to \Gr_{\bG_m}$
maps into $\Gr_{\bG_m}^{neg}$. Similarly,
$\AJ^{-1}$ maps $\sD^{\leq n}$ into 
$\Gr_{\bG_m}^{neg,\leq n}$.

We let $\act_{neg}:\Gr_{\bG_m}^{neg} \times \sY_{\log}
\to \sY_{\log}$ denote the action map.
We use evident variants of this as well; crucially,
we also let $\act_{neg}$ denote the composition:
\[
\sD^{\leq n} \times \sZ^{\leq n} 
\xar{\AJ^{-1} \times i_n} 
\Gr_{\bG_m}^{neg} \times \sY_{\log} \xar{\act_{neg}}
\sY_{\log}.
\] 

\subsubsection{}\label{ss:act-inv}

Now recall from Proposition \ref{p:z-flat-axes-var}
that we have a canonical map:
\[
\sD^{\leq n} \times \sZ^{\leq n} \to 
C/\bG_m \times \bG_m.
\]

In the notation of \S \ref{ss:c-ind},
the line bundle $\sL_2$ pulls
back to 
$\omega_{\sD^{\leq n}} \boxtimes \sO_{\sZ^{\leq n}}$  
by construction. The pullback of its section $y$ 
corresponds to the canonical map:
\[
\sZ^{\leq n} \to \LocSys_{\bG_m,\log}^{\leq n}
\xar{\Pol}
\fL^{pol,\leq n} \bA^1 dt = 
\Gamma(\sD^{\leq n},\omega_{\sD^{\leq n}})  
\] 

\noindent (with the last term really meaning the
scheme attached to this vector space). We sometimes
denote this section as $\omega^{univ}$ in what
follows.

We let 
$\act_{neg}^{-1}(\sZ^{\leq n}) \subset 
\sD^{\leq n} \times \sZ^{\leq n}$ denote the fiber
product:
\[
(\sD^{\leq n} \times \sZ^{\leq n})
\underset{C/\bG_m\times\bG_m}{\times}
(0/\bG_m \times \bA_y^1/\bG_m).
\]

\noindent (See \S \ref{ss:act-neg-rems} for
an explanation of the notation.) 
By Proposition \ref{p:z-flat-axes-var}, the
derived fiber product is classical.\footnote{This
fact is psychologically convenient, but not literally
necessary in what follows. I.e., a straightforward
restructuring of the discussion that follows could avoid 
directly appealing to this fact.} We let
$\alpha$ denote the embedding of
$\act_{neg}^{-1}(\sZ^{\leq n})$ into 
$\sD^{\leq n} \times \sZ^{\leq n}$.

Pulling back \eqref{eq:sections-ann}, we obtain 
a canonical map:
\begin{equation}\label{eq:supp}
\alpha_*^{\IndCoh}
(\sO_{\act_{neg}^{-1}(\sZ^{\leq n})}) \to 
\omega_{\sD^{\leq n}} \boxtimes \sO_{\sZ^{\leq n}}
\in \IndCoh^*(\sD^{\leq n} \times \sZ^{\leq n})^{\heart}
\simeq \QCoh(\sD^{\leq n} \times \sZ^{\leq n})^{\heart}.
\end{equation}

\noindent (As the notation indicates, the
superscript $\IndCoh$ on $\alpha_*$ is a matter
of perspective: it is convenient for us for later 
purposes to view this morphism occurring in $\IndCoh^*$
as opposed to $\QCoh$.)

\begin{rem}

In other words, the section $\omega^{univ}$
is scheme-theoretically supported on 
$\act_{neg}^{-1}(\sZ^{\leq n})$.

\end{rem}

\subsubsection{}\label{ss:act-neg-rems}

We now claim that $\act_{neg}^{-1}(\sZ^{\leq n})$
is the \emph{classical} (not derived!) fiber
product:
\begin{equation}\label{eq:act-neg-rems}
\act_{neg}^{-1}(\sZ^{\leq n}) = 
\Big(
(\sD^{\leq n} \times \sZ^{\leq n})
\underset{\sY_{\log}^{\leq n}}{\times}
\sZ^{\leq n} 
\Big)^{cl} \subset \sD^{\leq n} \times \sZ^{\leq n}.
\end{equation}

\noindent Here the morphism 
$\sD^{\leq n} \times \sZ^{\leq n} \to \sY_{\log}^{\leq n}$
is $\act_{neg}$.

Indeed, we have seen that 
$\act_{neg}^{-1}(\sZ^{\leq n})$
is a classical stack, so it suffices to check
the above identity on $A$-points for classical
commutative rings $A$, i.e., we are free to 
manipulate our $A$-points in the most naive way.

To calculate $\act_{neg}^{-1}(\sZ^{\leq n})$,
we need to recall from 
Proposition \ref{p:z-flat-axes-var} that
in the notation of \S \ref{ss:c-ind}, the
pullback of the line bundle $\sL$ comes from
the evident universal line bundle
on $\widehat{\sD} \times \sZ$; i.e.,
its fiber at a point 
$(\tau,(\sL,\nabla,s)) \in \widehat{\sD} \times \sZ$ 
is $\sL|_{\tau}$. Its canonical section,
denoted $x$ in \S \ref{ss:c-ind}, corresponds
to $s|_{\tau} \in \sL|_{\tau}$. 

Therefore, as a classical prestack,
$\act_{neg}^{-1}(\sZ^{\leq n}) \subset 
\sD^{\leq n} \times \sZ^{\leq n}$
corresponds to those data
$(\tau,(\sL,\nabla,s))$ with $s \in \sL(-\tau)$.

Clearly this description exactly matches
\eqref{eq:act-neg-rems}.

\begin{rem}

The derived version of the
fiber product \eqref{eq:act-neg-rems} is easily
seen to differ from 
$\act_{neg}^{-1}(\sZ^{\leq n})$, so our notation
is a bit abusive from our perspective, which
emphasizes derived geometry. (However, we do
not feel so bad about this as we have provided
another derivedly good definition of the space.)

\end{rem}

\subsubsection{}\label{ss:can-constr}

We now construct a canonical map:
\[
\on{can}:
i_{n,*}^{\IndCoh}\sO_{\sZ^{\leq n}} \to 
\act_{neg,*}^{\IndCoh}
(\omega_{\sD^{\leq n}} \boxtimes 
i_{n,*}\sO_{\sZ^{\leq n}}) \in 
\IndCoh^*(\sY_{\log}^{\leq n})
\]

\noindent as follows. 

By the above, we have a commutative diagram:
\[
\xymatrix{
\act_{neg}^{-1}(\sZ^{\leq n}) \ar[d]^{\act_{neg}^{\prime}} 
\ar[rr]^{\alpha} && 
\sD^{\leq n} \times \sZ^{\leq n} \ar[d]^{\act_{neg}} \\ 
\sZ^{\leq n} \ar[rr]^{i_n} && \sY_{\log}^{\leq n}.
}
\]

Now applying $i_{n,*}^{\IndCoh}$
to the evident counit map, we obtain:
\[
\begin{gathered}
i_{n,*}^{\IndCoh}\sO_{\sZ^{\leq n}} \to 
i_{n,*}^{\IndCoh} \act_{neg,*}^{\prime,\IndCoh}
\act_{neg}^{\prime,*,\IndCoh}(\sO_{\sZ^{\leq n}}) = 
i_{n,*}^{\IndCoh} \act_{neg,*}^{\prime,\IndCoh}
\sO_{\act_{neg}^{-1}(\sZ^{\leq n})} = \\
\act_{neg,*}^{\IndCoh} \alpha_*^{\IndCoh}
\sO_{\act_{neg}^{-1}(\sZ^{\leq n})}.
\end{gathered}
\]

\noindent We then applying $\act_{neg,*}^{\IndCoh}$
to \eqref{eq:supp} to obtain a map:
\[
\act_{neg,*}^{\IndCoh}\alpha_*^{\IndCoh}
(\sO_{\act_{neg}^{-1}(\sZ^{\leq n})}) \to 
\act_{neg,*}^{\IndCoh}
(\omega_{\sD^{\leq n}} \boxtimes \sO_{\sZ^{\leq n}}).
\]

\noindent Composing these two maps, we obtain
the desired map $\on{can}$.

\subsubsection{}

More generally, suppose $f \in k[[t]]/t^n$.

There is an associated action map 
$f\cdot -:\omega_{\sD^{\leq n}} \to 
\omega_{\sD^{\leq n}} \in \QCoh(\sD^{\leq n})$.
By functoriality, this gives rise to a morphism:
\[
f\cdot-:\act_{neg,*}^{\IndCoh}
(\omega_{\sD^{\leq n}} \boxtimes \sO_{\sZ^{\leq n}})
\to \act_{neg,*}^{\IndCoh}
(\omega_{\sD^{\leq n}} \boxtimes \sO_{\sZ^{\leq n}}).
\]

We then define: 
\[
\on{can}_f:
i_{n,*}^{\IndCoh}\sO_{\sZ^{\leq n}} \to 
\act_{neg,*}^{\IndCoh}
(\omega_{\sD^{\leq n}} \boxtimes 
i_{n,*}\sO_{\sZ^{\leq n}})
\]

\noindent as the composition of $\on{can} (= \on{can}_1)$
with $f\cdot-$.

\subsubsection{}

We now conclude the construction. 
Fix $f$ as above.

By adjunction,
there is a canonical morphism
$\AJ_*^{-1,\IndCoh}(\omega_{\sD^{\leq n}}) \to 
\omega_{\Gr_{\bG_m}^{\leq n}}$. 

Therefore, $\on{can}_f$ gives rise to a morphism:
\[
i_{n,*}^{\IndCoh}\sO_{\sZ^{\leq n}} \to 
\act_{neg,*}^{\IndCoh}
(\omega_{\Gr_{\bG_m}^{\leq n}} \boxtimes 
i_{n,*}\sO_{\sZ^{\leq n}}) = 
\Oblv\Av_!^{\Gr_{\bG_m}^{\leq n},w} 
(i_{n,*}\sO_{\sZ^{\leq n}}).
\]

\noindent By adjunction, we obtain the desired
morphism:
\[
\xi_f:\sF_n \to \sF_n.
\]

\subsubsection{}

Next, we show that the endomorphisms
$\xi_f$ of $\sF_n$ mutually commute.

By construction,\footnote{See \S \ref{ss:sigma-comp} for a much more general
setup.}
the assertion immediately reduces to the following one.

\begin{lem}

The section:
\[
p_{13}^*(\omega^{univ}) \otimes p_{23}^*(\omega^{univ}) \in 
p_1^*(\omega_{\sD^{\leq n}}) \otimes p_2^*(\omega_{\sD^{\leq n}}) =
\omega_{\sD^{\leq n}} \boxtimes \omega_{\sD^{\leq n}} 
\boxtimes \sO_{\sZ^{\leq n}}
\]

\noindent is $\bZ/2$-equivariant (for the natural action
permuting the first two factors).

\end{lem}

\begin{proof}

There is a canonical morphism:
\[
\sD^{\leq n} \times \sD^{\leq n} \times \fL_n^+ \bA^1/\bG_m
\]

\noindent given as follows. To a point of the left hand side 
defined by
$\tau_1 \in \sD^{\leq n}$, $\tau_2 \in \sD^{\leq n}$,
a line bundle $\sL$ on $\sD^{\leq n}$ and a section 
$s \in \sL$, our map assigns the line
$\sL|_{\tau_1} \otimes \sL|_{\tau_2}$ with
its section $s|_{\tau_1}\otimes s|_{\tau_2}$.

In fact, this map factors through a map:
\[
\Sym^2(\sD^{\leq n}) \times \fL_n^+ \bA^1/\bG_m.
\]

\noindent Indeed, for a point $(D,(\sL,s))$ of the left hand
side, we may regard $D$ as a finite subscheme of $\sL$.
The map then sends the datum to 
$(\Lambda^2(s) \in \Lambda^2(\sL|_D)) \in \bA^1/\bG_m$.
I.e., this is evident by the usual norm construction.

Unwinding the constructions, this implies the claim.

\end{proof}

\begin{rem}

We could have organized the discussion differently. 
Instead, we could have generalized our
construction of $\xi_f$ with $\Gr_{\bG_m}^{neg,\leq n}$ replacing
$\sD^{\leq n}$ via a suitable use of norms (as in the lemma
above). Then we would immediately deduce commutativity of the
operators $\xi_f$ by the same argument and for the operators
$\vph_{\omega}$.

\end{rem}

\subsection{Uncertainty}

We now check the remaining Weyl algebra relation
\eqref{eq:uncertainty}.

\subsubsection{}\label{ss:gpd}

First, let us put the above constructions a common framework to allow us to 
compute the relevant compositions.

Let $\sH^{\leq n} \coloneqq \sZ^{\leq n} \times_{\sY^{\leq n}} \sZ^{\leq n}$.
This stack is a groupoid over $\sZ^{\leq n}$ with its projections
$p_1,p_2:\sH^{\leq n} \rightrightarrows \sZ^{\leq n}$ ind-proper.

By adjunction and base-change, morphisms $\sF_n \to \sF_n$ are equivalent to
sections of $p_1^!(\sO_{\sZ^{\leq n}})$. For\footnote{For $V \in \Vect$,
$v\in V$ means $v \in \Omega^{\infty} V$, i.e., we have a point of the
underlying $\infty$-groupoid. If we concretely model $V$ as a cochain
complex, we obtain such data from cycles in degree $0$.}
$\sigma \in \Gamma^{\IndCoh}(\sH^{\leq n},p_1^!(\sO_{\sZ^{\leq n}}))$, 
we let $\psi_{\sigma}:\sF_n \to \sF_n$ denote the corresponding morphism.

In particular, suppose we are given a prestack $S$ with an ind-proper morphism
$\eta:S \to \sH^{\leq n}$. Given a section of 
$\eta^!p_1^!(\sO_{\sZ^{\leq n}})$, we obtain a section of 
$p_1^!(\sO_{\sZ^{\leq n}})$ by functoriality, hence a morphism
as above.

\begin{example}\label{e:fns}

Let $S = \sD^{\leq n} \times \sZ^{\leq n}$ and let
$\eta:S \to \sH^{\leq n}$
be the morphism: 
\[
(p_2,\act\circ (\AJ \times \id)):
(\tau,(\sL,\nabla,s) \mapsto ((\sL,\nabla,s),(\sL(\tau),\nabla,s)) \in 
\sZ^{\leq n} \underset{\sY^{\leq n}}{\times} \sZ^{\leq n} = 
\sH^{\leq n}.
\]

For $\omega \in (k[[t]]/t^n)^{\vee} \overset{Prop. \ref{p:cc-log}}{=} 
\Gamma^{\IndCoh}(\sD^{\leq n},\omega_{\sD^{\leq n}})$, we
obtain the evident section $p_1^*(\omega)$ of:
\[
\omega_{\sD^{\leq n}} \boxtimes \sO_{\sZ^{\leq n}} = 
p_2^! \sO_{\sZ^{\leq n}}  = 
\eta^!p_1^!(\sO_{\sZ^{\leq n}}).
\]

\noindent The corresponding map $\sF_n \to \sF_n$ is
\eqref{eq:fns}.

\end{example}

\begin{example}\label{e:vfs}

In the notation of \S \ref{ss:act-inv}, let 
$S = \act_{neg}^{-1}(\sZ^{\leq n}) \subset \sD^{\leq n} \times \sZ^{\leq n}$.
Let $\eta$ be the evident morphism defined using \eqref{eq:act-neg-rems}.\footnote{In the notation of \S \ref{ss:can-constr}, this 
morphism is given by $(p_1\alpha,\act_{neg}^{\prime})$.}
The construction of $\on{can}$ (resp. $\on{can}_f$) amounted to showing
that the section $\omega^{univ}$ (hence $p_1^*(f)\cdot \omega^{univ}$) 
comes from a canonical section of
$\eta^!p_1^!(\sO_{\sZ^{\leq n}})$ (i.e., these sections are scheme-theoretically
supported on $\act_{neg}^{-1}(\sZ^{\leq n}) $). By construction,
the corresponding map $\sF_n \to \sF_n$ is \eqref{eq:vfs}.

\end{example}

\begin{rem}\label{r:h-cl}

Although 
$\sH^{\leq n}$ is not classical, our constructions
here are not sensitive to this. The reason is that
we have a closed embedding $\iota:\sH^{\leq n} \into 
\Gr_{\bG_m}^{\leq n} \times \sZ^{\leq n}$ 
(cf. \S \ref{ss:h-embed} below). As 
$\Gr_{\bG_m}^{\leq n}$ is ind-finite and 
$\sZ^{\leq n}$ is classical, 
$\Gamma(p_1^!(\sO_{\sZ^{\leq n}})) = \Gamma(\iota^!(\omega_{\Gr_{\bG_m}^{\leq n}}
\boxtimes \sO_{\sZ^{\leq n}}))$ is a coconnective complex,
and its $H^0$ is the same for $\sH^{\leq n}$ and for its
underlying classical stack (both identify with 
sections of $\omega_{\Gr_{\bG_m}^{\leq n}}
\boxtimes \sO_{\sZ^{\leq n}}$ scheme-theoretically supported
on $\sH^{\leq n,cl}$).

\end{rem}

\subsubsection{}\label{ss:sigma-comp}

We now spell out how to concretely compose morphisms of the above type.

We have three morphisms:
\[
p_{12},p_{23},p_{13}:
\sH^{\leq n} \underset{\sZ^{\leq n}}{\times} \sH^{\leq n} = 
\sZ^{\leq n} \underset{\sY^{\leq n}}{\times}
\sZ^{\leq n} \underset{\sY^{\leq n}}{\times} \sZ^{\leq n}
\to 
\sZ^{\leq n} \underset{\sY^{\leq n}}{\times} \sZ^{\leq n} =
\sH^{\leq n}.
\]

\noindent We remark that $p_{13}$ corresponds to the multiplication
on the groupoid.

%Let $m:\sH^{\leq n} \times_{\sZ^{\leq n}} \sH^{\leq n} \to \sH^{\leq n}$ be the 
%multiplication for the groupoid (so the maps 
%$\sH^{\leq n} \to \sZ^{\leq n}$ are $p_2$ and $p_1$ respectively). 
We begin by claiming that there is a canonical 
isomorphism: % this is an iso where???

\begin{equation}\label{eq:p13-!}
p_{13}^!p_1^!(\sO_{\sZ^{\leq n}}) \simeq 
p_{12}^*p_1^!(\sO_{\sZ^{\leq n}}) \otimes 
p_{23}^*p_1^!(\sO_{\sZ^{\leq n}}). 
\end{equation}

\noindent To see this, first observe that
$p_1p_{13} = p_1p_{12}$, so:
\[
p_{13}^!p_1^!(\sO_{\sZ^{\leq n}}) = 
p_{12}^!p_1^!(\sO_{\sZ^{\leq n}}).
\]

Next, we have evident isomorphisms:
\[
\begin{gathered}
p_{12}^!p_1^! \sO_{\sZ^{\leq n}} = 
p_{12}^!(p_1^! \sO_{\sZ^{\leq n}}
\underset{\sO_{\sH^{\leq n}}}{\otimes} 
\sO_{\sH^{\leq n}}) = 
p_{12}^*p_1^! \sO_{\sZ^{\leq n}}
\underset{\sH^{\leq n} \underset{\sZ^{\leq n}} \sH^{\leq n} }
{\otimes} 
p_{12}^!\sO_{\sH^{\leq n}}.
\end{gathered}
\]

\noindent Then the Cartesian diagram:
\[
\xymatrix{
\sH^{\leq n} \underset{\sZ^{\leq n}}{\times} \sH^{\leq n} \ar[rr]^{p_{12}} 
\ar[d]_{p_{23}} && \sH^{\leq n} \ar[rr]^{p_1} \ar[d]^{p_2} && 
\sZ^{\leq n} \\
\sH^{\leq n} \ar[rr]^{p_1} && \sZ^{\leq n}
}
\]

\noindent yields an isomorphism: % reference?
\[
p_{12}^!\sO_{\sH^{\leq n}} = 
p_{12}^!p_2^*\sO_{\sZ^{\leq n}} = 
p_{23}^*p_1^!\sO_{\sZ^{\leq n}}.
\]

\noindent Combining the above isomorphims yields \eqref{eq:p13-!}.

Therefore, given $\sigma_1, \sigma_2$ sections of
$p_1^!(\sO_{\sZ^{\leq n}})$, we obtain a section:
\[
p_{12}^*(\sigma_1) \otimes p_{23}^*(\sigma_2)
\]

\noindent of $p_{13}^!p_1^!(\sO_{\sZ^{\leq n}})$
(using \eqref{eq:p13-!}). As $p_{13}$ is ind-proper,
we obtain an induced section $\sigma_1 \star \sigma_2$ 
of $p_1^!(\sO_{\sZ^{\leq n}})$. 

Then unwinding constructions, we find:
\[
\psi_{\sigma_2} \psi_{\sigma_1} = \psi_{\sigma_1 \star \sigma_2}.
\]

% ordering?

\subsubsection{}\label{ss:h-embed}

To analyze the composition above, 
it is convenient to embed our spaces as follows.

First, define a morphism:
\[
\iota:\sH^{\leq n} \into \Gr_{\bG_m}^{\leq n} \times \sZ^{\leq n}
\]

\noindent as:
\[
\sH^{\leq n} = 
\sZ^{\leq n} \underset{\sY^{\leq n}}{\times} \sZ^{\leq n} 
\into 
\sZ^{\leq n} \underset{\sY^{\leq n}}{\times} \sY_{\log}^{\leq n}
\overset{\simeq}{\leftarrow} \Gr_{\bG_m}^{\leq n} \times \sZ^{\leq n} 
\]

\noindent where the last arrow is $(p_2,\act)$. Note that
the composition of $\iota$ with the projection to $\sZ^{\leq n}$
is the projection $p_1:\sH^{\leq n} \to \sZ^{\leq n}$.

Next, we define a similar morphism:
\[
\widetilde{\iota}:
\sH^{\leq n} \underset{\sZ^{\leq n}}{\times}
\sH^{\leq n} \into 
\Gr_{\bG_m}^{\leq n} \times \Gr_{\bG_m}^{\leq n} \times \sZ^{\leq n}
\]

\noindent as:
\[
\sH^{\leq n} \underset{\sZ^{\leq n}}{\times}
\sH^{\leq n} = 
\sZ^{\leq n} \underset{\sY^{\leq n}}{\times} \sZ^{\leq n} 
\underset{\sY^{\leq n}}{\times} \sZ^{\leq n} 
\into 
\sZ^{\leq n} \underset{\sY^{\leq n}}{\times} \sY_{\log}^{\leq n}
\underset{\sY^{\leq n}}{\times} \sY_{\log}^{\leq n}
\overset{\simeq}{\leftarrow} 
\Gr_{\bG_m}^{\leq n} \times \Gr_{\bG_m}^{\leq n} \times \sZ^{\leq n} 
\]

\noindent where this time the last arrow is 
$(p_3,\act \circ p_{23}, \act \circ (\id \times \act))$. 

As in Remark \ref{r:h-cl}, a section of 
$\Gamma^{\IndCoh}(\sH^{\leq n},p_1^!(\sO_{\sZ^{\leq n}}))$
is the same as a section of 
$\omega_{\Gr_{\bG_m}^{\leq n}} \boxtimes \sO_{\sZ^{\leq n}}$
scheme-theoretically supported on $\sH^{\leq n}$ 
(equivalently, on the underlying classical stack).
The same applies for 
$\widetilde{\iota}$: a section
of $p_{13}^!p_1^!(\sO_{\sZ^{\leq n}})$ is the same
as a section of:
\[
\omega_{\Gr_{\bG_m}^{\leq n}} \boxtimes \omega_{\Gr_{\bG_m}^{\leq n}} \boxtimes \sO_{\sZ^{\leq n}}
\]

\noindent scheme-theoretically supported on 
$\sH^{\leq n} \times_{\sZ^{\leq n}} \sH^{\leq n}$ (which is
equivalent to saying at the 
classical level).

\subsubsection{}\label{sss:recipe}

Now fix $\omega \in (k[[t]]/t^n)^{\vee}$ and $f \in k[[t]]/t^n$.
In effect, \S \ref{ss:sigma-comp} constructed two sections of
$p_{13}^!p_1^!(\sO_{\sZ^{\leq n}})$, which
induce the compositions 
$\xi_f\vph_{\omega}$ and $\vph_{\omega}\xi_f$ (after 
pushforward along $p_{13}$).

Let us explicitly describe the resulting sections of 
$\omega_{\Gr_{\bG_m}^{\leq n}} \boxtimes \omega_{\Gr_{\bG_m}^{\leq n}} 
\boxtimes \sO_{\sZ^{\leq n}}$ under the
above dictionary.

For $\vph_{\omega}\xi_f$, the corresponding section
is supported on:
\[
\sD^{\leq n} \times \sD^{\leq n} \times \sZ^{\leq n}
\overset{\AJ \times \AJ^{-1} \times \id}{\into}
\Gr_{\bG_m}^{\leq n} \times \Gr_{\bG_m}^{\leq n} 
\times \sZ^{\leq n}.
\]

\noindent The corresponding section of:
\[
\omega_{\sD^{\leq n}} \boxtimes 
\omega_{\sD^{\leq n}} \boxtimes \sO_{\sZ^{\leq n}}
\]

\noindent is $\omega \boxtimes p_2^*(f)\cdot\omega^{univ}$,
with notation as in Examples \ref{e:fns} and \ref{e:vfs}.\footnote{As
we have shown, this section is scheme-theoretically supported
on $\sD^{\leq n} \boxtimes \act_{neg}^{-1}(\sZ^{\leq n})$,
hence on $(\sH^{\leq n} \times_{\sZ^{\leq n}}\sH^{\leq n})^{cl}$.} It is convenient to 
rewrite this section as:
\[
p_2^*(f) \cdot (p_1^*(\omega) \otimes p_{23}^*(\omega^{univ})) \in 
p_1^*(\omega_{\sD^{\leq n}}) \otimes 
p_{23}^*(\omega_{\sD^{\leq n}} \boxtimes \sO_{\sZ^{\leq n}}).
\]

For $\xi_f\vph_{\omega}$, the corresponding section
is supported on:
\[
\sD^{\leq n} \times \sD^{\leq n} \times \sZ^{\leq n}
\xar{\AJ^{-1} \times \AJ \times \id}
\Gr_{\bG_m}^{\leq n} \times \Gr_{\bG_m}^{\leq n} 
\times \sZ^{\leq n}.
\]

\noindent The corresponding section of:
\[
\omega_{\sD^{\leq n}} \boxtimes 
\omega_{\sD^{\leq n}} \boxtimes \sO_{\sZ^{\leq n}}
\]

\noindent is:\footnote{For clarity, we again
highlight that we have shown
this section is in fact supported on 
$(\sH^{\leq n} \times_{\sZ^{\leq n}}\sH^{\leq n})^{cl}$, as
it should be. Indeed, at the classical level, our previous
analysis shows that it is
supported on 
$(\act_{neg}\circ (\id\times\act))^{-1} (\sZ^{\leq n}) 
\subset  (\sH^{\leq n} \times_{\sZ^{\leq n}}\sH^{\leq n})^{cl}$.}
\[
\begin{gathered}
p_1^*(f) \cdot 
(p_2^*(\omega) \otimes (p_1\times \act)^*(\omega^{univ})) \in 
p_2^*(\omega_{\sD^{\leq n}}) \otimes 
(p_1\times \act)^*(\omega_{\sD^{\leq n}} \boxtimes \sO_{\sZ^{\leq n}}) = \\
\omega_{\sD^{\leq n}} \boxtimes \omega_{\sD^{\leq n}} 
\boxtimes \sO_{\sZ^{\leq n}}.
\end{gathered}
\]

Now note that we have a commutative diagram:
\[
\xymatrix{
\sH^{\leq n} \underset{\sZ^{\leq n}}{\times} \sH^{\leq n} 
\ar[rr]^{\widetilde{\iota}} \ar[d]_{p_{13}} &&
\Gr_{\bG_m}^{\leq n} \times \Gr_{\bG_m}^{\leq n} \times \sZ^{\leq n} 
\ar[d]^{m\times \id} \\
\sH^{\leq n} \ar[rr]^{\iota} && \Gr_{\bG_m}^{\leq n} \times \sZ^{\leq n}
}
\]

\noindent for $m$ the multiplication on $\Gr_{\bG_m}^{\leq n}$.
As this multiplication is commutative, we see that
$\xi_f\vph_{\omega}$ could as well have been defined
by the section as above obtained by transposing
the first two coordinates, i.e.:
\[
\begin{gathered}
p_2^*(f) \cdot 
(p_1^*(\omega) \otimes (p_2,\act \circ p_{13})^*(\omega^{univ})).
\end{gathered}
\]

\noindent This section has the advantage of
sharing its support with our previous one for 
$\vph_{\omega}\xi_f$.

\subsubsection{}

We can now conclude the argument. We use the next elementary lemmas.

\begin{lem}\label{l:cotr-1}

Let: 
\[
\cotr \in 
\Gamma(\sD^{\leq n},\sO_{\sD^{\leq n}}) \otimes
\Gamma(\sD^{\leq n},\sO_{\sD^{\leq n}})^{\vee} = 
\Gamma^{\IndCoh}(\sD^{\leq n} \times \sD^{\leq n},
\sO_{\sD^{\leq n}} \boxtimes \omega_{\sD^{\leq n}})
\]

\noindent denote the canonical vector. 

We then have an equality:
\[
p_{23}^*(\omega^{univ}) = 
(p_2,\act \circ p_{13})^*(\omega^{univ}))
+p_{12}^*(\cotr)
\in 
\sO_{\sD^{\leq n}} \boxtimes 
\omega_{\sD^{\leq n}} \boxtimes \sO_{\sZ^{\leq n}}.
\]

\end{lem}

\begin{proof}

For convenience, we give the first (resp. second) factor
of our triple product as 
$\sD_{\tau_1}^{\leq n}$ (resp. $\sD_{\tau_2}^{\leq n}$), 
where $\tau_1,\tau_2$ denote the respective coordinates. 

For a prestack $S$, note that a section of 
$\omega_{\sD_{\tau_2}^{\leq n}} \boxtimes \sO_S$ is equivalent to
a map $S \to \fL^{pol,\leq n} \bA^1 dt$.

For $S = \sD_{\tau_1}^{\leq n} \times \sZ^{\leq n}$,
our sections above correspond to maps:
\[
\sD_{\tau_1}^{\leq n} \times \sZ^{\leq n} 
\xar{(\tau_1,(\sL,\nabla,s)) \mapsto ???}
\fL^{pol,\leq n} \bA^1 dt
\]

\noindent as follows:

\begin{align*}
p_{23}^*(\omega^{univ}) & \leftrightsquigarrow 
%\sD_{\tau_1}^{\leq n} \times \sZ^{\leq n} \xar{p_2} 
%\sZ^{\leq n} \xar{\Pol} \fL^{pol,\leq n} \bA^1 dt \\
(\tau_1,(\sL,\nabla,s)) \mapsto 
\Pol(\nabla) \\
(p_2,\act \circ p_{13})^*(\omega^{univ})) 
& \leftrightsquigarrow 
%\sD_{\tau_1}^{\leq n} \times \sZ^{\leq n} \xar{\act} 
%\sZ^{\leq n} \xar{\Pol} \fL^{pol,\leq n} \bA^1 dt \\
(\tau_1,(\sL,\nabla,s)) \mapsto 
\Pol((\sL(\tau_1),\nabla,s)) = \\ 
& \hspace{2cm} \Pol(\nabla)+d\log\AJ(\tau_1) \\
p_{12}^*(\cotr) & \leftrightsquigarrow 
%\sD_{\tau_1}^{\leq n} \times \sZ^{\leq n} \xar{p_1}
%\sD_{\tau_1}^{\leq n} \xar{d\log \AJ} \fL^{pol,\leq n} \bA^1 dt
(\tau_1,(\sL,\nabla,s)) \mapsto 
-d\log\AJ(\tau_1)
\end{align*}

\noindent where in the last two lines the map
$d\log\AJ$ was considered in
\S \ref{ss:dlog-aj}-\ref{ss:aj-trun}; the last line
follows from Lemma \ref{l:dlog-aj}.

We clearly obtain the assertion.

\end{proof}

We also use the following observations.

\begin{lem}\label{l:cotr-2}

The above section $\cotr$ of 
$\sO_{\sD^{\leq n}} \boxtimes \omega_{\sD^{\leq n}}$
is scheme-theoretically supported on the diagonally embedded
$\sD^{\leq n} \subset \sD^{\leq n} \times \sD^{\leq n}$.

Moreover, for $f \in k[[t]]/t^n$ and 
$\omega \in t^{-n}k[[t]]dt/k[[t]]dt$, the section:
\[
p_2^*(f) \cdot (p_1^*(\omega) \otimes \cotr) \in 
\omega_{\sD^{\leq n}} \boxtimes \omega_{\sD^{\leq n}} 
\]

\noindent maps to $\Res(f\omega) \in k$ under the
adjunction morphism:
\[
\Gamma^{\IndCoh}(\sD^{\leq n} \times \sD^{\leq n},
\omega_{\sD^{\leq n}} \boxtimes \omega_{\sD^{\leq n}} ) \to k.
\]

\end{lem}

\begin{proof}

Let $A$ be a commutative, classical, finite $k$-algebra.

As $\id_A:A \to A$ is a morphism of $A$-bimodules,
the corresponding element $\cotr \in A^{\vee} \otimes A \in 
A\otimes A\mod^{\heart}$ is scheme-theoretically 
supported on the diagonal 
$\Spec(A) \subset \Spec(A) \times \Spec(A)$.

Moreover, for $f \in A$ and $\omega \in A^{\vee}$, we
have an element $f \otimes \omega \in A \otimes A^{\vee}$.
Tensoring over $A\otimes A$ with $A^{\vee} \otimes A$,
we obtain an element:
\[
f\otimes \omega \cdot \cotr \in A^{\vee} \otimes A^{\vee}.
\]

\noindent It is easy to see that when we evaluate this
tensor on $1 \otimes 1 \in A \otimes A$, we obtain
$\omega(f) \in k$.

Taking $A = k[[t]]/t^n$, we obtain our claim.

\end{proof}

By our earlier discussion, the endomorphism 
$\xi_f \vph_{\omega}-\vph_{\omega}\xi_f$ of $\sF_n$
corresponds to the section:
\[
p_2^*(f) \cdot 
(p_1^*(\omega) \otimes (p_2,\act \circ p_{13})^*(\omega^{univ})) -
p_2^*(f) \cdot 
(p_1^*(\omega) \otimes p_{23}^*(\omega^{univ}))
\]

\noindent of $\omega_{\sD^{\leq n}} \boxtimes 
\omega_{\sD^{\leq n}} \boxtimes \sO_{\sZ^{\leq n}}$.\footnote{Of 
course we consider
$\sD^{\leq n} \times \sD^{\leq n} \times
\sZ^{\leq n}$ mapping to $\Gr_{\bG_m}^{\leq n} \times \sZ^{\leq n}$
via $(\tau_1,\tau_2,(\sL,\nabla,s)) \mapsto 
(\AJ(\tau_1)\cdot\AJ^{-1}(\tau_2),(\sL,\nabla,s))$.}

By Lemma \ref{l:cotr-1}, the above section coincides with:
\[
-p_2^*(f) \cdot p_{12}^*(p_1^*(\omega)\otimes\cotr).
\]

\noindent By Lemma \ref{l:cotr-2}, this section is supported
on the diagonally embedded:
\[
\sD^{\leq n} \times \sZ^{\leq n} 
\xar{\Delta \times \id} 
\sD^{\leq n} \times \sD^{\leq n} \times \sZ^{\leq n},
\]
 
\noindent which evidently
maps to $1 \times \sZ^{\leq n} \subset \Gr_{\bG_m}^{\leq n} \times \sZ^{\leq n}$
under $p_{13}$. Moreover, Lemma \ref{l:cotr-2} implies that the
pushforward of the above section to $\sZ^{\leq n} = 1 \times \sZ^{\leq n}$
is: 
\[
-\Res(f\omega) \in k \subset \Gamma(\sO_{\sZ^{\leq n}}) \subset 
\Gamma^{\IndCoh}(\Gr_{\bG_m}^{\leq n} \times \sZ^{\leq n}, 
\omega_{\Gr_{\bG_m}^{\leq n}} \boxtimes \sZ^{\leq n}).
\]

\noindent This amounts to \eqref{eq:uncertainty}, concluding
the argument.

%
%
%\section{Explicit analysis in the regular singular case}\label{s:rs}
%
%\subsection{}

\section{Compatibility with class field theory}\label{s:cft}

In this section, we show that the construction of \S \ref{s:weyl} is compatible
with geometric class field theory in a suitable sense. 

\subsection{Conventions regarding $D$-modules}\label{ss:diff-1}

\subsubsection{} Before proceeding, we take a moment to establish some
notational conventions regarding $D$-modules. 
We refer to \cite{crystals} for details.

\subsubsection{}

Let $X$ be a smooth variety. As in \emph{loc. cit}., 
we have the so-called \emph{left} forgetful functor 
$\Oblv^{\ell}:D(X) \to \QCoh(X)$. It is normalized
so that the functor $\Oblv^{\ell}[\dim X]$
is $t$-exact; for example, this functor sends the dualizing
$\omega_X$ to $\sO_X$ (with its standard left $D$-module structure).

\subsubsection{}

The functor $\Oblv^{\ell}$ admits the left adjoint
$\ind^{\ell}$. For us, we take as a definition that the sheaf of
differential operators on $X$ is:
\[
D_X \coloneqq \ind^{\ell}(\sO_X).
\] 

\noindent We remark that this object is concentrated in degree $-\dim X$.
(On general grounds, it coincides with $\ind^r(\omega_X)$, where 
$\ind^r$ is the right $D$-module induction functor and
$\omega_X = \Omega_X^{\dim X}[\dim X] \in \QCoh(X)$ is the
dualizing sheaf.)

\subsubsection{}

Now suppose $X$ is smooth and affine. Let $\on{Diff}(X)$ be the 
algebra of global differential operators on $X$.

As in \cite{crystals}, $\End_{D(X)}(D_X)$ canonically coincides with
$\on{Diff}(X)$, but with reversed multiplication. It follows that there
is a canonical equivalence:
\begin{equation}\label{eq:diff-dmod}
\on{Diff}(X)\mod \simeq D(X)
\end{equation}

\noindent between left modules for $\on{Diff}(X)$ and the category $D(X)$,
sending $\on{Diff}(X)$ to $D_X$. We remark that this functor is only
$t$-exact up to shift. 

\subsection{Construction of functors}

\subsubsection{}

In \S \ref{s:weyl}, we constructed compact objects
$\sF_n \in \IndCoh^*(\sY^{\leq n})^c$ and morphisms:
\[
W_n^{op} \to \ul{\End}_{\IndCoh^*(\sY^{\leq n})}(\sF_n).
\]

\noindent We obtain a corresponding functor:
\[
\ol{\Delta}_n:W_n\mod 
\overset{\eqref{eq:diff-dmod}}{\simeq} 
D(\fL_n^+\bA^1) \to \IndCoh^*(\sY_n)
\in \DGCat_{cont}
\]

\noindent uniquely characterized by sending $D_{\fL_n^+\bA^1}$ to 
$\sF_n$ compatibly with the action of $W_n^{op} = 
\TwoEnd_{D(\fL_n^+\bA^1)}(D_{\fL_n^+\bA^1})$. 

\begin{rem}

Above, we carefully normalized $D_{\fL_n^+ \bA^1}$ to lie in cohomological
degree $-n$. This is not especially relevant in our analysis 
until \S \ref{s:main}.

\end{rem}

\subsubsection{}\label{sss:delta-n}

For later use, we also use the notation 
$\Delta_n$ to denote the further composition:
\[
D(\fL_n^+\bA^1) \xar{\ol{\Delta}_n} \IndCoh^*(\sY_n) 
\xar{\lambda_{n,*}^{\IndCoh}} \IndCoh^*(\sY).
\]

\subsection{Cartier duality}

We now review some constructions from geometric class field theory.

\subsubsection{}

Let $\sK_n \subset \fL^+ \bG_m$ be the $n$th congruence subgroup, 
i.e., $\sK_n \coloneqq \Ker(\fL^+\bG_m \to \fL_n^+\bG_m)$.

Recall that there is a canonical bimultiplicative pairing:
\begin{equation}\label{eq:cartier}
\LocSys_{\bG_m}^{\leq n} \times (\fL\bG_m/\sK_n)_{dR} \to \bB \bG_m.
\end{equation}

\noindent One basic property is that its restriction along the map:
\[
\LocSys_{\bG_m}^{\leq n} \times \fL\bG_m/\sK_n
\to 
\LocSys_{\bG_m}^{\leq n} \times (\fL\bG_m/\sK_n)_{dR} \to \bB \bG_m
\]

\noindent factors through the Contou-Carr\`ere duality pairing:
\[
\LocSys_{\bG_m}^{\leq n} \times \fL\bG_m/\sK_n \to 
\bB \fL^{\leq n} \bG_m \times \fL\bG_m/\sK_n \xar{CC} 
\bB \bG_m.
\]

\noindent Here we normalize the sign in the Contou-Carr\`ere pairing
so that we have a commutative diagram:
\[
\xymatrix{
\bB \fL \bG_m \times \fL^+\bG_m \ar[rr] \ar[d] & & 
\bB \fL \bG_m \times \fL\bG_m \ar[d]^{CC} \\
\bB \Gr_{\bG_m} \times \fL^+\bG_m \ar[rr] & &
\bB \bG_m
}
\]

\noindent where the bottom arrow is induced by the pairing of 
Proposition \ref{p:cc-log}.

From \eqref{eq:cartier} and functoriality, we obtain a symmetric
monoidal Fourier-Mukai functor:
\begin{equation}\label{eq:trun-lcft}
D(\fL\bG_m/\sK_n) \to \QCoh(\LocSys_{\bG_m}^{\leq n}).
\end{equation}

\begin{thm}\label{t:trun-lcft}

The functor \eqref{eq:trun-lcft} is an equivalence.
In particular, there is a canonical fully faithful symmetric monoidal functor
$D(\fL_n^+\bG_m) \into \QCoh(\LocSys_{\bG_m}^{\leq n})$.

\end{thm}

Passing to the limit over $n$, we also obtain the following theorem
of Beilinson-Drinfeld.

\begin{thm}\label{t:full-lcft}

There is a canonical symmetric monoidal equivalence:
\[
D^*(\fL\bG_m) \isom \QCoh(\LocSys_{\bG_m}).
\]

\end{thm}

\subsection{Formulation of the equivariance property}

Now observe that $\ol{\Delta}_n$ maps a category acted on by 
$D(\fL_n^+\bG_m)$ to one acted on by 
$\QCoh(\LocSys_{\bG_m}^{\leq n})$ (via the structure map 
$\sY^{\leq n} \to \LocSys_{\bG_m}^{\leq n}$).

Using Theorem \ref{t:trun-lcft}, we can regard both sides as 
acted on by $D(\fL_n^+\bG_m)$. 

In the remainder of this section, we show that 
$\ol{\Delta}_n$ is canonically a morphism of 
$D(\fL_n^+\bG_m)$-module categories.

\subsection{Digression on Harish-Chandra data}\label{ss:hc}

We will construct the equivariance structure using the theory of \emph{Harish-Chandra
data}. We review this theory below.

In what follows, $G$ is an affine algebraic group and $A \in \Alg$ is a DG algebra.
In our applications, $A$ will be classical, so at some points in the discussion we 
will assume that.

\subsubsection{} 

First, suppose that $G$ acts on $A$ as an associative algebra. In other words,
we assume we are given a lift of $A$ along the forgetful functor 
$\Alg(\Rep(G)) \to \Alg(\Vect) = \Alg$.

This defines a canonical weak action of $G$ on $A\mod$. Its basic property is that
the forgetful functor $A\mod \to \Vect$ is weakly $G$-equivariant.

\subsubsection{}

Now suppose that $\sC$ is a DG category with a weak $G$-action. 

Recall that $\Rep(G)$ naturally acts on $\sC^{G,w}$.
Therefore, we can consider $\sC^{G,w}$ as enriched over $\Rep(G)$.
For $\sF,\sG\in \sC^{G,w}$, we let:
\[
\ul{\Hom}_{\sC}^{enh}(\sF,\sG) \in \Rep(G)
\]

\noindent denote the corresponding mapping complex.

\begin{rem}

There are natural maps:
\[
\Oblv\ul{\Hom}_{\sC}^{enh}(\sF,\sG) \to \ul{\Hom}_{\sC}(\Oblv(\sF),\Oblv(\sG))
\]
\noindent that are easily seen to be isomorphisms for $\sF$ compact. 

\end{rem}

Similarly, for
$\sF \in \sC^{G,w}$, we let:
\[
\ul{\End}_{\sC}^{enh}(\sF,\sG) \in \Alg(\Rep(G)) 
\]

\noindent denote the inner endomorphism algebra.

\subsubsection{}\label{sss:weak-functor}

In our earlier setting, suppose we are given a weakly $G$-equivariant
functor:
\[
F:A\mod \to \sC.
\]

\noindent Note that $A \in A\mod^{G,w}$, so we obtain a natural object 
$\sF \coloneqq F^{G,w}(A) \in \sC^{G,w}$. Moreover, there is a canonical map:
\[
\vph:A \simeq \ul{\End}_{A\mod}^{enh}(A) \to \ul{\End}_{\sC^{G,w}}(\sF)^{op} \in \Alg(\Rep(G)). 
\]

\noindent (Here the superscript \emph{op} indicates the reversed multiplication.) 

Conversely, given $\sF \in \sC^{G,w}$ and $\vph$ as above, we obtain a functor 
$F$ as above. Indeed, as is standard, a datum $\vph$ is the
same as specifying a morphism 
\[
A\mod^{G,w} = A\mod(\Rep(G)) \to \sC^{G,w} 
\]

\noindent of $\Rep(G)$-module categories; by de-equivariantization
(i.e., tensoring over $\Rep(G)$ with $\Vect$), we obtain the 
desired functor $F:A\mod \to \sC$.

\subsubsection{}\label{sss:hc-recap}

Next, recall from \cite{methods} \S 10
that $\Alg(\Rep(G))$ carries a canonical monad $B \mapsto U(\fg) \# B$.

Here $U(\fg) \# B$ is the usual smash (or crossed) product construction.
Non-derivedly, modules for $U(\fg) \# B$ are vector spaces equipped
with an action of $B$ and an action of $\fg$ that are compatible under the
action of $\fg$ on $B$ by derivations. For a proper derived construction 
(in the more complicated setting of topological algebras), we refer to
\cite{methods}.

By definition, making $A \in \Alg(\Rep(G))$ into a module over this monad
is a \emph{Harish-Chandra} datum for $A$. When $A$ is classical, this 
amounts to the data of a morphism:
\[
i:\fg \to A 
\]

\noindent such that $i$ is a $G$-equivariant morphism of Lie algebras such that 
$[i(\xi),-]:A \to A$
coincides with the derivation defined by  $\xi \in \fg$ and the $G$-action on $A$.

From \emph{loc. cit}., it follows that specifying a Harish-Chandra datum for
$A$ is the same as extending the weak $G$-action on $A\mod$ to a strong $G$-action.

\subsubsection{}\label{sss:hc-functor}

Now suppose that $G$ acts strongly on $\sC$. Let $\sF \in \sC^{G,w}$
be a fixed object. We claim that $\ul{\End}_{\sC}^{enh}(\sF) \in \Alg(\Rep(G))$ 
carries a canonical Harish-Chandra datum.

The construction is formal. Let $\sD_0 \subset \sC^{G,w}$ be the
(non-cocomplete) DG category generated by $\sF$ under finite colimits and 
direct summands. As in \cite{dario-*/!}, the monoidal category:
\[
\sH\sC_G \coloneqq \TwoEnd_{G\mod}(D(G)^{G,w})^{op} = \fg\mod^{G,w}
\]

\noindent of Harish-Chandra bimodules acts canonically on $\sC^{G,w}$. 
Moreover, $\sH\sC_G$ is \emph{rigid} monoidal, so its (non-cocomplete) 
monoidal subcategory 
$\sH\sC_G^c$ of compact objects preserves $\sD_0$. Therefore, 
$\sH\sC_G$ acts on $\sD_1 \coloneqq \Ind(\sD_0)$. By de-equivariantization,
$\sD_1 = \sD^{G,w}$ for a canonical strong $G$-category $\sD$.
But clearly $\sD = \ul{\End}_{\sC}^{enh}(\sF)\mod$ by construction,
so we obtain the Harish-Chandra datum as desired (from \S \ref{sss:hc-recap}).

\begin{rem}

In particular, there is a canonical map 
$\fg \to \ul{\End}_{\sC}^{enh}(\sF)$ of Lie algebras in $\Rep(G)$.
This map is the standard \emph{obstruction to strong equivariance}
on $\sF$.

\end{rem}

Moreover, we see that if $A \in \Alg(\Rep(G))$ is equipped with a Harish-Chandra
datum, giving a strongly $G$-equivariant functor:
\[
F:A\mod \to \sC
\]

\noindent is equivalent to giving $\sF \in \sC^{G,w}$ and
a map $\vph:A \to \ul{\End}_{\sC}^{enh}(\sF)^{op}$ that is
\emph{compatible with the Harish-Chandra data}.

In the case that $A$ is classical and $\ul{\End}_{\sC}^{enh}(\sF)^{op}$ 
is coconnective, this simply amounts to saying that composition:
\[
\fg \xar{i} A \xar{\vph} \ul{\End}_{\sC}^{enh}(\sF)^{op}
\]

\noindent coincides with the obstruction to strong equivariance on the object $\sF$.

\subsubsection{}\label{sss:connected-obstr}

We now add one more mild observation before concluding our discussion. The
reader may safely skip this material and return back to it as necessary.

Suppose $G$ acts weakly on $\sC$ and $\sF,\sG \in \sC^{G,w}$.
There is a canonical map:

\[
\Oblv(\ul{\Hom}_{\sC}^{enh}(\sF,\sG)) \to \ul{\Hom}_{\sC}(\Oblv(\sF),\Oblv(\sG))
\in \Vect
\]

\noindent that is an isomorphism whenever $\sF$ is compact. (This is the
reason we use the subscript $\sC$ in $\ul{\Hom}_{\sC}^{enh}$ and
not $\ul{\Hom}_{\sC^{G,w}}^{enh}$: the underlying vector space
of the representation $\ul{\Hom}_{\sC}^{enh}(\sF,\sG)$ is 
comparable to, and often isomorphic to, the complex of maps
from $\sF$ to $\sG$ in $\sC$ itself, not in $\sC^{G,w}$.)

Now suppose that we are in the setting of \S \ref{sss:hc-functor}.
Suppose $\sF \in \sC^{G,w}$ is compact, $A$ is classical, and
$\ul{\End}_{\sC}^{enh}(\sF)$ is coconnective. Then the data
of a map:
\[
A \to \ul{\End}_{\sC}^{enh}(\sF)^{op}
\]

\noindent compatible with the Harish-Chandra data is equivalent to 
a map:
\[
\vph:A \to  \ul{\End}_{\sC}(\sF)^{op} \in \Alg = \Alg(\Vect)
\]

\noindent such that the composition:
\[
\fg \xar{i} A \xar{\vph} \ul{\End}_{\sC}(\sF)^{op}
\]

\noindent is the obstruction to strong equivariance. Indeed, this follows
from the assumption that $\sF$ is compact, so $\ul{\End}_{\sC}(\sF)^{op} = 
\ul{\End}_{\sC}^{enh}(\sF)^{op}$, and the connectedness of $G$, so that
the map $\vph$ will automatically be a morphism of $G$-representations.
(We remark that the co/connectivity assumptions allow us to replace 
$\ul{\End}_{\sC}(\sF)^{op}$ with the classical algebra 
$H^0\ul{\End}_{\sC}(\sF)^{op}$.)
 
\subsection{Reduction}

We now return to our particular setting, and apply the above material to 
reduce our construction to a calculation. 

\subsubsection{}

Let $G = \fL_n^+\bG_m$ act strongly on $\sC \coloneqq \IndCoh^*(\sY^{\leq n})$
via geometric class field theory. By construction, we have:
\[
\IndCoh^*(\sY^{\leq n})^{\fL_n^+\bG_m,w} \simeq  
\IndCoh^*(\sY_{\log}^{\leq n}).
\]

\noindent Under this dictionary, the forgetful functor 
$\sC^{G,w} \to \sC$ corresponds to the $\IndCoh$-pushforward 
$\sY_{\log}^{\leq n} \to \sY^{\leq n}$ (i.e., $!$-averaging
for $\Gr_{\bG_m}^{\leq n}$).

\subsubsection{}

We have the object $i_{n,*}(\sO_{\sZ^{\leq n}}) \in \IndCoh^*(\sY_{\log}^{\leq n})$
that maps to $\sF_n \in \IndCoh^*(\sY^{\leq n})$ under this forgetful functor.
In \S \ref{s:weyl}, we constructed the map:
\[
\vph:A \coloneqq W_n \to \ul{\End}_{\sC}^{enh}(\sF_n)^{op} \simeq 
\ul{\End}_{\sC}(\sF_n)^{op}.
\]

\noindent (We remind that the displayed isomorphism is a formal consequence
of compactness of $\sF_n$.)

Now on the one hand, we have a morphism: 
\[
\Lie(\fL_n^+\bG_m) \simeq k[[t]]/t^n \to \Gamma(\fL_n^+\bA^1,T_{\fL_n^+\bA^1}) 
\to W_n
\]

\noindent encoding the infinitesimal action of $\fL_n^+\bG_m$ on
$\fL_n^+\bA^1$. We denote this morphism by $\alpha$.

On the other hand, we have the composition:
\[
\Lie(\fL_n^+\bG_m) \simeq k[[t]]/t^n \simeq 
(t^{-n} k[[t]]dt/k[[t]] dt)^{\vee} \subset 
\Gamma(\fL^{pol,\leq n} \bA^1 dt,\sO_{\fL^{pol,\leq n} \bA^1 dt})
\to \ul{\End}_{\sC}^{enh}(\sF_n)^{op}
\]

\noindent where the last morphism here uses the map:
\[
\sZ^{\leq n} \subset \sY_{\log}^{\leq n} \to \LocSys_{\bG_m,\log}^{\leq n}
\xar{\Pol} \fL^{pol,\leq n} \bA^1 dt.
\]

\noindent We denote this morphism by $\beta$.

In \S \ref{ss:hc-calc-pf}, we will prove:

\begin{lem}\label{l:hc-calc}

The diagram:
\[
\xymatrix{
\Lie(\fL_n^+\bG_m) \simeq k[[t]]/t^n \ar[d]^{\alpha} \ar[drr]^{\beta} \\
W_n \ar[rr]^{\vph} && \ul{\End}_{\sC}(\sF_n)
}
\]

\noindent commutes.

\end{lem}

Observe that under Theorem \ref{t:trun-lcft}, $\beta$ is interpreted
as the obstruction to strong equivariance. Therefore, 
from \S \ref{sss:connected-obstr}, we deduce:

\begin{prop}\label{p:eq}

The functor $\ol{\Delta}_n$ is canonically a morphism of
$D(\fL_n^+\bG_m)$-module categories. 

\end{prop}

\subsection{Proof of Lemma \ref{l:hc-calc}}\label{ss:hc-calc-pf}

\subsubsection{}

Using the additive structure on $\fL_n^+\bA^1$,
we compute its global vector fields as:
\[
\Gamma(\fL_n^+\bA^1,T_{\fL_n^+\bA^1}) \simeq 
k[[t]]/t^n \otimes \Sym(t^{-n}k[[t]]dt/k[[t]]dt).
\]

\noindent A straightforward calculation shows that the infinitesimal action
map:
\[
\Lie(\fL_n^+\bG_m) \simeq k[[t]]/t^n \to \Gamma(\fL_n^+\bA^1,T_{\fL_n^+\bA^1})
\simeq k[[t]]/t^n \otimes \Sym(t^{-n}k[[t]]dt/k[[t]]dt)
\]

\noindent which was used in the definition of $\alpha$ above,
is the map:
\[
\begin{gathered}
(f \in k[[t]]/t^n) \mapsto (1 \otimes f) \cdot \cotr = (f \otimes 1) \cdot \cotr
\in k[[t]]/t^n \otimes t^{-n}k[[t]]dt/k[[t]]dt \\ \subset 
k[[t]]/t^n \otimes \Sym(t^{-n}k[[t]]dt/k[[t]]dt).
\end{gathered}
\]

\noindent Here we recall the canonical vector:
\[
\begin{gathered}
\cotr \in 
\Gamma(\sD^{\leq n},\sO_{\sD^{\leq n}}) \otimes
\Gamma(\sD^{\leq n},\sO_{\sD^{\leq n}})^{\vee} = \\
\Gamma^{\IndCoh}(\sD^{\leq n} \times \sD^{\leq n},
\sO_{\sD^{\leq n}} \boxtimes \omega_{\sD^{\leq n}}) = 
k[[t]]/t^n \otimes t^{-n}k[[t]]dt/k[[t]]dt
\end{gathered}
\]

\noindent from Lemma \ref{l:cotr-1}.

Therefore, we have:
\[
\alpha(f) = (1 \otimes f) \cdot \cotr \in 
k[[t]]/t^n \otimes t^{-n}k[[t]]dt/k[[t]]dt \subset W_n.
\]

\subsubsection{}

To simplify the discussion, we take $f = 1$ for the time being. 
We wish to show that $\vph\alpha(1) = \beta(1)$, i.e., we wish 
to show that $\vph(\cotr) = \beta(1)$.
We will adapt the calculation to discuss the
case of general $f$ in \S \ref{sss:gen'l-f}.

\subsubsection{}

Let us unwind the construction of $\vph(\cotr)$ in this case.

First, we form:
\[
\cotr \in \Gamma^{\IndCoh}(\sD^{\leq n} \times \sD^{\leq n},
\sO_{\sD^{\leq n}} \boxtimes \omega_{\sD^{\leq n}}).
\]

\noindent Second, we form:\footnote{Here we abuse notation: in \S \ref{ss:act-inv}, we used the
notation $\omega^{univ}$ for a similar construction with $\sZ^{\leq n}$ in place of
$\fL^{pol,\leq n} \bA^1 dt$; that setting is pulled back from the present
one.}
\[
\omega^{univ} \in \Gamma^{\IndCoh}(\sD^{\leq n} \times 
\fL^{pol,\leq n} \bA^1 dt,
\omega_{\sD^{\leq n}} \boxtimes \sO_{\fL^{pol,\leq n} \bA^1 dt}).
\]

\noindent We then form:

\[
p_{12}^*(\cotr) \cdot p_{23}^*(\omega^{univ}) \in 
\Gamma^{\IndCoh}(\sD^{\leq n} \times \sD^{\leq n} \times 
\fL^{pol,\leq n} \bA^1 dt,
\omega_{\sD^{\leq n}} \boxtimes \omega_{\sD^{\leq n}} \boxtimes 
\sO_{\fL^{pol,\leq n} \bA^1 dt}).
\]

According to \S \ref{sss:recipe}, $\vph(\cotr)$ is calculated as follows.
We pull back the above section to $\sD^{\leq n} \times \sD^{\leq n} \times 
\sZ^{\leq n}$, pushforward along:
\[
\sD^{\leq n} \times \sD^{\leq n} \times \sZ^{\leq n}
\xar{\AJ^{-1} \times \AJ \times \id}
\Gr_{\bG_m}^{\leq n} \times \Gr_{\bG_m}^{\leq n} 
\times \sZ^{\leq n} \xar{\on{mult} \times \id} 
\Gr_{\bG_m}^{\leq n} \times \sZ^{\leq n}
\] 

\noindent and apply the construction of \S \ref{ss:gpd}.

Of course, the pullback and the pushforward commute here. 
Therefore, in what follows we calculate the pushforward with $\sZ^{\leq n}$
replaced by $\fL^{pol,\leq n} \bA^1 dt$.

In what follows, let:
\[
\sigma \in \Gamma^{\IndCoh}(\Gr_{\bG_m}^{\leq n} \times \fL^{pol,\leq n} \bA^1 dt,
\omega_{\Gr_{\bG_m}^{\leq n}} \boxtimes \sO_{\fL^{pol,\leq n} \bA^1 dt})
\]

\noindent denote the resulting section. So our objective, in effect, is to calculate
$\sigma$.

\subsubsection{}

Recall from Proposition \ref{p:cc-log} that: 
\begin{equation}\label{eq:gr-cc}
\Gamma^{\IndCoh}(\Gr_{\bG_m}^{\leq n},\omega_{\Gr_{\bG_m}^{\leq n}}) \simeq 
\Gamma(\fL_n^+ \bG_m,\sO_{\fL_n^+ \bG_m}).
\end{equation}

\noindent Moreover, this is an isomorphism of commutative algebras. It is
convenient to reinterpret the above construction using this isomorphism.

By construction, note that the composition:

\[
t^{-n}k[[t]]dt/k[[t]]dt \simeq 
\Gamma^{\IndCoh}(\sD^{\leq n},\omega_{\sD^{\leq n}}) \to 
\Gamma^{\IndCoh}(\Gr_{\bG_m}^{\leq n},\omega_{\Gr_{\bG_m}^{\leq n}}) 
\simeq \Gamma(\fL_n^+ \bG_m,\sO_{\fL_n^+ \bG_m})
\]

\noindent using pushforward along $\AJ$ sends:
\[
\eta \in t^{-n}k[[t]]dt/k[[t]]dt
\] 

\noindent to the function:
\[
(g \in \fL_n^+\bG_m) \mapsto \Res(g \eta).
\]

\noindent Similarly, if we use $\AJ^{-1}$ instead, we obtain the
function:
\[
(g \in \fL_n^+\bG_m) \mapsto \Res(g^{-1} \eta).
\]

\subsubsection{}

Our original section:
\[
p_{12}^*(\cotr) \cdot p_{23}^*(\omega^{univ}) \in 
\Gamma^{\IndCoh}(\sD^{\leq n} \times \sD^{\leq n} \times 
\fL^{pol,\leq n} \bA^1 dt,
\omega_{\sD^{\leq n}} \boxtimes \omega_{\sD^{\leq n}} \boxtimes 
\sO_{\fL^{pol,\leq n} \bA^1 dt})
\]

\noindent can be interpreted as a map:
\[
\begin{gathered}
\fL^{pol,\leq n} \bA^1 dt = 
\ul{t^{-n}k[[t]]dt/k[[t]]dt} \to 
\ul{t^{-n}k[[t]]dt/k[[t]]dt \otimes t^{-n}k[[t]]dt/k[[t]]dt}
\end{gathered}
\]

\noindent where for a finite-dimensional vector
space $V$, we let $\ul{V}$ denote the scheme $\Spec(\Sym(V^{\vee}))$.
It is immediate to see that the above map is given
by the formula:
\[
\eta \mapsto (\eta \otimes 1) \cdot \cotr.
\]

Pushing $p_{12}^*(\cotr) \cdot p_{23}^*(\omega^{univ})$ forward along:
\[
\sD^{\leq n} \times \sD^{\leq n} \times \sZ^{\leq n}
\xar{\AJ^{-1} \times \AJ \times \id}
\Gr_{\bG_m}^{\leq n} \times \Gr_{\bG_m}^{\leq n} \times \fL^{pol,\leq n} \bA^1 dt
\]

\noindent and applying \eqref{eq:gr-cc}, we obtain a function on 
\[
\fL_n^+\bG_m \times \fL_n^+\bG_m \times \fL^{pol,\leq n} \bA^1 dt.
\]

\noindent We deduce that it is given by the formula:
\begin{equation}\label{eq:res-res}
(g_1,g_2,\eta) \mapsto 
(\Res \otimes \Res)((g_1^{-1} \otimes g_2) \cdot 
(\eta \otimes 1) \cdot \cotr).
\end{equation}

\subsubsection{}

We now wish to calculate \eqref{eq:res-res} more explicitly. 

First, we note the following, which is an immediate calculation.

\begin{lem}\label{l:res-cotr}

For $\eta \in \fL^{pol,\leq n} \bA^1 dt$, we have:
\[
(\id \otimes \Res)((\eta \otimes 1) \cdot \cotr) = \eta.
\]

\end{lem}

We then deduce:

\begin{lem}\label{l:res-res}

For $(g,\eta) \in  \fL_n^+\bA^1 \times \fL^{pol,\leq n} \bA^1 dt$, we have:
\[
(\Res \otimes \Res)((\eta \otimes g) \cdot \cotr) = \Res(g\eta).
\]

\end{lem}

\begin{proof}

We have:
\[
(1 \otimes g) \cdot \cotr = (g \otimes 1) \cdot \cotr.
\]

\noindent Indeed, this encodes the fact that the residue pairing is
$k[[t]]$-equivariant.

Therefore, we have:
\[
(\eta \otimes g) \cdot \cotr = (g\eta \otimes 1) \cdot \cotr
\]

\noindent which gives:
\[
(\Res \otimes \Res)((\eta \otimes g) \cdot \cotr) = 
\Res (\id \otimes \Res)((g\eta \otimes 1) \cdot \cotr).
\]

\noindent Now the result follows from Lemma \ref{l:res-cotr}.

\end{proof}

In the setting of \eqref{eq:res-res}, we deduce:
\begin{equation}\label{eq:res-res-calc}
(\Res \otimes \Res)((g_1^{-1} \otimes g_2) \cdot 
(\eta \otimes 1) \cdot \cotr) = 
\Res(\frac{g_2}{g_1} \eta).
\end{equation}

Now, our section $\sigma$ is obtained by using a pushforward
along $\Gr_{\bG_m}^{\leq n} \times \Gr_{\bG_m}^{\leq n} \to \Gr_{\bG_m}^{\leq n}$.
On the dual side, this means we restrict our function along: 
\[
\fL_n^+ \bG_m \xar{\Delta \times \id}
\fL_n^+\bG_m \times \fL_n^+\bG_m \times \fL^{pol,\leq n} \bA^1 dt.
\]

\noindent By \eqref{eq:res-res-calc}, we see the resulting function is:
\[
(g,\eta) \mapsto \Res(\eta).
\]

This clearly coincides with the endomorphism defined by $\beta(1)$, giving the
claim (in the $f=1$ case).

\subsubsection{}\label{sss:gen'l-f}

As promised, we now treat the case of general $f$.

Briefly, one replaces \eqref{eq:res-res} with:
\[
(g_1,g_2,\eta) \mapsto 
(\Res \otimes \Res)((g_1^{-1} \otimes g_2) \cdot 
(\eta \otimes 1) \cdot (1 \otimes f) \cdot \cotr),
\]

\noindent which Lemma \ref{l:res-res} implies is:
\[
(g_1,g_2,\eta) \mapsto \Res(\frac{f g_2}{g_1}\eta),
\]

\noindent which for $g_1 = g_2 = g$ is:
\[
(g,\eta) \mapsto \Res(f \eta).
\]

\noindent Again, the induced endomorphism clearly coincides with that 
defined by $\beta(f)$, concluding the argument.

\section{Fully faithfulness}\label{s:ff}

\subsection{Overview}

\subsubsection{}

The main result of this section is the following.

\begin{prop}\label{p:ff}

For every $n \geq 0$, the functor $\Delta_n$ is fully faithful.

\end{prop}

\begin{rem}\label{r:kash-y}

By Lemma \ref{l:lambda-ff}, the pushforward
functors $\IndCoh^*(\sY_n) \to \IndCoh^*(\sY)$ are fully faithful
for $n>0$. Therefore, for $n \neq 0$, the above result is equivalent
to fully faithfulness of $\ol{\Delta}_n$. 

\end{rem}

\begin{rem}

As $\lambda_{n,*}^{\IndCoh}(\sF_n)$ is compact in 
$\IndCoh^*(\sY)$, Proposition \ref{p:ff} is equivalent to showing
that the map:
\[
W_n^{op} \to \ul{\End}_{\IndCoh^*(\sY)}(\lambda_{n,*}^{\IndCoh}\sF_n)
\]

\noindent is an isomorphism.

\end{rem}

Our proof is essentially by induction. We settle the $n = 0$ case
by explicit analysis. We then use local class field theory
(or the Contou-Carr\`ere pairing) to perform the inductive step.

\subsection{Proof for $n = 0$}

We begin with the base case of our induction.

\subsubsection{A preliminary calculation}

We begin with the following explicit calculation.

Recall the scheme $C = \Spec(k[x,y]/xy)$ from \S \ref{ss:refinement}.
We equip it with the $\bG_m$-action of \S \ref{ss:z-refinement},
so $\deg(x) = 1$ and $\deg(y) = 0$ for the corresponding grading.

Let $\sO_{\bA_x^1/\bG_m} \in \QCoh(C/\bG_m)^{\heart}$ denote the structure
sheaf of the $x$-axis, i.e., the object corresponding to the graded
$k[x,y]/xy$-module $k[x] = k[x,y]/y$ (where the generator has degree $0$).

\begin{lem}\label{l:c-ext}

The unit map:
\[
k \to \ul{\End}_{\QCoh(C/\bG_m)}(\sO_{\bA_x^1/\bG_m})
\]

\noindent is an isomorphism.

\end{lem}

\begin{proof}

This follows by a standard $\Ext$-calculation that we give here.

Let $A = k[x,y]/xy$. We let $A(n)$ denote $A$ as a graded $A$-module
where the generator is given degree $n$.\footnote{We remark that the
sign conventions here are the same as in \S \ref{sss:z-pic}.}

The complex:
\[
\ldots 
\xar{y\cdot -} A(2) \xar{x\cdot -} A(1) \xar{y\cdot -} A(1) \xar{x\cdot -} A \xar{y\cdot -} \underset{\text{coh. deg. } 0} {A} \to 0 \to \ldots 
\]

\noindent gives a graded free resolution of $k[x] = A/y$.

Therefore, we may calculate $\ul{\End}_{A\mod}(k[x])$ as a
graded vector space via the complex:
\[
\ldots \to  0 \to \underset{\text{coh. deg. } 0}{k[x]} 
\xar{0} k[x] \xar{x \cdot-} k[x](-1)
\xar{0} k[x](-1) \xar{x \cdot-} k[x](-2) \xar{0} \ldotsplus
\]

\noindent The (graded) degree $0$ component of this complex,
which computes: $$\ul{\End}_{A\mod^{\bG_m,w}}(k[x]) = 
\ul{\End}_{\QCoh(C/\bG_m)}(\sO_{\bA_x^1/\bG_m}),$$ is:
\[
\ldots \to 0 \to \underset{\text{coh. deg. } 0}{k} 
\xar{0} k \xar{x \cdot-} k\cdot x
\xar{0} k\cdot x \xar{x \cdot-} k\cdot x^2 \xar{0} \ldotsplus 
\]

\noindent Clearly this complex is acyclic outside of degree cohomological $0$,
and its $H^0$ is 1-dimensional and generated by the unit. 

\end{proof}

\subsubsection{}

We have the following result, explicitly describing the geometry of our situation.

\begin{lem}\label{l:z0-emb}

\begin{enumerate}

\item There is a canonical isomorphism $\sZ^{\leq 0} \simeq \bA^1/\bG_m$.
Explicitly, given $(\sL,\nabla,s) \in \sZ^{\leq 0}$, we
take the line $\sL|_0$ with its section $s|_0$ (for $0 \in \sD$ the base-point).

\item The canonical map $\sZ^{\leq 0} \to \sY$ is an ind-closed embedding.
Its formal completion $\sY_{\sZ^{\leq 0}}^{\wedge}$ 
is isomorphic to $(C/\bG_m)_{\bA^1/\bG_m}^{\wedge}$ compatibly with
the above isomorphism $\sZ^{\leq 0} \simeq \bA^1/\bG_m = \bA_x^1/\bG_m$.

\end{enumerate}

\end{lem}

\begin{proof}

We fix a coordinate $t$ and then consider $\LocSys_{\bG_m}$ 
as mapping to $\bB \bG_m$ via the isomorphism
Proposition \ref{p:locsys}. 
For a prestack $S$ over $\LocSys_{\bG_m}$, let $\widetilde{S}$ denote
the base-change $S \times_{\bB \bG_m} \Spec(k)$. (We remark that this
notation is compatible with that of \S \ref{sss:z-pic}.)

By Proposition \ref{p:zn-cl} and Theorem \ref{t:cl}, 
$\sZ^{\leq 0}$ and $\sY$ are classical. Moreover, the formal completion
$\sY_{\sZ^{\leq 0}}^{\wedge}$ is classical: this follows immediately
from Proposition \ref{p:z-flat-axes} \eqref{i:z-axes-1}.

We can explicitly calculate $\widetilde{\sY_{\sZ^{\leq 0}}^{\wedge}}$
by classicalness and Proposition \ref{p:locsys}: it parametrizes
$y \in \bA_0^{1,\wedge}$ (defining the local system $(\sO,d-y\frac{dt}{t})$)
and an element $f \in (\fL \bA^1)_{\fL^+ \bA^1}^{\wedge}$
such that $df = y f \frac{dt}{t}$. 
If we expand  $f = \sum a_i t^i$, we find $i a_i = y a_i$. As $y$ is
nilpotent, this implies $a_i = 0$ for all $i \neq 0$ and $y a_0 = 0$.
Clearly $\widetilde{\sZ^{\leq 0}}$ corresponds to the locus $y = 0$.
Therefore, writing $x$ for $a_0$ and noting that the relevant $\bG_m$-action
scales $x$, we obtain the claims.

\end{proof}

\subsubsection{}

Now the $n = 0$ case of Proposition \ref{p:ff} asserts that the unit map:
\[
k \to \ul{\End}_{\IndCoh^*(\sY)}(\sF_0)
\]

\noindent is an isomorphism. Note that $\sF_0$ is just the pushforward
of the structure sheaf of $\sZ^{\leq 0}$ to $\sY$. The calculation of
the above depends only on the formal completion, so we
obtain the result from Lemmas \ref{l:c-ext} and \ref{l:z0-emb}.

\subsection{A fully faithfulness result}

Before proceeding to our induction, we will need the following observations.

\subsubsection{}

Recall the maps $\zeta$ and $\widetilde{\zeta}$ from \S \ref{ss:y-not}.

\begin{prop}\label{p:zeta}

For any $n > 0$, the map:
\[
\ul{\End}_{\IndCoh^*(\LocSys_{\bG_m}^{\leq n})}(
\widetilde{\pi}_{n,*}^{\IndCoh}\sO_{\LocSys_{\bG_m}^{\leq n}}) \to 
\ul{\End}_{\IndCoh^*(\sY^{\leq n})}(
\zeta_*^{\IndCoh}\widetilde{\pi}_{n,*}^{\IndCoh}
\sO_{\LocSys_{\bG_m,\log}^{\leq n}}) 
\]

\noindent is an isomorphism.

\end{prop}

We prove this result in \S \ref{ss:zeta-pf}, after some 
preliminary remarks.

\subsubsection{}

Let $A = \oplus_{i \geq 0} A_i $ be a 
$\bZ^{\geq 0}$-graded (classical, say) algebra.

We let $A\mod^{\bG_m,w} \in \DGCat_{cont}$ denote the 
DG category of graded $A$-modules. For $M \in A\mod^{\bG_m,w}$,
we write $M = \oplus_{i \in \bZ} M_i$ for its
decomposition into weight spaces.

\begin{lem}\label{l:gr-hom}

Suppose $M,N \in A\mod^{\bG_m,w}$. Suppose that 
$M$ is concentrated in positive (graded) degrees,
i.e., $M = \oplus_{i>0} M_i$. Suppose $N$ is concentrated
in degree $0$, i.e., $N = N_0$.

Then:
\[
\ul{\Hom}_{A\mod^{\bG_m,w}}(M,N) = 0.
\]

\end{lem}

\begin{proof}

Let $A(i) \in A\mod^{\bG_m,w}$ be as in the proof of Lemma \ref{l:c-ext}, i.e.,
$A$ graded with generator in degree $i$. The modules $A(i)$ for $i > 0$ 
generate the subcategory of $A\mod^{\bG_m,w}$ consisting of modules concentrated
in positive degrees. So we reduce to the case $M = A(i)$ for $i>0$, for which the claim is obvious.
	
\end{proof}

\subsubsection{}

Suppose now that we are given a $\bZ^{\geq 0}$-graded 
algebra $B = \oplus_{i \geq 0} B_i$ 
and a graded map $\iota^*:A \to B$ inducing
an isomorphism $A_0 \isom B_0$. We let $\iota_*:B\mod^{\bG_m,w} \to A\mod^{\bG_m,w}$
and $\iota^*:A\mod^{\bG_m,w} \to B\mod^{\bG_m,w}$ denote the induced
adjoint functors.

\begin{prop}\label{p:gr-hom}

In the above setting, suppose $M,N \in B\mod^{\bG_m,w}$ with 
$M$ concentrated in non-negative graded\footnote{As opposed to cohomological.} 
degrees and $N$ concentrated in degree $0$. 

Then the natural map:
\[
\ul{\Hom}_{B\mod^{\bG_m,w}}(M,N) \to 
\ul{\Hom}_{A\mod^{\bG_m,w}}(\iota_*M,\iota_* N) 
\]

\noindent is an isomorphism.

\end{prop}

\begin{proof}

\step\label{st:gr-1} 

Let $\widetilde{M} \in A\mod^{\bG_m,w}$ be concentrated in non-negative graded degrees.
We claim that $\Ker(\widetilde{M} \to \iota_*\iota^* \widetilde{M})$ is concentrated
in (strictly) positive graded degrees.

First, in the case $\widetilde{M} = A$, this map is $\Ker(A \to \iota_*(B)$, in which case
the claim follows as we assumed $A \to B$ an isomorphism in graded degree $0$.

In general, we have:
\[
\Ker(\widetilde{M} \to \iota_*\iota^* \widetilde{M}) = \Ker(A \to \iota_*(B))  \underset{A}{\otimes}
\widetilde{M} 
\]

\noindent which is the tensor product of a module in graded degrees $\geq 0$ and one in graded degrees $>0$, so 
is in graded degrees $>0$ (as $A$ is non-negatively graded), as claimed.

\step\label{st:gr-2}

Next, we claim that $\Ker(\iota^*\iota_* M \to M)$ is concentrated in strictly positive graded degrees.

Clearly $\iota^*\iota_* M = B \otimes_A \iota_* M$ is in degrees $\geq 0$, so it suffices to see that
$\iota^*\iota_* M \to M$ is an isomorphism in degree $0$. We can check this after applying $\iota_*$.
Then we have a retraction:
\[
\iota_* M \to \iota_* \iota^*\iota_* M \to \iota_* M,
\]

\noindent so it suffices to see that the first map is an isomorphism in degree $0$.
Taking $\widetilde{M} = \iota_*M$, this follows from Step \ref{st:gr-1}.

\step Applying Lemma \ref{l:gr-hom} to $\Ker(\iota^*\iota_*M \to M)$ and $N$, we see that the map:
\[
\ul{\Hom}_{B\mod^{\bG_m,w}}(M,N) \to \ul{\Hom}_{B\mod^{\bG_m,w}}(\iota^*\iota_*M,N) 
\]

\noindent is an isomorphism. The right hand side identifies with $\ul{\Hom}_{A\mod^{\bG_m,w}}(\iota_*M,\iota_*N)$
by adjunction, and the induced map is given by $\iota_*$ functoriality, so we obtain the claim.
 	
\end{proof}

\subsubsection{}

We will apply the above in the following setting.

Let $A_n$ be the graded ring of \S \ref{ss:coords}.
We write $A_n = \oplus_{i \geq 0} A_{n,i}$ for 
its decomposition into graded pieces. We remind
from \emph{loc. cit}. that $\Spec(A_n)/\bG_m = \sZ^{\leq n}$.

By construction, we have:
\[
\LocSys_{\bG_m,\log}^{\leq n} \underset{\bB \bG_m}{\times} \Spec(k) = \Spec(A_{n,0})
\]

\noindent such that the graded algebra maps $A_{n,0} \to A_n \to A_{n,0}$
correspond to the maps

\begin{equation}\label{eq:zeta-proj}
\LocSys_{\bG_m,\log}^{\leq n} \xar{\widetilde{\zeta}} \sZ^{\leq n} \to 
\LocSys_{\bG_m,\log}^{\leq n}
\end{equation}

\noindent (the latter map being the projection).

\subsubsection{}\label{sss:gr-hom}

Recall the maps $\iota$, $\iota^r$ from \S \ref{ss:y-not}.

The map $\iota$ arises\footnote{Non-canonically: i.e., we need a choice
of coordinate on the disc to turn $\iota$ into a map of 
stacks over $\bB \bG_m$.} from a map $\iota^*:A_n \to A_n$ of graded rings.
This map is an isomorphism
on $0$th graded components, as is evident from 
\eqref{eq:zeta-proj}. 

\begin{cor}\label{c:gr-hom}

Suppose $\sG \in \QCoh(\LocSys_{\bG_m,\log}^{\leq n})$. Then the morphism:
\[
\ul{\Hom}_{\QCoh(\sZ^{\leq n})}(\sO_{\sZ^{\leq n}},\widetilde{\zeta}_* \sG) \to 
\ul{\Hom}_{\QCoh(\sZ^{\leq n})}(\iota_*\sO_{\sZ^{\leq n}},\iota_*\widetilde{\zeta}_* \sG) 
\]

\noindent is an isomorphism. More generally, 
for any $r,s \geq 0$, the morphism:
\[
\ul{\Hom}_{\QCoh(\sZ^{\leq n})}(\iota_*^r(\sO_{\sZ^{\leq n}}),\iota_*^s\widetilde{\zeta}_* \sG) \to 
\ul{\Hom}_{\QCoh(\sZ^{\leq n})}(\iota_*^{r+1}(\sO_{\sZ^{\leq n}}),\iota_*^{s+1}\widetilde{\zeta}_* \sG) 
\]

\noindent is an isomorphism.
	
\end{cor}

\begin{proof}

Translating to graded modules and using \eqref{eq:zeta-proj},
the hypotheses of Proposition \ref{p:gr-hom} are clearly satisfied, giving the result.
	
\end{proof}

\subsubsection{Conclusion}\label{ss:zeta-pf}

We now prove Proposition \ref{p:zeta}.

Observe that $\LocSys_{\bG_m}^{\leq n} = \lim_n \sZ^{\leq n}$, where the limit
is formed under the (affine) morphism $\iota$. Therefore, we have:
\[
\widetilde{\zeta}_*\sO_{\LocSys_{\bG_m,\log}^{\leq n}} = 
\underset{r}{\colim} \, \iota_*^r \sO_{\sZ^{\leq n}} \in \QCoh(\sZ^{\leq n}).
\]

\noindent As all of these terms are in the heart of the $t$-structure, 
we obtain a similar identity in $\IndCoh(\sZ^{\leq n})$, using $\IndCoh$-pushforwards 
in the notation instead.

Define:
\[
\sG \coloneqq \Oblv\Av_!^{\Gr_{\bG_m}^{\leq n},w}
\sO_{\LocSys_{\bG_m,\log}^{\leq n}} \in 
\QCoh(\LocSys_{\bG_m,\log}^{\leq n})^{\heart}.
\]

We let $i:\sZ^{\leq n} \to \sY_{\log}^{\leq n} = \colim_{n,\iota} \sZ^{\leq n}$ 
denote the structural map. Using adjunction, we rewrite:
\begin{equation}\label{eq:zeta-rhs}
\ul{\End}_{\IndCoh^*(\sY^{\leq n})}
(\zeta_*^{\IndCoh}\widetilde{\pi}_{n,*}^{\IndCoh}
\sO_{\LocSys_{\bG_m,\log}^{\leq n}}),
\end{equation}

\noindent which is the right hand side of Proposition \ref{p:zeta}, as:
\[
\begin{gathered}
\ul{\Hom}_{\IndCoh^*(\sY_{\log}^{\leq n})}
(i_*^{\IndCoh}\widetilde{\zeta}_*^{\,\IndCoh}\sO_{\LocSys_{\bG_m,\log}^{\leq n}},
i_*^{\IndCoh}\widetilde{\zeta}_*^{\,\IndCoh}\Oblv\Av_!^{\bG_m^{\leq n},w}\sO_{\LocSys_{\bG_m,\log}^{\leq n}}) = \\
\ul{\Hom}_{\IndCoh^*(\sY_{\log}^{\leq n})}
(i_*^{\IndCoh}\widetilde{\zeta}_*^{\,\IndCoh}\sO_{\LocSys_{\bG_m,\log}^{\leq n}},
i_*^{\IndCoh}\widetilde{\zeta}_*^{\,\IndCoh}\sG). 
\end{gathered}
\]

\noindent By the above, we can further express this term as:
\[
\begin{gathered}
\ul{\Hom}_{\IndCoh^*(\sY_{\log}^{\leq n})}
(i_*^{\IndCoh}\underset{r}{\colim} \, \iota_*^{r,\IndCoh} \sO_{\sZ^{\leq n}},
i_*^{\IndCoh}\widetilde{\zeta}_*^{\,\IndCoh}\sG) = \\
\underset{r}{\lim} \, \ul{\Hom}_{\IndCoh^*(\sY_{\log}^{\leq n})}
(i_*^{\IndCoh} \iota_*^{r,\IndCoh} \sO_{\sZ^{\leq n}},
i_*^{\IndCoh}\widetilde{\zeta}_*^{\,\IndCoh}\sG).
\end{gathered}
\]

\noindent As $\iota_*^r \sO_{\sZ^{\leq n}} \in \Coh(\sZ^{\leq n})$ for all $r$,
by \cite{cpsii} Corollary 6.5.3, we may calculate the above as:
\[
\underset{r}{\lim} \, 
\underset{s}{\colim} \, 
\ul{\Hom}_{\IndCoh^*(\sZ^{\leq n})}
(\iota_*^{r+s,\IndCoh} \sO_{\sZ^{\leq n}},
\iota_*^s\widetilde{\zeta}_*^{\,\IndCoh}\sG).
\]

Now for fixed $r$, Corollary \ref{c:gr-hom} implies that all of the
maps in the above colimit are isomorphisms. Therefore, the above
may be calculated as:
\[
\underset{r}{\lim} \, 
\ul{\Hom}_{\IndCoh^*(\sZ^{\leq n})}
(\iota_*^{r,\IndCoh} \sO_{\sZ^{\leq n}},
\widetilde{\zeta}_*^{\,\IndCoh}\sG).
\]

\noindent Moreover, the analysis from \emph{loc. cit}. shows that
$\sO_{\sZ} \to \iota_*^r\sO_{\sZ}$ corresponds to a map of
non-negatively graded modules that is an isomorphism in degree $0$, 
so Lemma \ref{l:gr-hom} shows that the maps is the above
limit are isomorphisms as well. Therefore, we may calculate this
limit as:
\begin{equation}\label{eq:no-r-s}
\ul{\Hom}_{\IndCoh^*(\sZ^{\leq n})}
(\sO_{\sZ^{\leq n}},
\widetilde{\zeta}_*^{\,\IndCoh}\sG).
\end{equation}

\noindent This term is:
\[
\begin{gathered}
\Gamma^{\IndCoh}(\sZ^{\leq n},\widetilde{\zeta}_*^{\,\IndCoh}\sG) = 
\Gamma^{\IndCoh}(\LocSys_{\bG_m,\log}^{\leq n},\sG) = \\
\ul{\Hom}_{\IndCoh^*(\LocSys_{\bG_m,\log}^{\leq n})}
(\sO_{\LocSys_{\bG_m,\log}^{\leq n}},\Oblv\Av_!^{\Gr_{\bG_m}^{\leq n},w} 
\sO_{\LocSys_{\bG_m,\log}^{\leq n}}) = \\ 
\ul{\End}_{\IndCoh^*(\LocSys_{\bG_m}^{\leq n})}
(\Av_!^{\Gr_{\bG_m}^{\leq n},w} 
\sO_{\LocSys_{\bG_m,\log}^{\leq n}}).
\end{gathered}
\]

\noindent This last expression is another way of writing the left hand side
of Proposition \ref{p:zeta}. It is straightforward to see that the isomorphism
we have just constructed is inverse to the map in \emph{loc. cit}., yielding the result.

\subsection{Structure of the argument}

We are now positioned to prove Proposition \ref{p:ff}.
We begin by outlining the argument.

\subsubsection{}

The starting point is the following elementary observation.

\begin{lem}\label{l:weyl-bimod}

Suppose $f:M \to N \in W_2\mod$ is a morphism (in the derived category) of 
modules over the Weyl algebra $W_2$ in two variables,
with generators denoted $\alpha$, $\beta$, $\partial_{\alpha}$ and $\partial_{\beta}$. 
Then $f$ is an isomorphism if and only if the morphisms:
\begin{equation}\label{eq:weyl-bimod}
\begin{gathered}
\underset{r,\alpha}{\colim} \, M \coloneqq \colim(M \xar{\alpha} M \xar{\alpha} \ldots) \to 
\underset{r,\alpha}{\colim} \, N \\
\underset{r,\beta}{\colim} \, M \to \underset{r,\beta}{\colim} \, N \\
M^{\alpha}/\beta \to N^{\alpha}/\beta	
\end{gathered}
\end{equation}

\noindent are isomorphisms in $\Vect$. 
In the last line, we define e.g. $M^{\alpha}$ as the
(homotopy) kernel of $\alpha:M \to M$, and the quotient by $\beta$ means
the homotopy cokernel of the induced map $\beta:M^{\alpha} \to M^{\alpha}$
(which exists because $\alpha$ and $\beta$ commute).
	
\end{lem}

\begin{proof}

In terms of $D$-modules on $\bA^2$, our assumptions are that $f$ induces an isomorphism 
on restriction to $(\bA^1\setminus 0) \times \bA^1$, $\bA^1 \times (\bA^1 \setminus 0)$,
and on $!$-restriction to $0$ (up to a cohomological shift). This implies the claim by 
$D$-module excision.

\end{proof}

\subsubsection{}\label{sss:ff-induction-strat}

We choose coordinates on the disc $\sD$. We obtain an isomorphism 
$\fL_n^+ \bA^1 \simeq \bA^n$; we let $\alpha_0,\ldots,\alpha_{n-1}$ denote the coordinates
(so e.g. $\alpha_0 = \ev$ as maps to $\bA^1$).
We obtain an evident decomposition $W_n \simeq W_1 \otimes W_{n-1}$, where the first
tensor factor uses the coordinate $\alpha_0$.

Let $f$ denote the morphism $W_n^{op} \to \ul{\End}_{\IndCoh^*(\sY^{\leq n})}(\sF_n) \in \Alg$.
As we have a morphism $W_1^{op} \to W_n^{op}$ of algebras as above, we may regard $f$
as a morphism of $W_1$-bimodules (via left and right multiplication).

Our argument applies Lemma \ref{l:weyl-bimod} using this bimodule structure.
Generally speaking, the first two conditions from the lemma 
amount to Proposition \ref{p:zeta} and its relatives,
and the last one is given by induction.

\subsection{Verifying the axioms} % better heading?

\subsubsection{}

In the setting of \S \ref{sss:ff-induction-strat}, 
it remains to verify the conditions of Lemma \ref{l:weyl-bimod}.
We do so below. 

\subsubsection{Colimits}

We begin with checking the first two maps from \eqref{eq:weyl-bimod} are isomorphisms.

Clearly:
\[
\underset{r,\alpha_0\cdot-}{\colim} \, W_n = \underset{r,-\cdot \alpha_0}{\colim} \, W_n = W_n[\alpha_0^{-1}]
= \Gamma(\fL_n^+ \bG_m,D_{\fL_n^+ \bG_m}).
\]

\noindent I.e., in our example, the left hand sides of the first two equations in \eqref{eq:weyl-bimod}
each identify with global differential operators on $\fL_n^+ \bG_m$.

Below, we verify the same for the right hand sides in our examples. Then we check that our map ($f$, in the notation of Lemma \ref{l:weyl-bimod}), is compatible
with these identifications.

\subsubsection{}\label{sss:alpha-0-left}

The left action of $\alpha_0$ on $\ul{\End}_{\IndCoh^*(\sY^{\leq n})}(\sF_n)$ 
is, by construction, obtained by right composition with the map:
\[
\sF_n = 
\Av_!^{\Gr_{\bG_m}^{\leq n},w}(i_{n,*}^{\IndCoh}(\sO_{\sZ^{\leq n}})) \to 
\Av_!^{\Gr_{\bG_m}^{\leq n},w}(\iota_*^{\IndCoh}i_{n,*}^{\IndCoh}(\sO_{\sZ^{\leq n}})) = \sF_n
\]

\noindent obtained by applying $\Av_!^{\Gr_{\bG_m}^{\leq n},w}$ to 
the canonical map:
\[
\sO_{\sZ^{\leq n}} \to \iota_*^{\IndCoh}\sO_{\sZ^{\leq n}}.
\]

Therefore, by compactness of $\sF_n$, we see that:
\[
\underset{r,\alpha_0\cdot-}{\colim} \, \ul{\End}_{\IndCoh^*(\sY^{\leq n})}(\sF_n)
\isom 
\ul{\Hom}_{\IndCoh^*(\sY^{\leq n})}(\sF_n, 
\Av_!^{\Gr_{\bG_m}^{\leq n},w} \underset{r}{\colim} \, 
(\iota_*^{r,\IndCoh})i_{n,*}^{\IndCoh}(\sO_{\sZ^{\leq n}}).
\]

\noindent The right hand side clearly identifies with:
\[
\ul{\Hom}_{\IndCoh^*(\sY^{\leq n})}(\sF_n, 
\zeta_*^{\IndCoh}
\Oblv\Av_!^{\Gr_{\bG_m}^{\leq n},w} (\sO_{\LocSys_{\bG_m,\log}^{\leq n}})).
\]

\noindent Unwinding the definition of $\sF_n$, we can further identify it with:
\[
\ul{\Hom}_{\IndCoh^*(\sY_{\log}^{\leq n})}\big(i_{n,*}^{\IndCoh}\sO_{\sZ^{\leq n}}, 
i_{n,*}^{\IndCoh}\widetilde{\zeta}_*^{\,\IndCoh}
\Oblv\Av_!^{\Gr_{\bG_m}^{\leq n},w} (\sO_{\LocSys_{\bG_m,\log}^{\leq n}})\big).
\]

\noindent The comparison of \eqref{eq:zeta-rhs} and \eqref{eq:no-r-s} 
from \S \ref{ss:zeta-pf} yields that the map:
\[
\begin{gathered}
\Gamma\big(\LocSys_{\bG_m,\log}^{\leq n},\Oblv\Av_!^{\Gr_{\bG_m}^{\leq n},w} (\sO_{\LocSys_{\bG_m,\log}^{\leq n}})\big) = \\
\Gamma^{\IndCoh}\big(\sZ^{\leq n},\widetilde{\zeta}_*^{\,\IndCoh}
\Oblv\Av_!^{\Gr_{\bG_m}^{\leq n},w})(\sO_{\LocSys_{\bG_m,\log}^{\leq n}})\big) = \\
\ul{\Hom}_{\IndCoh^*(\sZ^{\leq n})}\big(\sO_{\sZ^{\leq n}}, 
\widetilde{\zeta}_*^{\,\IndCoh}
\Oblv\Av_!^{\Gr_{\bG_m}^{\leq n},w} (\sO_{\LocSys_{\bG_m,\log}^{\leq n}})\big) \to \\
\ul{\Hom}_{\IndCoh^*(\sY_{\log}^{\leq n})}\big(i_{n,*}^{\IndCoh}\sO_{\sZ^{\leq n}}, 
i_{n,*}^{\IndCoh}\widetilde{\zeta}_*^{\,\IndCoh}
\Oblv\Av_!^{\Gr_{\bG_m}^{\leq n},w} (\sO_{\LocSys_{\bG_m,\log}^{\leq n}})\big)
\end{gathered}
\]

\noindent is an isomorphism. 

By Theorem \ref{t:trun-lcft}, the left hand side identifies canonically with:
\[
\Gamma(\fL_n^+ \bG_m,D_{\fL_n^+ \bG_m}).
\]

Putting this together, we obtain an isomorphism:
\[
\underset{r,\alpha_0\cdot-}{\colim} \, \ul{\End}_{\IndCoh^*(\sY^{\leq n})}(\sF_n)
\simeq 
\Gamma(\fL_n^+ \bG_m,D_{\fL_n^+ \bG_m}).
\] 

\subsubsection{}\label{sss:alpha-alg}

In \S \ref{sss:alpha-0-left}, we constructed an isomorphism:
\[
\underset{r,\alpha_0\cdot-}{\colim} \, \ul{\End}_{\IndCoh^*(\sY^{\leq n})}(\sF_n)
\simeq 
\Gamma(\fL_n^+ \bG_m,D_{\fL_n^+ \bG_m}).
\]

\noindent We claim that the resulting map:

\[
\ul{\End}_{\IndCoh^*(\sY^{\leq n})}(\sF_n) \to 
\underset{r,\alpha_0\cdot-}{\colim} \, \ul{\End}_{\IndCoh^*(\sY^{\leq n})}(\sF_n)
\simeq \Gamma(\fL_n^+ \bG_m,D_{\fL_n^+ \bG_m})
\]

\noindent is in fact (canonically) a map of (DG) algebras, at least
once the target is given the opposite multiplication. 

For convenience, define:
\[
\o{\sF}_n \coloneqq 
\zeta_*^{\IndCoh}\Av_!^{\Gr_{\bG_m}^{\leq n},w}\sO_{\LocSys_{\bG_m,\log}^{\leq n}}.
\]

\noindent Note that there is a canonical morphism:
\[
\sF_n \to \o{\sF}_n.
\]

\noindent We in effect showed that the natural maps:
\[
\ul{\Hom}_{\IndCoh^*(\sY^{\leq n})}(\sF_n,\o{\sF}_n) \to 
\ul{\End}_{\IndCoh^*(\sY^{\leq n})}(\o{\sF}_n)  \leftarrow 
\ul{\End}_{\IndCoh^*(\LocSys_{\bG_m}^{\leq n})}
(\Av_!^{\Gr_{\bG_m}^{\leq n},w}\sO_{\LocSys_{\bG_m,\log}^{\leq n}})
\]

\noindent are isomorphisms. The map:

\[
\ul{\End}_{\IndCoh^*(\sY^{\leq n})}(\sF_n) \to 
\underset{r,\alpha_0\cdot-}{\colim} \, \ul{\End}_{\IndCoh^*(\sY^{\leq n})}(\sF_n)
\isom 
\ul{\Hom}_{\IndCoh^*(\sY^{\leq n})}(\sF_n,\o{\sF}_n)
\]

\noindent is clearly a map of right 
$\ul{\End}_{\IndCoh^*(\sY^{\leq n})}(\sF_n)$-modules.
But $\ul{\Hom}_{\IndCoh^*(\sY^{\leq n})}(\sF_n,\o{\sF}_n)$ carries
a commuting left 
$\ul{\End}_{\IndCoh^*(\sY^{\leq n})}(\o{\sF}_n)$-module structure,
for which it is a free module of rank $1$ by the above. 
As the class field theory isomorphism:
\begin{equation}\label{eq:lcft-diff}
\ul{\End}_{\IndCoh^*(\LocSys_{\bG_m}^{\leq n})}
(\Av_!^{\Gr_{\bG_m}^{\leq n},w}\sO_{\LocSys_{\bG_m,\log}^{\leq n}})
\simeq \Gamma(\fL_n^+ \bG_m,D_{\fL_n^+ \bG_m})
\end{equation}

\noindent is an isomorphism of algebras, we obtain the claim.

\subsubsection{}\label{sss:lcft-compat-left}

The above calculations verify that there are canonical isomorphisms:
\[
\underset{r,\alpha_0\cdot-}{\colim} \, W_n \simeq 
\Gamma(\fL_n^+ \bG_m,D_{\fL_n^+ \bG_m}) \simeq 
\underset{r,\alpha_0\cdot-}{\colim} \, \ul{\End}_{\IndCoh^*(\sY^{\leq n})}(\sF_n).
\]

We claim that this map is induced by our comparison map $f$ from 
\S \ref{sss:ff-induction-strat}.

By \S \ref{sss:alpha-alg}, it suffices to check that the composition:
\[
W_n^{op} \to \ul{\End}_{\IndCoh^*(\sY^{\leq n})}(\sF_n)
\to
\underset{r,\alpha_0\cdot-}{\colim} \, \ul{\End}_{\IndCoh^*(\sY^{\leq n})}(\sF_n)
\simeq \Gamma(\fL_n^+ \bG_m,D_{\fL_n^+ \bG_m})^{op}
\]

\noindent is the restriction map. We first check this on some particular elements.

For the linear functions $t^{-n}k[[t]]dt/k[[t]]dt \subset W_n^{op}$,
this follows by construction. Indeed, the geometric incarnation of the
action of these elements is the observation that 
$\sZ^{\leq n} \subset \sY_{\log}^{\leq n}$ is closed under
the action of $\Gr_{\bG_m}^{pos,\leq n}$. The same is true
for $\LocSys_{\bG_m,\log}^{\leq n} \subset \sY_{\log}^{\leq n}$ \textemdash{}
in fact, it is closed under $\Gr_{\bG_m}^{\leq n}$. Tracing the
constructions, these observations provide the claim.

Next, we claim that the vector fields on $\fL_n^+\bA^1$ defined by the
infinitesimal $\fL_n^+\bG_m$-action: 
\[
k[[t]]/t^n \simeq \Lie(\fL_n^+\bG_m) \subset W_n^{op}
\]

\noindent have the correct image. Indeed, this follows by 
construction of the class field theory isomorphism and Lemma \ref{l:hc-calc}.

By the above, we see that the algebra map:

\[
W_n^{op} \to \Gamma(\fL_n^+ \bG_m,D_{\fL_n^+ \bG_m})^{op}
\]

\noindent factors through the (non-commutative) localization of
$W_n^{op}$ at $\alpha_0$, i.e., through a map:

\[
W_n^{op} \to \Gamma(\fL_n^+ \bG_m,D_{\fL_n^+ \bG_m})^{op} \xar{\gamma}
\Gamma(\fL_n^+ \bG_m,D_{\fL_n^+ \bG_m})^{op}
\]

\noindent for some algebra map $\gamma$. As the subspaces
$t^{-n}k[[t]]dt/k[[t]]dt$ and $k[[t]]/t^n$ 
generate $\Gamma(\fL_n^+ \bG_m,D_{\fL_n^+ \bG_m})^{op}$ as an algebra, 
and as we have seen $\gamma$ is the identity on these elements, $\gamma$
itself must be the identity map, as desired.

\subsubsection{}

Next, we calculate:
\[
\underset{r,-\cdot \alpha_0}{\colim} \, \ul{\End}_{\IndCoh^*(\sY^{\leq n})}(\sF_n).
\]

For $r \in \bZ$, let $\presup{r}{\sZ^{\leq n}} \subset \sY_{\log}^{\leq n}$
denote the closed where the meromorphic section $s$ has a pole of order $\leq r$;
therefore, $\presup{0}{\sZ^{\leq n}} = \sZ^{\leq n}$,
$\iota(\sZ^{\leq n}) = \presup{-1}{\sZ^{\leq n}}$, and $\alpha_0$ can be
considered as obtained by restriction 
$\presup{r}{\sZ^{\leq n}} \to \presup{r-1}{\sZ^{\leq n}}$ for any $r$.

We let $i_{n,r}:\presup{r}{\sZ^{\leq n}} \to \sY_{\log}^{\leq n}$ denote
the embedding.

For any $\sG \in \IndCoh^*(\sY^{\leq n})$, it follows by definition
that:
\[
\underset{r,-\cdot \alpha_0}{\colim} \, \ul{\Hom}_{\IndCoh^*(\sY^{\leq n})}(\sF_n,\sG) =
\underset{r,-\cdot \alpha_0}{\colim} \, \ul{\Hom}_{\IndCoh^*(\sY^{\leq n})}
(i_{n,r,*}^{\IndCoh}(\sO_{\presup{r}\sZ^{\leq n}}),\Oblv\sG).
\]

\noindent We have a canonical isomorphism:
\[
\underset{r,-\cdot \alpha_0}{\colim} \, \ul{\Hom}_{\IndCoh^*(\sY^{\leq n})}
(i_{n,r,*}^{\IndCoh}(\sO_{\presup{r}\sZ^{\leq n}}),\Oblv\sG) \isom 
\Gamma^{\IndCoh}(\sY_{\log}^{\leq n},\Oblv \sG).
\]

Now taking $\sG = \sF_n$, we see that:
\[
\underset{r,-\cdot \alpha_0}{\colim} \, \ul{\End}_{\IndCoh^*(\sY^{\leq n})}(\sF_n)
\simeq \Gamma^{\IndCoh}(\sY_{\log}^{\leq n},\Oblv \Av_!^{\Gr_{\bG_m}^{\leq n},w}
i_{n,*}^{\IndCoh}\sO_{\sZ^{\leq n}}).
\]

By the analysis of \S \ref{sss:gr-hom},\footnote{Specifically, the 
observations that the ring $A_n$ from 
\emph{loc. cit}. is positively graded with degree $0$ component corresponding to
$\LocSys_{\bG_m,\log}^{\leq n} \subset \sZ^{\leq n} = \Spec(A_n)/\bG_m$.}, 
the map:
\[
\Gamma^{\IndCoh}(\sY_{\log}^{\leq n},\Oblv \Av_!^{\Gr_{\bG_m}^{\leq n},w}
i_{n,*}^{\IndCoh}\sO_{\sZ^{\leq n}}) \to 
\Gamma^{\IndCoh}(\sY_{\log}^{\leq n},\Oblv \Av_!^{\Gr_{\bG_m}^{\leq n},w}
i_{n,*}^{\IndCoh}\widetilde{\zeta}_*^{\,\IndCoh}\sO_{\LocSys_{\bG_m,\log}^{\leq n}})
\]

\noindent is an isomorphism. By functoriality, the target identifies with:
\[
\Gamma(\LocSys_{\bG_m,\log}^{\leq n},
\Oblv \Av_!^{\Gr_{\bG_m}^{\leq n},w}
\sO_{\LocSys_{\bG_m,\log}^{\leq n}}).
\] 

Now passing to the colimit over $\alpha_0$ acting on the left and
right simultaneously, the above analysis actually shows that the maps: 
\[
\underset{r,-\cdot \alpha_0}{\colim} \, \ul{\End}_{\IndCoh^*(\sY^{\leq n})}(\sF_n)
\to \underset{(r_1,r_2),-\cdot \alpha_0,-\cdot \alpha_0}{\colim} \, \ul{\End}_{\IndCoh^*(\sY^{\leq n})}(\sF_n)
\leftarrow 
\underset{r,\alpha_0 \cdot -}{\colim} \, \ul{\End}_{\IndCoh^*(\sY^{\leq n})}(\sF_n.
\]

\noindent are isomorphisms. It then follows from \S \ref{sss:lcft-compat-left} that
the comparison map:

\[
\underset{r,-\cdot \alpha_0}{\colim} \, W_n \to 
\underset{r,-\cdot \alpha_0}{\colim} \, \ul{\End}_{\IndCoh^*(\sY^{\leq n})}(\sF_n)
\]

\noindent is an isomorphism. 

\subsubsection{Hamiltonian reduction}

It remains to check that the third map in \eqref{eq:weyl-bimod} 
is an isomorphism. We will see that this amounts to our
inductive hypothesis.

\subsubsection{}\label{sss:qh-1}

We begin with some notation.

Suppose $A$ is a (DG) algebra, and that we are given a map:
\[
k[t] \to A.
\] 

\noindent We let $\alpha \in \Omega^{\infty} A$ denote the image of $t$.

In this case, we let:
\[
\on{QH}_{\alpha}(A) \coloneqq A^{\alpha\cdot-}/-\cdot \alpha 
\]

\noindent with notation on the right hand side understood as
in Lemma \ref{l:weyl-bimod}, i.e., we take the homotopy quotient of 
right multiplication by $\alpha$ and then the homotopy kernel of the
left action of $\alpha$ (and of course, the order of these two operations is
not important).

As is standard, we refer to $\on{QH}_{\alpha}(A)$ 
as the \emph{quantum Hamiltonian reduction} 
of $A$ by $\alpha$.

\subsubsection{}

First, we note that $\on{QH}_{\alpha_0}(W_1^{op}) = k$ (with $1 \in k$ mapping to the
unit in $W_1$), so similarly, 
$\on{QHH}_{\alpha_0}(W_n^{op}) = W_{n-1}^{op}$. This is the left hand side of
the third map in \eqref{eq:weyl-bimod} in our setting.

\subsubsection{}

Next, we calculate the right hand side of \eqref{eq:weyl-bimod}.

First, we have:
\[
\ul{\End}_{\IndCoh^*(\sY^{\leq n})}(\sF_n) \isom 
\ul{\End}_{\IndCoh^*(\sY)}(\lambda_{n,*}^{\IndCoh}(\sF_n))
\]

\noindent by Remark \ref{r:kash-y} (and the assumption that $n>0$).

By \eqref{eq:fund-ses}, we then have:
\[
\ul{\End}_{\IndCoh^*(\sY)}\big(\lambda_{n,*}^{\IndCoh}(\sF_n)\big)^{\alpha_0\cdot-} = 
\ul{\Hom}_{\IndCoh^*(\sY)}\big(\lambda_{n,*}^{\IndCoh}(\sF_n),
\lambda_{n-1,*}^{\IndCoh}(\sF_{n-1})\big). 
\]

\noindent Applying \eqref{eq:fund-ses} again, we have:
\begin{equation}\label{eq:qh-end}
\begin{gathered}
\on{QH}_{\alpha_0}\Big(\ul{\End}_{\IndCoh^*(\sY)}\big(\lambda_{n,*}^{\IndCoh}(\sF_n)\big)\Big) 
\coloneqq \\
\ul{\End}_{\IndCoh^*(\sY)}\big(\lambda_{n,*}^{\IndCoh}(\sF_n)\big)^{\alpha_0\cdot-}/-\cdot\alpha_0 =
\ul{\End}_{\IndCoh^*(\sY)}\big(\lambda_{n-1,*}^{\IndCoh}(\sF_{n-1})\big).
\end{gathered}
\end{equation}

By induction, we know that the right hand side identifies with 
$W_{n-1}^{op}$ as an algebra. Below, we will show that the
comparison map from Lemma \ref{l:weyl-bimod} coincides with this
identification.

\subsubsection{}

Let us return to the general setting of \S \ref{sss:qh-1}.

First, observe that $\on{QH}_{\alpha}(A)$ carries a canonical algebra structure.
Indeed, it obviously identifies with:
\[
\ul{\End}_{A\mod}(A/-\cdot\alpha) = 
\ul{\End}_{A\mod}(A \underset{k[t]}{\otimes} k).
\]

\noindent (Here $k[t]$ acts on $k$ with $t$ acting by $0$.)
We equip $A^{\alpha\cdot-}/-\cdot \alpha$ with the multiplication 
\emph{opposite} to this one.\footnote{This convention is natural 
because $\ul{\End}_{A\mod}(A) = A^{op}$.}

Suppose now that $A$ and $B$ are (DG) algebras, and that we are given a map:
\[
B[t] \coloneqq B \otimes k[t] \to A.
\] 

We can further identify:
\[
A \underset{k[t]}{\otimes} k \simeq A \underset{B[t]}{\otimes} B.
\]

\noindent I.e., $A \underset{k[t]}{\otimes} k$ is canonically 
an $(A,B)$-bimodule. Therefore, by construction, we obtain
a canonical morphism: 
\[
B \to \on{QH}_{\alpha}(A).
\]

\begin{rem}

We note that the diagram:
\[
\xymatrix{
B \ar[rr] \ar[d] && \on{QH}_{\alpha}(A)\ar[d] \\
A \ar[rr] && A/-\cdot\alpha
}
\]

\noindent commutes by construction. 

\end{rem}

\subsubsection{}

We apply the above with:
\[
A = \ul{\End}_{\IndCoh^*(\sY^{\leq n})}(\sF_n),
\alpha = \alpha_0, B = W_{n-1}^{op}.
\]

We first see that:
\[
\on{QH}_{\alpha_0}\Big(\ul{\End}_{\IndCoh^*(\sY)}\big(\lambda_{n,*}^{\IndCoh}(\sF_n)\big)\Big) 
\]

\noindent carries a canonical algebra structure; by construction, 
\eqref{eq:qh-end} is an isomorphism of algebras, where the right hand side
is given the usual algebra structure on endomorphisms. 

Moreover, we deduce that the map:
\[
W_{n-1}^{op} \to \on{QH}_{\alpha_0}\Big(\ul{\End}_{\IndCoh^*(\sY)}\big(\lambda_{n,*}^{\IndCoh}(\sF_n)\big)\Big).
\]

\noindent (coming from Lemma \ref{l:weyl-bimod}) is an algebra morphism.

\subsubsection{}

Finally, we claim that the resulting map:
\[
W_{n-1}^{op} \to \on{QH}_{\alpha_0}\Big(\ul{\End}_{\IndCoh^*(\sY)}\big(\lambda_{n,*}^{\IndCoh}(\sF_n)\big)\Big) = 
\ul{\End}_{\IndCoh^*(\sY)}\big(\lambda_{n-1,*}^{\IndCoh}(\sF_{n-1})\big)
\]

\noindent coincides with the construction of \S \ref{s:weyl} (for $n-1$ instead of $n$).

Indeed, as this is an algebra map, we can check this on generators of
$W_{n-1}$. There it is essentially tautological from the constructions.

In particular, we see by induction that this map is an isomorphism.
Therefore, the conditions of Lemma \ref{l:weyl-bimod} are satisfied,
so we have proved Proposition \ref{p:ff}.

\section{Proof of the main theorem}\label{s:main}

\subsection{Overview}

In this section we prove Theorem \ref{t:main}. The idea is simple given the
work we have done so far: bootstrap from the functors $\Delta_n$ constructed
earlier.

\subsection{Arcs}\label{ss:delta-arcs}

We begin by constructing a strongly $\fL^+\bG_m$-equivariant functor:
\[
\Delta_{\infty}:D^!(\fL^+\bA^1) \to \IndCoh^*(\sY).
\]

\subsubsection{Notation}

In what follows, we let:
\[
p_n:\fL^+ \bA^1 \to \fL_n^+\bA^1
\]

\noindent denote the structural map.
For $m \geq n$, we similarly let: 
\[
p_{n,m}:\fL_m^+\bA^1 \to \fL_n^+\bA^1 
\]

\noindent denote the structural map.

\subsubsection{Overview}

First, let us unwind what it means to construct such a functor
$\Delta_{\infty}$.

Recall from \cite{dmod} that by definition we have: 
\[
D^!(\fL^+\bA^1) \coloneqq \colim_{n \geq 0} \, D(\fL_n^+\bA^1)  
\]

\noindent where the structural functors are the standard 
$!$-pullback functors of $D$-module theory. We let: 
\[
p_n^!: D(\fL_n^+\bA^1) \to 
D^!(\fL^+\bA^1)
\]

\noindent denote the structural functor.

The above colimit is a colimit of $D^*(\fL^+\bG_m)$-module categories.
Moreover, as it is indexed by $\bZ^{\geq 0}$, it suffices to 
construct functors:
\begin{equation}\label{eq:term-wise}
D(\fL_n^+\bA^1) \to \IndCoh^*(\sY) \in 
D^*(\fL^+\bG_m)\mod 
\end{equation}

\noindent and commutative diagrams:
\begin{equation}\label{eq:delta-infty-comm}
\vcenter{\xymatrix{
D(\fL_n^+\bA^1) \ar[drr]\ar[d]_{p_{n,n+1}^!} && \\ 
D(\fL_{n+1}^+\bA^1) \ar[rr] && \IndCoh^*(\sY) 
}}
\end{equation}

\noindent of $D^*(\fL^+\bG_m)$-module categories for each $n$.
(Here $D^*(\fL^+\bG_m)\mod $ acts on $\IndCoh^*(\sY)$ via
geometric class field theory, as in \S \ref{s:cft}.)

\subsubsection{}

For us, by definition, the functor \eqref{eq:term-wise}
is $\Delta_n$. It remains to construct the commutative 
diagrams \eqref{eq:delta-infty-comm} as above. 

In \S \ref{sss:diff-func}-\ref{sss:diff-compat}, we will formulate 
a compatibility for these commutative diagrams that 
characterizes them uniquely. We then turn to proving their existence.

\subsubsection{Preliminary constructions with differential operators}\label{sss:diff-func}

We digress to some general constructions with $D$-modules, continuing
the discussion from \S \ref{ss:diff-1}. 

Suppose $\pi:X \to Y$ is a morphism between smooth varieties. In this
case, one has a canonical natural isomorphism:
\[
\Oblv^{\ell} \pi^! \simeq \pi^*\Oblv^{\ell}
\]

\noindent of functors:
\[
D(Y) \to \QCoh(X).
\]

\noindent Here $\pi^*$ is the quasi-coherent pullback functor. We refer
to \cite{crystals} for the construction.

The canonical map:
\[
\sO_Y \to \Oblv^{\ell}\ind^{\ell}(\sO_Y) = \Oblv^{\ell}(D_Y)
\]

\noindent then gives rise to a map:
\[
\sO_X \to \Oblv^{\ell}(D_Y) \simeq 
\Oblv^{\ell}\pi^!(D_Y)
\]

\noindent so by adjunction a map:
\begin{equation}\label{eq:can-x-y}
\on{can}:D_X \to \pi^!(D_Y).
\end{equation}

In the special case where $\pi$ is smooth, the map $\on{can}$
is between objects concentrated in cohomological degree $-\dim X$, and
is an epimorphism in that abelian category.

Finally, we note that if an algebraic group $G$ acts on $X$ and $Y$, and the map
$\pi$ is $G$-equivariant, the map $\on{can}$ upgrades canonically to a map 
of $G$-weakly equivariant $D$-modules (by functoriality of the above constructions).  

\subsubsection{A compatibility}\label{sss:diff-compat}

Let us now return to our setting. For each $n$, the above
provides canonical morphisms:
\[
\on{can}_n:
D_{\fL_{n+1}^+ \bA^1} \to p_{n,n+1}^! D_{\fL_n^+ \bA^1} \in 
D(\fL_{n+1}^+ \bA^1)^{\fL_{n+1}^+ \bG_m,w}.
\]

Suppose we are given a commutative diagram \eqref{eq:delta-infty-comm}.
We then obtain a morphism:
\begin{equation}\label{eq:delta-can-n}
\begin{gathered}
i_{n+1,*}^{\IndCoh}(\sO_{\sZ^{\leq n+1}}) \eqqcolon 
\Delta_{n+1}(D_{\fL_{n+1}^+ \bA^1}) \xar{\on{\Delta_{n+1}(\on{can}_n)}}
\Delta_{n+1}(p_{n,n+1}^! D_{\fL_n^+ \bA^1}) 
\overset{\eqref{eq:delta-infty-comm}}{\simeq} \\
\Delta_n(D_{\fL_n^+ \bA^1}) 
\coloneqq
i_{n,*}^{\IndCoh}(\sO_{\sZ^{\leq n}}) \in \IndCoh^*(\sY)^{\fL^+\bG_m,w} \simeq 
\IndCoh^*(\sY_{\log}).
\end{gathered}
\end{equation}

\noindent (We have omitted terms $\lambda_{-,*}^{\IndCoh}$ to simplify 
the notation. We are also abusing notation in letting
e.g. $\Delta_n$ denote the induced functor on 
weakly equivariant categories.)

On the other hand, there is another such map coming from the embedding
$\delta_{n+1}:\sZ^{\leq n} \into \sZ^{\leq n+1}$ and the adjunction map:
\begin{equation}\label{eq:delta-delta}
\sO_{\sZ^{\leq n+1}} \to \delta_{n+1,*}(\sO_{\sZ^{\leq n}}).
\end{equation}

We will show:

\begin{prop}\label{p:delta-infty-comm}

There exist unique commutative diagrams \eqref{eq:delta-infty-comm}
so that the resulting map \eqref{eq:delta-can-n} is \eqref{eq:delta-delta}.

\end{prop}

In this remainder of this subsection, we prove Proposition \ref{p:delta-infty-comm}.

\subsubsection{}\label{sss:hc-comm-diags}

Let us make a preliminary observation in a general setting. Fix $G$
an affine algebraic group.

Suppose that $A$ is an associative algebra with a 
$G$-action and a Harish-Chandra datum as in \S \ref{ss:hc}, so 
$G$ acts strongly on $A\mod$.

Suppose we are given $\sC,\sD \in G\mod$ and functors:
\[
\begin{gathered}
F_1:A\mod \to \sC \\
F_2:\sC \to \sD \\
F_3:A \mod \to \sD.
\end{gathered}
\]

We let:
\[
\sF_1 \coloneqq F_1(A) \in \sC^{G,w}, \sF_3 \coloneqq F_3(A) \in \sD^{G,w}.
\]

Moreover, to remove discussion of higher homotopical considerations,
we assume:

\begin{itemize}

\item $A$ is classical.

\item $\ul{\End}_{\sC}(\sF_1)$ and $\ul{\End}_{\sD}(\sF_3)$ are classical.

\end{itemize}

Then, as in \S \ref{ss:hc}, the 
functor $F_1$ is completely encoded in the datum of compact
object $\sF_1 \in \sC^{G,w}$ with its action of $A^{op} \in \Alg(\Rep(G))$ by 
endomorphisms (the data should satisfy the property that the action of
$\fg$ through $i:\fg \to A$ coincides with the obstruction to strong equivariance).
The same applies for $F_3$.

Therefore, we see that (under the homotopically simplifying assumptions
above), the data of a commutative diagram:
\[
\xymatrix{
A\mod \ar[d]_{F_1} \ar[drr]^{F_3} && \\ 
\sC \ar[rr]^{F_2} && \sD 
}
\]

\noindent is equivalent to giving an isomorphism:
\[
F_2(\sF_1) \simeq \sF_3 \in \sD^{G,w}
\]

\noindent compatible with the actions of $A^{op} \in \Rep(G)$ on both sides.
(In particular, under these assumptions, it is equivalent to construct
a commutative diagram of strongly or weakly $G$-equivariant categories.)

\subsubsection{}

We apply this in our setting as follows. 

First, the actions on the source categories factor through the localization 
$D^*(\fL^+\bG_m) \to D(\fL_{n+1}^+\bG_m)$, and the functors
factor through the subcategory $\IndCoh^*(\sY^{\leq n+1}) \subset \IndCoh^*(\sY)$,
through which $D(\fL_n^+\bG_m)$ acts. So the question only concerns the
action of an affine algebraic group, not an affine group scheme. 

We then see
it suffices to show that there is a unique isomorphism:
\begin{equation}\label{eq:comm-iso}
\Delta_{n+1}(p_{n,n+1}^! D_{\fL_n^+ \bA^1}) 
\simeq 
\Delta_n(D_{\fL_n^+ \bA^1}) 
\coloneqq
i_{n,*}^{\IndCoh}(\sO_{\sZ^{\leq n}})
\end{equation}

\noindent such that: 

\begin{itemize}

\item The two resulting maps from $i_{n+1,*}^{\IndCoh}(\sO_{\sZ^{\leq n+1}})$
(one being \eqref{eq:delta-can-n}, the other being the canonical map)
coincide.

\item The two resulting actions of $W_n^{op}$ on 
$\Av_!^{\Gr_{\bG_m}^{\leq n}}i_{n,*}^{\IndCoh}(\sO_{\sZ^{\leq n}})$ (coming from 
the tautological action of 
$W_n^{op}$ on $D_{\fL_n^+\bA^1} \in D(\fL_n^+\bA^1)$, and functoriality of
$\Delta_{n+1}p_{n,n+1}^!$ and $\Delta_n$ respectively) coincide
under the isomorphism.

\end{itemize}

As the canonical map:
\[
i_{n+1,*}^{\IndCoh}(\sO_{\sZ^{\leq n+1}}) \to 
i_{n,*}^{\IndCoh}(\sO_{\sZ^{\leq n}}) \in \Coh(\sY_{\log})^{\heart}
\]

\noindent is an epimorphism, we see that the first point above 
implies that such an isomorphism is unique if it exists. 

To see existence, it is convenient to fix a coordinate $t$ on the
disc. We use the
notation $\alpha$ notation of \S \ref{sss:ff-induction-strat}. 

We have an evident short exact sequence:
\[
0 \to D_{\fL_{n+1}^+ \bA^1}(-1) \xar{-\cdot \partial\alpha_n}
D_{\fL_{n+1}^+ \bA^1} \xar{\on{can}_n} 
p_{n,n+1}^! D_{\fL_n^+ \bA^1} \to 0 
\]

\noindent in the abelian category 
$D(\fL_{n+1}^+\bA^1)^{\fL_{n+1}^+ \bG_m,w,\heart}[n+1]$.
Here in the first term, the twist $(-1)$ indicates that we modify the
weakly equivariant structure by tensoring with the representation 
$k(-1) \in \Rep(\fL^+\bG_m)$ obtained by restriction along $\ev$ from the
standard representation of $\bG_m$.

By construction of $\Delta_{n+1}$, on weakly equivariant
categories, we have:
\[
\begin{gathered}
\Delta_{n+1}(D_{\fL_{n+1}^+ \bA^1}) \simeq i_{n+1,*}^{\IndCoh}(\sO_{\sZ^{\leq n+1}}) \\
\Delta_{n+1}(D_{\fL_{n+1}^+ \bA^1}(-1)) 
\simeq i_{n+1,*}^{\IndCoh}\iota_*^{\IndCoh}(\sO_{\sZ^{\leq n+1}}).
\end{gathered}
\]

\noindent Moreover, it is easy to see that $\Delta_{n+1}(-\cdot \partial_{\alpha_n})$
goes to the map induced from:
\[
\iota_*^{\IndCoh}(\sO_{\sZ^{\leq n+1}}) \xar{b_{-n}\cdot-} 
\sO_{\sZ^{\leq n+1}},
\]

\noindent i.e., the left arrow in \eqref{eq:fund-ses-2}. From 
\eqref{eq:fund-ses-2}, it follows that there is a unique
isomorphism \eqref{eq:comm-iso} compatible with the projection from 
$i_{n+1,*}^{\IndCoh}(\sO_{\sZ^{\leq n+1}})$. It is straightforward
to check that this isomorphism is compatible with the action of
$W_n^{op}$, concluding the argument.

\subsection{Compatibility with translation}

We now establish an additional equivariance property 
for the functor $\Delta_{\infty}$.

Specifically, recall the positive loop space $\fL^{pos} \bG_m \subset \fL \bG_m$ 
for $\bG_m$, as defined in \S \ref{ss:pos-gr}. 
This space is an ind-closed submonoid of $\fL \bG_m$ containing
$\fL^+\bG_m$, so we have 
a fully-faithful monoidal functor:
\[
D^*(\fL^{pos} \bG_m) \subset D^*(\fL \bG_m).
\]

By construction, $\fL^{pos} \bG_m$ acts on $\fL^+\bA^1$, so 
$D^*(\fL^{pos} \bG_m)$ acts on $D^!(\fL^+\bA^1)$.

On the other hand, we have the local class field theory isomorphism:
\[
D^*(\fL \bG_m) \simeq \QCoh(\LocSys_{\bG_m})
\]

\noindent defining a (fully faithful, symmetric) monoidal functor:
\[
D^*(\fL^{pos} \bG_m) \to \QCoh(\LocSys_{\bG_m})
\]

\noindent and so an action of $D^*(\fL^{pos} \bG_m)$ on 
$\IndCoh^*(\sY)$.

Our goal is to canonically upgrade the strong 
$\fL^+\bG_m$-equivariance of $\Delta_{\infty}$ to make $\Delta_{\infty}$
into a morphism of $D^*(\fL^{pos} \bG_m)$-module categories.

\subsubsection{}

Before proceeding, we wish to pin down this equivariance property uniquely.

Suppose $\sC$ is a $D^*(\fL^{pos} \bG_m)$-module category. 
We obtain a canonical functor $T:\sC^{\fL^+\bG_m} \to \sC^{\fL^+\bG_m}$
given by the action of (the $\delta$ $D$-module at) 
some (equivalently any) uniformizer $t \in \fL^{pos} \bG_m$.

Our uniqueness property will involve the corresponding functors $T$ on both
sides. Below, we collect preliminary observations regarding how the
functor $T$ interacts with the source and target of our functor. 

\subsubsection{} 

Suppose $\sC = D^!(\fL^+\bA^1)$. A choice of uniformizer $t$ defines
a closed embedding $\mu_t:\fL^+\bA^1 \into \fL^+ \bA^1$ given by multiplication
by $t$. The corresponding functor $T$ above is given by the left adjoint
to the pullback functor:
\[
\mu_t^!:D^!(\fL^+\bA^1) \to D^!(\fL^+ \bA^1).
\]

\noindent I.e., in the notation of \cite{dmod}, we have $T = \mu_{t,*,!-dR}$.

By adjunction, we observe that there is a canonical map:
\begin{equation}\label{eq:can-unit-geom}
\on{can}:T(\omega_{\fL^+\bA^1}) = \mu_{t,*,!-dR}(\omega_{\fL^+\bA^1}) = 
\mu_{t,*,!-dR}\mu_t^!(\omega_{\fL^+\bA^1}) \to 
\omega_{\fL^+\bA^1}.
\end{equation}

\subsubsection{}

Now suppose $\sC = \IndCoh^*(\sY)$. Then we have:
\[
\begin{gathered}
\IndCoh^*(\sY)^{\fL^+\bG_m} \simeq \IndCoh^*(\sY)_{\fL^+\bG_m} \simeq \\
\QCoh(\LocSys_{\bG_m}^{\leq 0}) \underset{\QCoh(\LocSys_{\bG_m})}{\otimes}
\IndCoh^*(\sY) \subset 
\IndCoh^*(\sY \underset{\LocSys_{\bG_m}}{\times} \LocSys_{\bG_m}^{\leq 0}) = 
\IndCoh^*(\sY_{\log}^{\leq 0}).
\end{gathered}
\]

Under the above
identifications, the functor $T$ is given by tensoring with 
the \emph{dual}\footnote{Because the Contou-Carr\`ere pairing
is skew symmetric, one has to fix a convention in setting up geometric
class field theory. The sign compatibility was implicitly fixed in 
\S \ref{s:aj}. It is a straightforward check to trace the constructions
to see that our sign conventions force the appearance of the dual to the
standard representation; essentially, we are on the ``dual" side of the
Contou-Carr\`ere duality.}
to the standard representation of $\bB \bG_m$, i.e., tensoring with the
bundle $\sO_{\sY^{\leq 0}}(1)$ considered earlier.

Now observe that there is a canonical map:
\begin{equation}\label{eq:can-unit-spec}
\on{can}:T(\sO_{\sZ^{\leq 0}}) = 
\sO_{\sZ^{\leq 0}}(1) \to \sO_{\sZ^{\leq 0}}
\end{equation}

\noindent from \eqref{eq:o(1)}.

\subsubsection{}

We can now formulate our results precisely. 

Observe that by construction, we have
an isomorphism:
\[
\Delta_{\infty}(\omega_{\fL^+\bA^1}) = \Delta_0(\omega_{\fL_0^+\bA^1}) \simeq
i_{0,*}^{\IndCoh}(\sO_{\sZ^{\leq 0}}) \in \IndCoh^*(\sY_{\log}^{\leq 0}).
\]

\noindent (Here we consider $\omega_{\fL^+\bA^1}$ as a $\fL^+\bG_m$-equivariant
object.)

\begin{prop}

There is a unique structure of 
$D^*(\fL^{pos}\bG_m)$-equivariance on the functor $\Delta_{\infty}$
such that:

\begin{itemize}

\item The underlying $D^*(\fL^+\bG_m)$-equivariance is the one
constructed in \S \ref{ss:delta-arcs}.

\item The map:
\[
T(\sO_{\sZ^{\leq 0}}) \to \sO_{\sZ^{\leq 0}} \in \IndCoh^*(\sY_{\log}^{\leq 0})
\]

\noindent obtained by functoriality from:
\[
\Delta_{\infty}(\on{can}:T(\omega_{\fL^+\bA^1}) \xar{\eqref{eq:can-unit-geom}}  
\omega_{\fL^+\bA^1})
\]

\noindent coincides with \eqref{eq:can-unit-spec}.

\end{itemize}

\end{prop}

\begin{proof}

We begin with the existence. For this, it is convenient to fix a 
uniformizer $t$ once and for all. Note that this choice
induces a morphism $\LocSys_{\bG_m} \to \bB \bG_m$ using 
Proposition \ref{p:locsys}. We let $\sO_{\LocSys_{\bG_m}}(-1)$ and
$\sY(-1)$ denote
the corresponding line bundles, and generally use notation as in
\S \ref{sss:z-pic}. 

We obtain a splitting:
\[
D^*(\fL^{pos} \bG_m) \simeq D^*(\fL^+ \bG_m) \otimes D(\bZ^{\geq 0})
\]

\noindent of (symmetric) monoidal DG categories. Therefore, we find
that $D^*(\fL^{pos} \bG_m)$-module categories are equivalent to
$D^*(\fL^+\bG_m)$-module categories with an endofunctor.
Therefore, the existence amounts to showing that $\Delta_{\infty}$
intertwines $\mu_{t,*,!-dR}$ with tensoring with 
$\sO_{\sY}(1)$.

Unwinding the constructions, this amounts to constructing commutative diagrams
of strong $\fL^+\bG_m$-module categories:
\[
\xymatrix{
D(\fL_n^+\bA^1) \ar[rr]^{\mu_{t,*,dR}}\ar[d]^{\Delta_n} && D(\fL_{n+1}^+\bA^1) \ar[d]^{\Delta_{n+1}} \\
\IndCoh^*(\sY) \ar[rr]^{\sO_{\sY}(1) \otimes -} &&  \IndCoh^*(\sY)
}
\] 

\noindent compatibly in $n$ (using the $p_{-}^!$ functors and 
the commutativity data of Proposition \ref{p:delta-infty-comm});
here we have also let $\mu_t$ denote the induced map
$\fL_n^+\bA^1 \into \fL_{n+1}^+\bA^1$ given by multiplication with $t$.

Using the logic of \eqref{sss:hc-comm-diags} (as in 
the proof of Proposition \ref{p:delta-infty-comm}),
we see that this amounts to constructing
isomorphisms:
\[
\sO_{\sZ^{\leq n}}(1) \simeq \Delta_{n+1}(\mu_{t,*,dR} D_{\fL_n^+\bA^1}) \in 
\IndCoh^*(\sY_{\log}^{\leq n+1})
\]

\noindent satisfying some compatibilities (that do not involve higher
homotopy theory).

For this, we observe that there is a standard short exact sequence:
\[
0 \to D_{\fL_{n+1}^+\bA^1} \xar{-\cdot \alpha_0} D_{\fL_{n+1}^+\bA^1}(1) \xar{}
\mu_{t,*,dR} D_{\fL_n^+\bA^1}[1] \to 0 
\]

\noindent in the abelian category $D(\fL_{n+1}^+\bA^1)^{\fL_{n+1}\bG_m,w,\heart}[n+1]$
(with notation as before). 
The map $-\cdot \alpha_0$ maps under $\Delta_{n+1}$ to the canonical
map:
\[
\sO_{\sZ^{\leq n+1}} \to \iota_*^{\IndCoh}(\sO_{\sZ^{\leq n+1}}).
\]

Therefore, we obtain an isomorphism:
\[
\Delta_{n+1}(\mu_{t,*,dR} D_{\fL_n^+\bA^1}) \simeq 
\Coker(\sO_{\sZ^{\leq n+1}}[-1] \to \iota_*^{\IndCoh}(\sO_{\sZ^{\leq n+1}}[-1])) 
\simeq 
\Ker(\sO_{\sZ^{\leq n+1}} \to \iota_*^{\IndCoh}(\sO_{\sZ^{\leq n+1}})).
\]

\noindent The latter is canonically isomorphism to 
$\delta_{n+1,*}^{\IndCoh}\sO_{\sZ^{\leq n}}(1)$ by \eqref{eq:fund-ses}, 
as desired.

This isomorphism is easily seen to be compatible with the actions of $W_n^{op}$ 
and with varying $n$, proving existence.

For uniqueness, we observe that $\Delta_{\infty}$ is fully faithful
(cf. Lemma \ref{l:colim-ff} below), so the observation follows from the 
observation that:
\[
k \isom \End_{\TwoEnd_{\fL^+\bG_m\mod}(D^!(\fL^+\bA^1))}(\mu_{t,!,dR}) \in \Gpd
\]

\noindent (which e.g. follows by replacing functors with kernels
on the square $\fL^+\bA^1 \times \fL^+\bA^1$), 
meaning that any isomorphism we construct is unique up to scalars;
clearly the second compatibility in the proposition pins down this
scalar multiple uniquely.

\end{proof}

\subsection{Conclusion of the argument}

We now have the following precise form of our main theorem.

\begin{thm}\label{t:main-precise}

There is a unique equivalence:
\[
\Delta:D^!(\fL\bA^1) \isom  \IndCoh^*(\sY) 
\]

\noindent of $D^*(\fL\bG_m)$-module categories such that the induced morphism:
\[
D^!(\fL^+\bA^1) \subset D^!(\fL\bA^1) \to \IndCoh^*(\sY) 
\]

\noindent of $D^*(\fL^{pos}\bG_m)$-module categories coincides with the functor
$\Delta_{\infty}$ from above.

\end{thm}

\begin{proof}

There is a natural functor:
\[
D^*(\fL\bG_m) \underset{D^*(\fL^{pos}\bG_m)}{\otimes} D^!(\fL^+\bA^1) \to 
D^!(\fL\bA^1)
\]

\noindent that we observe is an equivalence. Indeed, this follows
by choosing a coordinate $t$ on the disc as before, splitting
$\fL\bG_m$ as $\fL^+\bG_m \times \bZ$, which then implies: 
\[
D^*(\fL\bG_m) \underset{D^*(\fL^{pos}\bG_m)}{\otimes} D^!(\fL^+\bA^1) \simeq 
\colim \, \big(D^!(\fL^+\bA^1) \xar{\mu_{t,*,!-dR}} D^!(\fL^+\bA^1) 
\xar{\mu_{t,*,!-dR}} \ldots \big).
\]

\noindent Using $t^{-n}$ to identify the $n$th term here with 
$t^{-n} \cdot \fL^+ \bA^1 \subset \fL \bA^1$ and applying the
definition from \cite{dmod}, we see that this 
colimit canonically identifies with $D^!(\fL\bA^1)$ compatibly with the
natural functor considered above.

Therefore, there is a unique functor:
\[
\Delta:D^!(\fL\bA^1) \isom  \IndCoh^*(\sY) 
\]

\noindent of $D^*(\fL\bG_m)$-module categories restricting to 
$\Delta_{\infty}$ (compatibly with the $D^*(\fL^{pos}\bG_m)$-equivariance).
It remains to show that $\Delta$ is an equivalence.

First, we note that $\Delta_{\infty}$ is fully faithful by 
Proposition \ref{p:ff} and Lemma \ref{l:colim-ff} below. Next,
we deduce that $\Delta$ itself is fully faithful by another
application of Lemma \ref{l:colim-ff}. 

Therefore, it remains to show that $\Delta$ is essentially surjective.
Clearly its essential image contains each object:
\[
\pi_{n,*}^{\IndCoh}(\sO_{\sZ^{\leq n}}(m))
\]

\noindent for all $n \geq 0$, $m \in \bZ$. Therefore, it suffices
to show that $\IndCoh^*(\sY)$ is generated under colimits by these objects.
But this follows immediately from the definitions and Corollary \ref{c:z-gen}.

\end{proof}

Above, we used the following simple categorical lemma.

\begin{lem}\label{l:colim-ff}

Suppose we are given a filtered diagram $i \mapsto \sC_i$ of compactly 
generated DG categories, and a functor:
\[
F:\sC \coloneqq \underset{i}{\colim} \, \sC_i \to \sD \in \DGCat_{cont}
\]

\noindent preserving compact objects. Suppose each structural functor:
\[
\sC_i \to \sC_j
\]

\noindent and each composition:
\[
\sC_i \to \sC \to \sD
\]

\noindent is fully faithful. Then $F$ is fully faithful.

\end{lem}

\bibliography{bibtex.bib}{}
\bibliographystyle{alphanum}

\end{document}